\documentclass[11pt,A4paper]{article}
\RequirePackage[a4paper,margin=2.5cm]{geometry}

\RequirePackage[OT1]{fontenc}
\RequirePackage{amsthm,amsmath}
\RequirePackage[numbers]{natbib}
\RequirePackage[colorlinks,citecolor=blue,urlcolor=blue]{hyperref}
\usepackage{authblk}
\usepackage{amsmath}
\usepackage{amssymb}
\usepackage{graphicx}

\usepackage[mathscr]{eucal}
\usepackage{enumerate}

\usepackage{bbold}

\usepackage{cases}

\usepackage{diagbox}
\usepackage{subcaption}

\usepackage{mathtools}

\theoremstyle{plain}
\newtheorem{thm}{Theorem}[section]
\newtheorem{coro}{Corollary}[section]
\newtheorem{lem}{Lemma}[section]
\newtheorem{prop}{Proposition}[section]
\newtheorem{rem}{Remark}[section]

\numberwithin{equation}{section}

\newcommand{\abs}[1]{\left| #1\right|}

\newcommand{\E}{\mathbb{E}}
\newcommand{\PP}{\mathbb{P}}
\newcommand{\R}{\mathbb{R}}
\newcommand{\N}{\mathbb{N}}
\newcommand{\F}{\mathcal{F}}

\begin{document}

\title{\bf Efficient estimation of jump parameters for stochastic differential equations driven by L\'evy processes}

\author[1]{\textsc{Elise Bayraktar}}
\author[1]{\textsc{Emmanuelle Cl\'ement}}
\affil[1]{LAMA, Univ Gustave Eiffel, Univ Paris Est Creteil, CNRS, F-77447 Marne-la-Vall\'ee, France.}

\date{September, 1 2025}
\maketitle

\noindent
{\bf Abstract.}
In a high-frequency context, we investigate the efficient estimation of scaling and jump activity parameters for a stochastic differential equation driven by a L\'evy process with both diffusion component and pure-jump component.
We first study efficiency for the prototype  L\'evy process. With an in-depth analysis  of the behavior of the density of the process in small time,  we prove  that the LAN Property holds for the joint estimation of the diffusion, scaling and jump activity  parameters. We next consider  a stochastic equation driven both by a Brownian Motion and a locally stable pure-jump L\'evy process. Using a quasi-likelihood estimation method, we exhibit an estimator that attains the optimal rate of convergence previously identified. The asymptotic properties of the estimator are derived from sharp approximation results for the stochastic equation and for the locally stable distribution.

\noindent
\textbf{MSC $2020$}.  60G51; 60G52; 60J75; 62F12.

\noindent
\textbf{Key words}: LAN property; L\'evy process; parametric inference; stable process; stochastic differential equation.

\section{Introduction}

L\'evy driven stochastic equations constitute a wide class of processes, characterized by the superposition of a drift term, a diffusion term and a pure-jump component, that are appropriate and intensively used for modeling  many  phenomena.
Statistical inference for such processes reveals various asymptotic behaviors depending on the estimated component, in particular optimal rates of convergence differ for the estimation of the drift, diffusion or jump parameters.
When sampled at high-frequency, the dominant component of the process is the diffusion part and it is proved  in Aït-Sahalia and Jacod  \cite{AitSahaliaJacodInfoSigma, AitSahalia2008}  that the optimal rate and asymptotic variance that can be expected in estimating diffusion parameters are not affected by the presence of jumps. Conversely, the estimation of jump parameters relies critically on the presence or not of a Brownian Motion in the dynamic of the process. 
In this paper, we focus on the parametric estimation of the jump component in presence of a Brownian Motion and we especially are interested in investigating efficient estimation of   jump activity  and scaling parameters.

A fundamental concept in parametric inference to characterize efficiency is the Local Asymptotic Normality (LAN) Property 
introduced by Le Cam \cite{LeCamLAN} and developed by Hájek \cite{HajekConvolution, HajekMinimax} and Le Cam \cite{LeCamMinimax} (we also refer to Le Cam and Yang \cite{CamYangLAN}, Höpfner \cite{Hopfner} and Ibragimov and Has$'$minskii \cite{Ibragimov}). Considering 
 $(\Omega_n, \F_n, (\PP_n^\theta ; \theta \in \Theta))_n$  a sequence of statistical experiments parametrized by $\Theta \subset \R^d$,  the LAN Property holds at $\theta$ with rate $u_n$ (decreasing to 0 as $n$ goes to infinity) if for any $ h \in \R^d$ the log-likelihood ratio admits the expansion as $n \to + \infty$
$$\ln \frac{d \PP_n^{\theta + u_n h}}{d \PP_n^\theta} = h^T \Delta_{n, \theta} - \frac{1}{2} h^T I_n(\theta) h + o_{\PP_n^\theta}(1),$$
 where
$(\Delta_{n, \theta},  I_n(\theta) )$ converges in law to $(I(\theta)^{1/2} \mathcal{N}, I(\theta))$ under $\PP_n^\theta$, with $\mathcal{N}$ a standard Gaussian variable and $I(\theta)>0$ a deterministic matrix. As an  important consequence of the LAN Property, we deduce  from the Hájek’s convolution Theorem that  
if the LAN Property holds at $\theta$ with rate $u_n$ and information $ I(\theta)$, then  for any regular estimator $\hat{\theta}_n$  such that
$u_n^{-1} (\hat{\theta}_n - \theta)$ converges in law to a random variable  $Z_\theta$, 
$Z_\theta$ admits the decomposition $Z_\theta = I(\theta)^{-1/2} \mathcal{N} + R$ with $\mathcal{N}$ a standard Gaussian vector and $R$ independent of $\mathcal{N}$. The estimator is  then asymptotically efficient if  $R=0$. 
The minimax Theorem also applies. For any regular estimator $\hat{\theta}_n$ of $\theta$ and $l:\R^d \to [0, + \infty)$ a symmetric loss function satisfying $l(0)=0$ and $l(|x|)\leq l(|y|)$ if $|x|\leq |y|$, then
$$ \liminf_{n \to + \infty} \E [l(u_n^{-1}(\hat{\theta}_n - \theta))] \geq \E [l(I(\theta)^{-1/2}\mathcal{N})],$$
with $\mathcal{N}$ a standard Gaussian variable.
In particular, we get a lower bound for the asymptotic variance of any estimator by taking $l(u)=|u|^2$ as the loss function.

The LAN Property for  Lévy processes, based on high-frequency observations, has been investigated in several papers, we refer to Masuda \cite{Masuda2015Chapter} for a general overview.
For the pure-jump L\'evy model
$$X_t = \mu t + \delta S_t^\alpha$$
where $(S_t^{\alpha})_{t \geq 0}$ is a symmetric $\alpha$-stable L\'evy process with index $\alpha \in (0, 2)$,
it was proved by Aït-Sahalia and Jacod \cite{AitSahalia2008}, by studying the Fisher information, that the optimal convergence rates for the marginal estimation of $\mu$, $\delta$ and $\alpha$ were respectively $1/n^{1/\alpha - 1/2}$, $1/ \sqrt{n}$ and $1/ \ln(n)\sqrt{n}$. But 
it was shown by  Masuda \cite{Masuda2009} that  the (marginally optimal) diagonal rate of convergence $u_n = \text{diag}(1/n^{1/\alpha - 1/2} , 1/\sqrt{n}, 1/\ln(n)\sqrt{n})$  leads to a singular Fisher information matrix for the joint estimation of $(\mu, \delta, \alpha)$, where the singularity appears in the estimation of the jump parameters $(\delta, \alpha)$, suggesting that the optimal rate is slower and non-diagonal.
This was confirmed by Brouste and Masuda \cite{BrousteMasuda2018}, who finally established  the LAN Property  with a non-singular information matrix for the joint estimation of $(\mu, \delta, \alpha)$  using the non-diagonal rate 
$$
u_n=\begin{pmatrix}\frac{1}{n^{1/\alpha - 1/2}} & 0 \\
    0 & \frac{1}{\sqrt{n}}  v_n \end{pmatrix}
    \quad \text{ where }
     v_n = \begin{pmatrix}1 & -\frac{\delta}{ \alpha^2}\ln(n) \\
    0 & 1 \end{pmatrix}.
$$
For a Lévy process with both a Brownian Motion $(B_t)_{t\geq0}$ and a symmetric $\alpha$-stable jump component $(S_t^{\alpha})_{t \geq 0}$,
$$    X_t = \sigma B_t + \delta S_t^\alpha,$$
 the LAN Property with non-singular information matrix has not been established, and only  partial results exist, suggesting a similar interplay between the scaling coefficient $\delta$ and the jump activity $\alpha$.  Aït-Sahalia and Jacod \cite{AitSahalia2012} proved that  the optimal marginal rates of convergence for the estimation of the scale parameter $\delta$ and the jump activity index $\alpha$ are respectively $\ln(n)^{\alpha/4}/n^{\alpha/4}$ and $\ln(n)^{\alpha/4}/(\ln(n)n^{\alpha/4})$. 
Mies \cite{Mies} additionally established  that using the corresponding diagonal rate $u_n = \text{diag}(\ln(n)^{\alpha/4}/n^{\alpha/4},\ln (n)^{\alpha/4}/(\ln(n)n^{\alpha/4} ))$
when estimating $(\delta, \alpha)$ jointly leads 
to the singular Fisher information matrix 
$$\frac{2  c_{\alpha}}{\alpha(2-\alpha)^{\alpha/2}} \Bigl(\frac{ \delta }{\sigma} \Bigr)^{\alpha}
\begin{pmatrix}\frac{\alpha^2}{\delta^2} & \frac{\alpha}{2\delta} \\
    \frac{\alpha}{2\delta} & \frac{1}{4} \end{pmatrix}.$$

In this paper,  we solve the problem of efficiency in estimating the jump parameters of a L\'evy process with both a Brownian and an $\alpha$-stable jump component, in a high-frequency setting. This is conducted with an in-depth study of the density  of $X_t$ (which is a convolution between a Gaussian distribution and a stable distribution)  for $t$ small. 
By performing precise asymptotic expansions for the density and its derivatives (that extend the work of Aït-Sahalia and Jacod \cite{AitSahalia2012}), we
prove that  the LAN Property holds  for the estimation of $(\sigma, \delta, \alpha)$  with the non-diagonal rate 
$$    
u_n = \begin{pmatrix}\frac{1}{\sqrt{n}} & 0 \\
    0 & \frac{\ln(n)^{\alpha/4}}{n^{\alpha/4}} v_n\end{pmatrix} \quad \text{ where }
     v_n = \begin{pmatrix}1 & -\frac{\delta}{2 \alpha}[\ln(n) - \ln(\ln n)] \\
    0 & 1 \end{pmatrix},
$$
and with a non-singular Fisher information matrix.  The non-diagonal correction in the rate of convergence that leads to a non-singular information matrix is not easy to identify, especially the $\ln(\ln(n))$ term,  that  appears to provide some subtle compensations in the expansion of the density derivatives.

We next investigate the efficient estimation of  the scaling and  jump activity parameters for a stochastic equation driven by a Brownian Motion and a pure-jump locally stable L\'evy process. Approximating successively the stochastic equation by an Euler scheme and the locally stable distribution by the stable one, we consider a quasi-likelihood function and prove that it is possible to exhibit a rate optimal (and probably also variance optimal)  estimator of the parameters, including also the diffusion parameter. This was not obtained before  for a stochastic equation.
 For It\^{o} semimartingales with a diffusion component, Aït-Sahalia and Jacod have proposed in
\cite{AitSahalia2009}  an estimator of the jump activity $\alpha$ based on jump counting  with rate of convergence   $n^{\alpha/10}$ for locally stable jumps ($n^{\alpha/8}$ in the stable case) far from the optimal rate. A similar estimation method was studied in Jing et al. \cite{Jing}. A more accurate estimator was proposed in Bull  \cite{Bull}, by  combining increments of the process with different time-scales, achieving the near optimal rate of convergence $n^{\alpha/4-\epsilon}$ for $\epsilon$ arbitrarily small. Using a method of moments, rate optimality  has been attained by  Mies  \cite{Mies} for a L\'evy process with characteristic triplet $(\mu, \sigma^2, \nu)$, assuming that the jump part is close to the superposition of $M$ stable processes. The moment estimation method has been extended to the stable Cox-Ingersoll-Ross model in Bayraktar and Cl\'ement \cite{BC}, but exhibits a loss of order $\ln(n)^{\epsilon}$ in the rate of convergence.  \\

The paper is organized as follows. The LAN Property  for a L\'evy process with both Brownian and stable components is established in Section \ref{S:LAN}. Efficient estimation for a stochastic equation is studied in Section \ref{S:effiSDE}. We also present in this section sharp bounds for the weak one-step error of the Euler discretization scheme and for the approximation of the locally stable distribution by the stable one. The technical Section \ref{S:Section2} is devoted to the study of the density of the L\'evy process given as the convolution product between the Gaussian distribution and the stable distribution. We develop in this section  a generic approach to expand the density and its derivatives which is not only crucial for the estimation purpose but also interesting in itself. Proofs are gathered in Section \ref{S:proof}. \\


Throughout the paper, we denote by $C$ all positive constants (or $C_p$ if it depends on an auxiliary parameter $p$), which might change from line to line.
We also use the following notation.
For any process $Y$, we denote the increments with step size $\Delta_n=1/n$ by $\Delta _i^n Y = Y_{(i+1)/n} - Y_{i/n}$. We write $\xRightarrow[]{}$ for the convergence in law, $\xRightarrow{\mathcal{L}_s}$ for the stable convergence in law and $\xrightarrow{\PP}$ for the convergence in probability.
The expectation under $\PP^{\theta}$ is denoted by $\E^{\theta}$, we simply write $\E$ if there is no ambiguity. We will use the notation $\E_i Z$ for the conditional expectation 
$\E(Z | \mathcal{F}_{i/n})$.
For any matrix $M$, $M^T$ denotes the transpose of $M$ and $|| M ||$ denotes the Euclidean norm. For any differentiable function $f$ depending on $(y, \theta)$ with $y \in \R$ and  $\theta=(\sigma, \delta, \alpha)$, we denote by $f^{\prime}$ the derivative  of $f$ with respect to $y$, for $a \in \{\sigma, \delta, \alpha\}$ $\partial_a f$ is the derivative with respect to $a$ and  $\nabla_\theta f = (\partial_\sigma f, \partial_\delta f, \partial_\alpha f)^T$ is the gradient of $f$. 


\section{LAN Property for a stable L\'evy process with Brownian component} \label{S:LAN}

In this section, we consider the L\'evy process $(Y_t)_{t \geq 0}$ defined  on a filtered probability space  $(\Omega, \mathcal{F}, (\mathcal{F}_t)_{t \geq 0}, \PP)$ by
\begin{equation}\label{E:Levy}
    Y_t = \sigma B_t + \delta S_t^\alpha,
\end{equation}
where $(B_t)_{t\geq0}$ is a standard Brownian Motion  and $(S_t^{\alpha})_{t \geq 0}$ is an independent symmetric stable L\'evy process with  characteristic function
\begin{equation} \label{E:stable}
    \E \left(e^{iz S_1^{\alpha}} \right) =  \exp \left( - |z|^\alpha \right), \quad \alpha \in (0,2).
\end{equation}
The  L\'evy measure of $(S^\alpha_t)_{t\geq0}$ is given by
\begin{equation}\label{eq:mesureLevy}
    F^{\alpha}(dz) = \frac{c_\alpha}{|z|^{\alpha+1}} {\bf 1}_{z \neq 0} \, dz,
\end{equation}
where from Lemma 14.11 in Sato \cite{Sato},  $c_1=1/ \pi$ and for $\alpha\neq 1$
\begin{equation} \label{eq:calpha}
c_{\alpha}= -\frac{\alpha(\alpha-1)}{2\Gamma(2-\alpha) \cos(\pi \alpha/2)}, \mbox{ with } \Gamma(y)= \int_0^{\infty} x^{y-1} e^{-x} dx, \quad y>0.
\end{equation}
Our aim is to prove that the Local Asymptotic Normality Property holds for the estimation of 
$\theta = (\sigma, \delta, \alpha) \in (0, \infty) \times (0, \infty) \times (0, 2)$, based on high-frequency observations $(Y_{i/n})_{0 \leq i \leq n}$.
By assumption, the increments $(\Delta _i^n Y)_{0\leq i \leq n-1}$ are independent and identically distributed and denoting by  $p_{1/n}(\cdot, \theta)$ the density of  $Y_{1/n}$ 
the log-likelihood function has the expression
\begin{equation*}\label{def Ln}
     \ln L_n (\theta) = \sum_{i=0}^{n-1} \ln p_{1/n}(\Delta _i^n Y, \theta).
\end{equation*} 
To prove the LAN Property, we will study the asymptotic behavior of  the score function and Hessian matrix, defined respectively by
\begin{equation} \label{def:Gn} 
 G_n(\theta)= \nabla_\theta \ln L_n (\theta), \quad J_n(\theta) = \nabla_\theta G_n(\theta).
\end{equation}
It is not immediate to identify the appropriate rate of convergence of these quantities as $p_{1/n}$ is the  convolution between the Gaussian distribution and the stable distribution. 
Let $\phi$ be the density of $B_1$ (the standard Gaussian variable) and $\varphi_{\alpha}$ be the density of $S^{\alpha}_1$, then
observing that $(B_{1/n}, S^{\alpha}_{1/n})$ has the distribution of $(n^{-1/2}B_1, n^{-1/\alpha}S_1^\alpha)$, the density of $Y_{1/n}$ is given by
 \begin{equation}\label{def pi}
p_{1/n}(y, \theta)
=\frac{\sqrt{n}}{\sigma} \int \phi \left(\frac{y- n^{-1/\alpha} \delta  z}{n^{-1/2} \sigma} \right) \varphi_{\alpha}(z) dz.
\end{equation}
Approximation results for this convolution product and its derivatives with respect to the parameters are given in the technical section \ref{S:Section2}, showing how  to identify the optimal rate of convergence $u_n(\theta)$ that we define  below.

Our first result is a Central Limit Theorem for the score $ G_n(\theta_0)$ with non-diagonal rate $u_n(\theta_0)^T$ given by
\begin{equation}\label{eq:un}
    u_n(\theta) = \begin{pmatrix}\frac{1}{\sqrt{n}} & 0 \\
    0 & \frac{\ln(n)^{\alpha/4}}{n^{\alpha/4}} v_n(\theta) \end{pmatrix} \quad \text{ with } v_n(\theta) = \begin{pmatrix}1 & - \frac{\delta}{2 \alpha}[\ln(n) - \ln(\ln n)] \\
    0 & 1 \end{pmatrix}.
\end{equation}

\begin{prop}\label{Th:TCL}We have the convergence in law under $\PP^{\theta_0}$
\begin{equation*}
    u_n(\theta_0)^T \;  G_n(\theta_0) \xRightarrow[]{} \mathcal{N}(0,I(\theta_0)),
\end{equation*}
where $I(\theta)$ is the symmetric positive definite matrix defined by 
\begin{equation}\label{eq:I}
    I(\theta) = \begin{pmatrix} \frac{2}{\sigma_0^2} & 0 & 0 \\
    0 & \frac{\alpha^2}{\delta^2} \kappa_0(\theta)  & \frac{\alpha}{\delta}\kappa_0(\theta)  \kappa_1(\theta)\\ 
    0 & \frac{\alpha}{\delta}\kappa_0(\theta)  \kappa_1(\theta) & \kappa_0(\theta)(\kappa_1(\theta)^2+\frac{1}{\alpha^2}) \end{pmatrix},
\end{equation}
with $\kappa_0(\theta)=\frac{2  c_{\alpha}}{\alpha(2-\alpha)^{\alpha/2}} (\frac{ \delta }{\sigma})^{\alpha}$ and $\kappa_1(\theta)=\ln(\frac{ \delta }{\sigma}) + \frac{ \partial_\alpha c_{\alpha}}{c_{\alpha}} - \frac{\ln(2-\alpha)}{2} - \frac{1}{\alpha}$, where $c_{\alpha}$ is given by \eqref{eq:calpha}.

\end{prop}
Next we study the convergence of the Hessian matrix  $J_n(\theta)$ and prove a uniform law of large numbers, where the uniformity is obtained for $\theta \in V_n^{(r)}$ the neighborhood of $\theta_0$ 
defined for $r>0$ by
\begin{equation}\label{def:V}
    V_n^{(r)} = \{ \theta=(\sigma,\delta, \alpha) ; || u_n(\theta_0)^{-1}(\theta - \theta_0) ||\leq r \}.
\end{equation}

\begin{prop} \label{Th:CvIn}
We have
    \begin{equation} \label{E:CvIn}
    u_n(\theta_0)^T J_n(\theta_0) u_n(\theta_0) \xrightarrow[]{\PP^{\theta_0}} -I(\theta_0),
\end{equation}
and $\forall r>0$, we have the convergence in probability
    \begin{equation} \label{E:CvUnif} 
    \sup_{\theta \in V_n^{(r)} } ||u_n(\theta_0)^T (J_n(\theta) - J_n(\theta_0)) u_n(\theta_0)|| \xrightarrow[]{\PP^{\theta_0}} 0. 
\end{equation}
\end{prop}


From these propositions, we deduce easily the LAN Property with rate $u_n$.
\begin{thm}\label{Th:LAN}
The LAN Property holds at any parameter value $\theta_0 \in (0, \infty) \times (0, \infty) \times (0, 2)$
\begin{equation*}\label{eq:LAN}
    \sup_{h\in K} \left|\ln \left( \frac{L_n (\theta_0 + u_n(\theta_0) h)}{ L_n (\theta_0)} \right) - h^T \Delta_n(\theta_0) + \frac{1}{2}h^T I(\theta_0) h \right| \xrightarrow[]{\PP^{\theta_0}} 0,
\end{equation*}
for any compact set $K \subset \R^3$, where $\Delta_n(\theta_0) \xRightarrow[]{} \mathcal{N}(0, I(\theta_0))$ under $\PP^{\theta_0}$, and $I(\theta_0)$ is the positive definite matrix defined by \eqref{eq:I}.
\end{thm}
\begin{proof}[ Proof of Theorem \ref{Th:LAN}]
   From Taylor's formula,  we have
$$\ln L_n (\theta_0 + u_n(\theta_0) h) = \ln L_n (\theta_0) + h^T u_n(\theta_0)^T  G_n(\theta_0) + \frac{1}{2} h^T u_n(\theta_0)^T J_n(\theta_0) u_n(\theta_0) h + R_n(h), $$
with $G_n$ and $J_n$ defined by \eqref{def:Gn} and
$$R_n(h) = h^T u_n(\theta_0)^T \int_0^1 (1-t) \big( J_n(\theta_0+t u_n(\theta_0) h) - J_n(\theta_0)\big) dt \; u_n(\theta_0) h.$$
We deduce then the result of Theorem \ref{Th:LAN} from
 Proposition \ref{Th:TCL} and  Proposition \ref{Th:CvIn}.
\end{proof}

\begin{rem}
If we assume that $\delta$ is known, then we can prove the LAN Property for the estimation of $(\sigma, \alpha)$ with rate 
$u_n = \text{diag}(1/ \sqrt{n},\ln (n)^{\alpha/4}/(\ln(n)n^{\alpha/4} ))$  and information
$\text{diag}(2/ \sigma^2,   \kappa_0(\theta) /4)$.

In the same way, if $\alpha$ is known, the LAN Property holds for the estimation of $(\sigma, \delta)$ with rate
$u_n = \text{diag}(1/ \sqrt{n},\ln (n)^{\alpha/4}/n^{\alpha/4} )$  and information
$\text{diag}(2/ \sigma^2,   \kappa_0(\theta) \alpha^2/\delta^2)$.
\end{rem}

\begin{rem}
The LAN Property established in Theorem \ref{Th:LAN} addresses the conjecture raised by Mies in \cite{Mies} for the joint estimation of $(\delta, \alpha)$ and shows that the moment estimator proposed in \cite{Mies} (with the choice of the tuning parameter $u_n$ of order $\sqrt{n}/ \ln(n)$) is rate optimal.
\end{rem}
A direct consequence of Theorem \ref{Th:LAN} is the behavior of the maximum likelihood estimator, which follows from Theorem 2 in Sweeting \cite{Sweeting}.

\begin{coro}\label{Th:Estimator}
    There exists a local maximum $\hat{\theta}_n = (\hat{\sigma}_n, \hat{\delta}_n, \hat{\alpha}_n)$ of the log-likelihood function, with probability tending to 1, which converges in probability to $(\sigma_0, \delta_0, \alpha_0)$. Moreover, we have the convergence in law under $\PP^{\theta_0}$
    \begin{equation*}
    u_n(\theta_0)^{-1} (\hat{\theta}_n - \theta_0) \xRightarrow[]{} \mathcal{N}(0,I(\theta_0)^{-1}),
    \end{equation*}
    with $I(\theta_0)$ defined in \eqref{eq:I}.
\end{coro}

\section{Efficient estimation of a stochastic equation driven by a Brownian Motion and a locally stable L\'evy process} \label{S:effiSDE}

We now consider the class of stochastic one-dimensional equations  
\begin{equation} \label{E:EDS}
X_t=x_0+ \int_0^t b(X_s) ds + \int_0^t a(X_{s}, \sigma) d B_s + \delta \int_0^t c(X_{s-})d L_s^{\alpha}
\end{equation}
where $(B_t)_{t \geq 0} $ is a standard Brownian Motion, $(L_t^{\alpha})_{t \geq 0}$ is a pure jump locally $\alpha$-stable process, independent of $(B_t)_{t \geq 0}$. Both processes are  defined on a filtered space $(\Omega, \mathcal{F}, (\mathcal{F}_t)_{t \geq 0}, \mathbb{P})$, where $\mathcal{F}_t$ is the natural augmented filtration. We observe the discrete time process $(X_{t_i})_{0 \leq i \leq n}$ with $t_i=i /n$, for $i=0, \ldots ,n$  that solves \eqref{E:EDS} for the parameter value $\theta_0= (\sigma_0, \delta_0,\alpha_0)$ and
our aim is  to propose an efficient estimation method of the parameter $\theta_0$.

We make the following  assumptions on the coefficients $a$, $b$ and $c$ of the equation that ensure in particular that \eqref{E:EDS} admits a unique strong solution. We also specify the behavior of the L\'evy measure near zero of the process $(L_t^{\alpha})_{t \in [0,1]}$. 

\quad

\noindent
{\bf H1(Regularity)} 

 
 (a) Let $V_{\sigma_0} $ be a neighborhood of $\sigma_0$, we assume that $a$ is $\mathcal{C}^3$  on  $\mathbb{R} \times V_{\sigma_0} $, $b$ and $c$ are $\mathcal{C}^2$  on  $\mathbb{R}$,  and $a^{\prime}( \cdot, \sigma_0),$ $b^{\prime}$, $c^{\prime}$ are bounded.
 
(b) 
 $
\exists \overline{a}>0 $ and  $ \overline{c}>0$   such that $ \forall x \in \R$, $\forall \sigma \in  V_{\sigma_0}$,  $a(x, \sigma) \geq \overline{a}$ and $c(x) \geq \overline{c}$.

\noindent
{\bf H2 (L\'evy measure)} 

 Let $F^{\alpha, \tau}$ be the L\'evy measure of $(L_t^{\alpha})$. We assume that 
$$
F^{\alpha, \tau}(dz)= \frac{c_{\alpha} \tau(z)}{\abs{z}^{\alpha+1}} 1_{\mathbb{R}\setminus \{0\}}(z) dz,
$$
 where $\alpha \in (0,2)$, $c_{\alpha}$ is given by  \eqref{eq:calpha} and $\tau$ is  a real non-negative bounded even function with 
 $\tau(z)=1 + O(|z|)$ on $\{0< \abs{z} \leq \eta\}$ for some $\eta>0$.

To estimate the parameter, the main idea is to consider the Euler approximation of \eqref{E:EDS} without drift and next to replace the locally stable distribution by the stable distribution. Consequently the first approximation is to replace the increments $\Delta_i^nX $ by $\Delta_i^n\overline{X} $ given by
\begin{equation} \label{E:euler}
\Delta_i^n\overline{X} = a(X_{i/n}, \sigma_0)\Delta_i^nB + \delta_0c(X_{i/n},) \Delta_i^n L^{\alpha_0},
\end{equation}
and for the second approximation, we replace  the distribution of $\Delta_i^n L^{\alpha_0}$ by $n^{-1/ \alpha_0} S_1^{\alpha_0}$, so we finally approximate the distribution of  $\Delta_i^n X$ given $X_{i/n}=x$
by the distribution of $Y_{1/n}$ with
$$
Y_{1/n}=a(x, \sigma_0)n^{-1/2} B_1  + \delta_0 c(x) n^{-1/ \alpha_0} S_1^{\alpha_0},
$$
where $B_1$ is a standard Gaussian variable independent of $S_1^{\alpha_0}$, a stable variable with characteristic function \eqref{E:stable}.
Using the notation of Section \ref{S:LAN}, we will approach 
the transition density of $X_{(i+1)/n}$ given $X_{i/n}=x$ by the function
\begin{equation} \label{E:tildep}
p_{1/n}(y-x, \beta(x, \theta_0)), \; \text{with} \; \beta(x, \theta)=(a(x,\sigma),  \delta c(x), \alpha)
\end{equation}
where  $p_{1/n}$ is given by \eqref{def pi}. Considering the quasi-likelihood function
$$
\ln \tilde{L}_n(\theta)= \sum_{i=0}^{n-1} \ln p_{1/n}(\Delta_i^n X, \beta( X_{i/n}, \theta)),
$$
we estimate $\theta$ by solving the estimating equation
$\tilde{G}_n(\theta)=0$ with
\begin{equation} \label{E:tildeG}
\tilde{G}_n(\theta)= \nabla_{\theta}\ln  \tilde L_n(\theta).
\end{equation}
We prove  that the estimator has the optimal rate of convergence $u_n(\theta_0)$  defined by \eqref{eq:un}.
\begin{thm} \label{Th:estEDS}
We assume {\bf H1} and  {\bf H2}. Then there exists $\hat{\theta}_n$ that solves the estimating equation $\tilde{G}_n(\theta)=0$ with probability tending to 1. This estimator converges in probability to $\theta_0$, moreover we have the stable convergence in law with respect to $\sigma (L_s^{\alpha_0}, s\leq 1)$
$$
u_n^{-1}(\theta_0) ( \hat{\theta}_n- \theta_0) \xRightarrow{\mathcal{L}_s}  \overline{I}(\theta_0)^{-1/2} \mathcal{N},
$$
where $\mathcal{N}$ is a standard Gaussian variable independent of $\overline{I}(\theta_0)$ and
$$
\overline{I}(\theta)= \left(
\begin{array}{ccc}
2\int_0^1 \left(\frac{ \partial_{\sigma} a(X_s, \sigma)}{a(X_s, \sigma)}\right)^2ds & 0 & 0 \\
 0 & \frac{\alpha^2}{\delta^2} \int_0^1\overline{\kappa}_0(\theta, X_s )ds & \frac{\alpha}{\delta} \int_0^1\overline{\kappa}_0(\theta, X_s) \overline{\kappa}_1(\theta, X_s) ds\\
 0 &\frac{\alpha}{\delta}  \int_0^1 \overline{\kappa}_0(\theta, X_s) \overline{\kappa}_1(\theta, X_s) ds & \int_0^1 \overline{\kappa}_0(\theta, X_s)(  \overline{\kappa}_1(\theta, X_s)^2+\frac{1}{\alpha^2})ds
\end{array}
\right),
$$
with for $x \in \R$
\begin{align*}
\overline{\kappa}_0(\theta, x)=\frac{2 c_{\alpha}}{\alpha(2-\alpha)^{\alpha/2}}\left(\frac{\delta c(x)}{a(x, \sigma)} \right)^{\alpha},\\
\overline{\kappa}_1(\theta, x)= \ln \left(\frac{\delta c(x)}{a(x, \sigma)} \right) + \frac{ \partial_{\alpha} c_{\alpha}}{c_{\alpha}}-\frac{\ln(2- \alpha)}{2} -\frac{1}{\alpha}.
\end{align*}

\end{thm}

\begin{rem}
Until now, no estimation method  for the jump parameters of a stochastic equation has achieved the optimal rate of convergence $v_n(\theta)\ln(n)^{\alpha/4}/n^{\alpha/4}$ (with $v_n$ defined by \eqref{eq:un}). A near optimal rate is obtained by Bull \cite{Bull}  for the estimation of $\alpha$ up to a factor $n^{\epsilon}$ with $\epsilon$ theoretically arbitrarily small. The moment method proposed in \cite{BC} for the stable Cox-Ingersoll-Ross model is a little bit faster with a loss in rate  of order $\ln(n)^{\epsilon} $. 
The result of Theorem \ref{Th:estEDS} shows that it is possible to construct an estimator of the jumps parameters $(\delta, \alpha)$ that attains the optimal rate of convergence $v_n(\theta)\ln(n)^{\alpha/4}/n^{\alpha/4}$. We also conjecture that this estimator is efficient.
\end{rem}

\begin{rem}
The asymptotic properties of the estimator of $(\delta, \alpha)$  depend crucially on the jump coefficient  of the stochastic equation. In \eqref{E:EDS}, the jump coefficient has the form $\delta c(x)$, so the parameter $\delta$ corresponds to a scaling parameter for the L\'evy process $(L_t^{\alpha})$ and it is not surprising in that case, that is the natural extension of \eqref{E:Levy}, to obtain the rate of convergence $u_n(\theta)$. 

Considering a more general jump coefficient $c(x, \delta)$ in the stochastic equation and assuming  that $s\mapsto (\partial_{\delta}c/c)(X_s, \delta_0)$ is almost surely non constant on $[0,1]$ (and also that $c$ is $\mathcal{C}^3$  on  $\mathbb{R} \times V_{\delta_0} $ for  $V_{\delta_0}$ a neighborhood of $\delta_0$ and that $ \forall x \in \R$, $\forall \delta \in  V_{\delta_0}$,  $c(x, \sigma) \geq \overline{c}$), then  the estimator $\hat{\theta}_n$ has a faster diagonal rate of convergence that  achieves the optimal rate in estimating marginally $\delta$ and $\alpha$. More precisely, we can prove the stable convergence
$$
\text{diag}( \sqrt{n}, n^{\alpha_0/4}/ \ln(n)^{\alpha_0/4}, n^{\alpha_0/4} \ln(n)/ \ln(n)^{\alpha_0/4}) ( \hat{\theta}_n- \theta_0) \xRightarrow{\mathcal{L}_s}  \tilde{I}(\theta_0)^{-1/2} \mathcal{N},
$$
where $\tilde{I}(\theta)$ is the non-singular matrix given by
$$
\tilde{I}(\theta)= \left(
\begin{array}{ccc}
2\int_0^1 \left(\frac{ \partial_{\sigma} a(X_s, \sigma)}{a(X_s, \sigma)}\right)^2ds & 0 & 0 \\
 0 & \alpha^2 \int_0^1 \frac{ \partial_{\delta}c}{c}(X_s, \delta)^2\tilde{\kappa}_0(\theta, X_s )ds & \frac{\alpha }{2}\int_0^1 \frac{ \partial_{\delta}c}{c}(X_s, \delta)\tilde{\kappa}_0(\theta, X_s) ds\\
 0 &\frac{\alpha}{2}  \int_0^1\frac{ \partial_{\delta}c}{c}(X_s, \delta) \tilde{\kappa}_0(\theta, X_s) ds & \frac{1}{4}\int_0^1 \tilde{\kappa}_0(\theta, X_s)ds
\end{array}
\right),
$$
and $\tilde{\kappa}_0(\theta, x)=\frac{2 c_{\alpha}}{\alpha(2-\alpha)^{\alpha/2}}\left(\frac{c(x, \delta)}{a(x, \sigma)} \right)^{\alpha}$. The proof of this result is omitted as it follows the same scheme that the proof of Theorem \ref{Th:estEDS},  with some simplifications (essentially for the last coordinate of the normalized score) that come from the faster rate of convergence. 
 \end{rem}

The proof of Theorem \ref{Th:estEDS} is based on the two following key approximation results that are interesting in themselves. The first proposition 
gives  a bound for the total variation distance between the rescaled locally stable distribution and the stable distribution and extend it to slow growing functions. This result is similar  to Theorem 4.2 in \cite{CG18} but  we obtain here a better bound with weaker assumptions (and without the use of Malliavin calculus).  We denote by   $d_{TV}(X,Y)$ the total variation distance between the distributions of the random variables $X$ and $Y$.

\begin{prop} \label{P:TV}
Let $(L_t^{\alpha})_{t \geq 0}$ be a locally stable process satisfying {\bf H2} and $(S_t^{\alpha})_{t \geq 0}$ be a stable process with characteristic function \eqref{E:stable}. 

(i) We have
$$
d_{TV}( n^{1/\alpha} L_{1/n}^{\alpha}, S_1^{\alpha}) \leq 
\left\{
\begin{array}{l}
C/n \; \text{if} \; \alpha<1,  \\
C\ln(n)/n\; \text{if} \; \alpha=1, \\
C/ n^{1/ \alpha} \; \text{if} \; \alpha >1. 
\end{array}
\right.
$$

(ii) Let $(g_n)_{n \geq 1}$ be a sequence of real functions satisfying
$$
\exists  \varepsilon \in (0,1/4), \; \exists  C>0 \; \text{such that} \;  \forall z \in \R, \;  \forall n, \; |g_n(z)| \leq C \ln(n)^p(1+ |z|^{\varepsilon \alpha}),
$$
then for any bounded sequence $(c_n)$
$$
|\E g_n(  B_1+ c_n n^{1/\alpha} L_{1/n}^{\alpha})-\E g_n(  B_1+ c_n S_1^{\alpha})| \leq
C (1/n^{1-2\varepsilon} \lor 1/n^{1/ \alpha - 2\varepsilon}),
$$
where $B_1$ is a standard Gaussian variable independent of $L_{1/n}^{\alpha}$ and $S_1^{\alpha}$.
\end{prop}
\begin{rem}
We deduce from this result that if $ |g_n(z)| \leq C \ln(n)^p(1+ \ln(1+|z|)^q)$ for some $p,q>0$ then 
$$
|\E g_n(  B_1+ c_n n^{1/\alpha} L_{1/n}^{\alpha})-\E g_n(  B_1+ c_n S_1^{\alpha})| \leq
C (1/n^{1-\varepsilon} \lor 1/n^{1/ \alpha - \varepsilon}),
$$
with $\varepsilon$ arbitrarily small. We prove in Section \ref{S:Section2} (Lemma  \ref{L:der2der3}) that this condition is  satisfied by the derivatives with 
respect to $\theta$ of $p_{1/n}$.  
\end{rem}
Our second result controls the weak and strong errors in approximating $\Delta_i^n X$ (where $(X_t)_{t\geq 0}$ solves \eqref{E:EDS}) by $\Delta_i^n \overline{X}$ given by \eqref{E:euler}. It is obtained under stronger assumptions on the coefficients of the stochastic equation and the L\'evy measure $F^{\alpha, \tau}$, that will be satisified after a localization procedure.
\begin{prop} \label{P:euler}
We assume {\bf H1} and {\bf H2}. We also assume that the functions $a$ and $c$ are bounded with bounded derivatives and that the support of the L\'evy measure of $L^{\alpha_0}$ is bounded. Let $(f_n)_{n \geq 1}$ be a sequence of  even $\mathcal{C}^2$ functions with polynomial growth and such that
$$
\exists p>0, \; \forall x \in \R, \;  \forall n, \;  |f^{\prime}_n(x)| + |f_n^{\prime \prime}(x)|  \leq C \ln(n)^p.
$$
Then $\forall \epsilon>0$, $\exists C_{\epsilon} $ such that for $0 \leq i \leq n-1$ 
\begin{eqnarray} \label{E:weak}
| \E^{\theta_0}_i f_n(\sqrt{n} \Delta_i^n X)- \E^{\theta_0}_i f_n(\sqrt{n} \Delta_i^n \overline{X}) | \leq 
\left\{
\begin{array}{l}
C_{\epsilon}/n^{1/ \alpha_0 - \epsilon}, \; \text{if} \; \alpha_0 >1, \\
C_{\epsilon}/n^{1 - \epsilon}, \; \text{if} \; \alpha_0 \leq 1,
\end{array}
\right.
\end{eqnarray}
and $\exists C $ such that for $0 \leq i \leq n-1$ 
\begin{eqnarray} \label{E:square}
 \E^{\theta_0}_i (|f_n(\sqrt{n} \Delta_i^n X)- f_n(\sqrt{n} \Delta_i^n \overline{X}) |^2) \leq C \ln(n)^{2p}/n.
\end{eqnarray}
\end{prop}

\section{Asymptotic expansion for the density and its derivatives} \label{S:Section2}

This section is devoted to the study of the density $p_{1/n}$ (given by \eqref{def pi} as a convolution  between the Gaussian distribution and the stable distribution) and its derivatives with respect to the parameter $\theta$.
 We recall that $\phi$ is the density of $B_1$ (the standard Gaussian variable) and $\varphi_{\alpha}$ is the density of $S^{\alpha}_1$. From Section 14 in Sato \cite{Sato}, the map $(z, \alpha) \to \varphi_{\alpha}(z)$ defined on $\R \times (0,2)$ admits derivatives of any order with respect to both $z$ and $\alpha$.

\subsection{Auxiliary functions}
We first introduce some generic functions to simplify the expressions  of the score $G_n$ and the Hessian matrix $J_n$ defined by \eqref{def:Gn} and to make clearer their  asymptotic behavior.  

We start by calculating $\nabla_{\theta} p_{1/n}$, with $p_{1/n}$ defined by \eqref{def pi}. We have
\begin{align*}
    \partial_\sigma p_{1/n}(y, \theta)
    =& - \frac{\sqrt{n}}{\sigma^2} \int \phi\left(\frac{y - n^{-1/\alpha}\delta  z}{n^{-1/2} \sigma} \right) \varphi_{\alpha}(z) dz \\
    & \quad- \frac{\sqrt{n}}{\sigma}\int \left(\frac{y - n^{-1/\alpha}\delta  z}{n^{-1/2} \sigma^2} \right) \phi^{\prime} \left(\frac{y - n^{-1/\alpha}\delta  z}{n^{-1/2} \sigma} \right) \varphi_{\alpha}(z) dz \\
    =& - \frac{\sqrt{n}}{\sigma^2} \int (x\phi)^{\prime} \left(\frac{y - n^{-1/\alpha}\delta  z}{n^{-1/2} \sigma} \right) \varphi_{\alpha}(z) dz.
\end{align*}
We get after integrating by parts 
\begin{align*}
    \partial_\delta p_{1/n}(y, \theta)&
    =-\frac{\sqrt{n}}{\sigma} \int \frac{ n^{-1/\alpha}z}{n^{-1/2} \sigma} \phi'\left(\frac{y - n^{-1/\alpha}\delta  z}{n^{-1/2} \sigma} \right) \varphi_{\alpha}(z) dz\\
    &=- \frac{\sqrt{n}}{\sigma \delta} \int \phi\left(\frac{y - n^{-1/\alpha}\delta  z}{n^{-1/2} \sigma} \right) (z\varphi_{\alpha})'(z) dz,
\end{align*}
and we similarly obtain 
$$\partial_\alpha p_{1/n}(y, \theta)
    =\frac{\sqrt{n}}{\sigma} \int \phi \left(\frac{y - n^{-1/\alpha}\delta  z}{n^{-1/2} \sigma} \right) \left(\partial_\alpha \varphi_{\alpha}(z) - \frac{\ln(n)}{\alpha^2} (z\varphi_{\alpha})'(z) \right) dz.$$
In the same way, we can compute the second and third  order derivatives, and to simplify the notation it is convenient to introduce the operator
\begin{align}\label{def:D}
\mathcal{D} : \left\{ 
  \begin{array}{ l l }
    \mathcal{C}^{\infty} & \rightarrow \mathcal{C}^{\infty} \\
    f                 & \mapsto \mathcal{D}(f): \left\{ 
  \begin{array}{ l l }
    \R & \rightarrow \R \\
    y                 & \mapsto (yf)'(y)= f(y)+yf'(y).
  \end{array}
\right. 
  \end{array}
\right.
\end{align}
For $k \geq 2$, we denote by  $\mathcal{D}^{(k)}$ the operator iterated $k$-times defined by  $\mathcal{D}^{(k)}=\mathcal{D}(\mathcal{D}^{(k-1)})$, which  will appear in the derivatives of order $k$  of $p_{1/n}$. For $k=0$, $\mathcal{D}^{(0)}$ is the identity.
With this notation, we can simplify the expression of $\nabla_{\theta} p_{1/n}$.  By setting
\begin{align}\label{def:fghk} 
&f_{\alpha}(y, w) = \int \phi(y- w z) \varphi_{\alpha}(z) dz, &g_{\alpha}(y, w) = \int \mathcal{D}(\phi)(y- w z) \varphi_{\alpha}(z) dz,\\
    &h_{\alpha}(y, w) = \int \phi(y- w z) \mathcal{D}(\varphi_{\alpha})(z) dz, & k_{\alpha}(y, w) = \int \phi(y- w z) \partial_{\alpha}\varphi_{\alpha}(z) dz, \nonumber
\end{align}
and
\begin{equation} \label{E:wn}
 w_n(\theta) = n^{1/2 - 1/\alpha} \frac{\delta}{\sigma},
\end{equation}
we have
\begin{equation}\label{eq:lienpf}
    p_{1/n}(y, \theta) = \frac{\sqrt{n}}{\sigma}f_{\alpha}\Bigl(\frac{\sqrt{n}}{\sigma}y, w_n(\theta)\Bigr), 
 \end{equation}
and the score function can be written as
\begin{eqnarray} \label{eq:GnCalcul}
G_n(\theta)=  \sum_{i=0}^{n-1} \nabla_{\theta} \ln  p_{1/n}(\Delta _i^n Y, \theta) =-\sum_{i=0}^{n-1}\left(
\begin{array}{c}
    \frac{1}{\sigma} \frac{g_{\alpha}}{f_{\alpha}}\Bigl(\frac{\sqrt{n}}{\sigma} \Delta _i^n Y, w_n(\theta)\Bigr)\\
     \frac{1}{\delta} \frac{h_{\alpha}}{f_{\alpha}} \Bigl(\frac{\sqrt{n}}{\sigma} \Delta _i^n Y, w_n(\theta)\Bigr)  \\
    \frac{\frac{\ln(n)}{\alpha^2} h_{\alpha} - k_{\alpha}}{f_{\alpha}}\Bigl(\frac{\sqrt{n}}{\sigma} \Delta _i^n Y, w_n(\theta)\Bigr)
\end{array}
\right).
\end{eqnarray}
We now give some insights on the choice of the rate $u_n(\theta)$ given by \eqref{eq:un}. We  prove in Lemma \ref{l:approxD} the following approximations (where $\psi^{(0)}$ and $\psi^{(1)}$ are defined in \eqref{eq:fp})
\begin{align*}
    h_{\alpha}(y, w_n(\theta)) &\simeq -\alpha c_\alpha w_n(\theta)^\alpha \psi_{\alpha}^{(0)}(y),\\
    \big[\frac{\ln(n)}{\alpha^2} h_{\alpha} - k_{\alpha} \big](y, w_n(\theta)) &\simeq - \frac{c_\alpha}{2} \ln(n) w_n(\theta)^\alpha \psi_{\alpha}^{(0)}(y) 
    + c_\alpha w_n(\theta)^\alpha \psi_{\alpha}^{(1)}(y)
    \\& \qquad \qquad - [\partial_\alpha c_\alpha + \ln(\frac{\delta}{\sigma})]w_n(\theta)^\alpha \psi_{\alpha}^{(0)}(y).
\end{align*}
This makes clear the difference $\ln(n)$ in rate when estimating $\delta$ and $\alpha$ marginally. 
Now, considering the non-diagonal estimation rate
$$
u_n(\theta) = \begin{pmatrix}\frac{1}{\sqrt{n}} & 0 \\ 0 & \frac{\ln(n)^{\alpha/4}}{n^{\alpha/4}} v_n\end{pmatrix} \quad \text{ with } v_n = \begin{pmatrix}1 & c_n \\ 0 & 1 \end{pmatrix},$$
we deduce
\begin{eqnarray*}
u_n(\theta)^T  G_n(\theta)= \left(
\begin{array}{c}
-\frac{1}{\sqrt{n}} \sum_{i=0}^{n-1} \frac{1}{\sigma} \frac{g_{\alpha}}{f_{\alpha}}\Bigl(\frac{\sqrt{n}}{\sigma} \Delta _i^n Y, w_n(\theta)\Bigr) \\
-\frac{\ln(n)^{\alpha/4}}{n^{\alpha/4}} \sum_{i=0}^{n-1} \frac{1}{\delta} \frac{h_{\alpha}}{f_{\alpha}}\Bigl(\frac{\sqrt{n}}{\sigma} \Delta _i^n Y, w_n(\theta)\Bigr)\\
- \frac{\ln(n)^{\alpha/4}}{n^{\alpha/4}}\sum_{i=0}^{n-1} \frac{(\frac{\ln(n)}{\alpha^2} + \frac{c_n}{\delta}) h_{\alpha} - k_{\alpha}}{f_{\alpha}}\Bigl(\frac{\sqrt{n}}{\sigma} \Delta _i^n Y, w_n(\theta)\Bigr)
\end{array}
\right).
\end{eqnarray*}
From Lemma \ref{l:approxD} we have the approximation 
\begin{align*}
    \biggl[\bigl(\frac{\ln(n)}{\alpha^2}+\frac{c_n}{\delta}\bigr) h_{\alpha} - k_{\alpha}\biggr](y, w_n(\theta)) &\simeq - c_\alpha \bigl(\frac{ \ln(n)}{2} +\frac{\alpha}{\delta}c_n\bigr) w_n(\theta)^{\alpha} \psi_{\alpha}^{(0)}(y) 
    \\
    + c_{\alpha} & w_n(\theta)^{\alpha} \psi_{\alpha}^{(1)}(y)
    - [\partial_\alpha c_{\alpha} + \ln(\frac{\delta}{\sigma})]w_n(\theta)^{\alpha} \psi_{\alpha}^{(0)}(y),
\end{align*}
so choosing $c_n = -\frac{\delta}{2\alpha}\ln(n)$ allows to cancel out the term in $\ln(n) w_n(\theta)^{\alpha} \psi_{\alpha}^{(0)}$ in the last coordinate of $u_n(\theta)^T  G_n(\theta)$, but leads to a singular information matrix. To obtain a non-singular information we chose in
 \eqref{eq:un} $c_n =-\frac{\delta}{2 \alpha}[\ln(n) - \ln(\ln n)]$, where 
the term in $\ln(\ln n)$ provides a subtle partial compensation of the term in $\psi_{\alpha}^{(1)}$ when taking the expectation (the details of this compensation can be found in the proof of Proposition \ref{L:tout}).

We finally can  write the normalized score as
\begin{eqnarray} \label{E:scoren}
u_n(\theta)^T  G_n(\theta)= \left(
\begin{array}{c}
-\frac{1}{\sqrt{n}} \sum_{i=0}^{n-1} \frac{1}{\sigma} \frac{g_{\alpha}}{f_{\alpha}}\Bigl(\frac{\sqrt{n}}{\sigma} \Delta _i^n Y, w_n(\theta)\Bigr) \\
-\frac{\ln(n)^{\alpha/4}}{n^{\alpha/4}} \sum_{i=0}^{n-1} \frac{1}{\delta} \frac{h_{\alpha}}{f_{\alpha}}\Bigl(\frac{\sqrt{n}}{\sigma} \Delta _i^n Y, w_n(\theta)\Bigr)\\
 \frac{\ln(n)^{\alpha/4}}{n^{\alpha/4}}\sum_{i=0}^{n-1} \frac{l_{\alpha}}{f_{\alpha}}\Bigl(\frac{\sqrt{n}}{\sigma} \Delta _i^n Y, w_n(\theta)\Bigr)
\end{array}
\right)
\end{eqnarray}
by setting
\begin{equation}\label{def:l}
    l_{\alpha}(y, w_n(\theta)) = \left[\frac{1}{2 \alpha}\bigl(\ln(n) - \ln(\ln n)\bigr) - \frac{\ln(n)}{\alpha^2} \right] h_{\alpha}(y, w_n(\theta)) + k_{\alpha}(y, w_n(\theta)).
\end{equation}
In order to prove the convergence of $u_n(\theta_0)^T   G_n(\theta_0)$, we will give estimates for the functions $f_\alpha, g_\alpha, h_\alpha$ and $l_\alpha$ in Section \ref{S:oneder}, by extending the results stated  in the Supplement to Aït-Sahalia and Jacod \cite{AitSahalia2012}.

Turning to the Hessian matrix  $J_n$ given by in \eqref{def:Gn}, we have
$$ J_n(\theta) = \sum_{i=0}^{n-1}  \begin{pmatrix} \partial^2_{\sigma \sigma} \ln p_{1/n} & \partial^2_{ \delta \sigma} \ln p_{1/n} & \partial^2_{\alpha\sigma } \ln p_{1/n} \\
    \partial^2_{\sigma\delta } \ln p_{1/n} & \partial^2_{\delta \delta} \ln p_{1/n}& \partial^2_{\alpha\delta } \ln p_{1/n}\\ 
    \partial^2_{\sigma\alpha } \ln p_{1/n} & \partial^2_{\delta\alpha } \ln p_{1/n} & \partial^2_{\alpha \alpha}\ln p_{1/n} \end{pmatrix}(\Delta _i^n Y, \theta),$$
where for $a,b \in \{\sigma, \delta, \alpha \}$ $$\partial^2_{ab} \ln p_{1/n}=\frac{\partial^2_{a b} p_{1/n}  p_{1/n}  -\partial_{a} p_{1/n} \partial_{b} p_{1/n} }{(p_{1/n})^2}. $$
Consequently, the study of the convergence of $J_n$ requires some bounds on the second and third order derivatives of $p_{1/n}$.
Similarly to the functions defined in \eqref{def:fghk}, which naturally appear in the first derivatives of the density and in order to have some generic expressions for the higher order derivatives,  we introduce a larger set of functions built as a convolution between the density of $B_1$ and the density of $S_1^\alpha$. For $k, l, m \in \N$, we set
\begin{equation}\label{def:fklm}
    f_{\alpha}^{(k, l, m)}(y, w) = \int \mathcal{D}^{(k)}(\phi)(y- w z) \mathcal{D}^{(l)}(\partial^m_{\alpha^m} \varphi_{\alpha})(z) dz.
\end{equation}
From Schwarz's Theorem $(\partial_{\alpha}\varphi_{\alpha})'=\partial_{\alpha}(\varphi_{\alpha}')$ hence $\mathcal{D}(\partial_{\alpha} \varphi_{\alpha}) = \partial_{\alpha}(\mathcal{D}( \varphi_{\alpha}))$, and after iterating we get $\forall l,m \in \N$
$$ \mathcal{D}^{(l)}(\partial^m_{\alpha^m} \varphi_{\alpha}) = \partial^m_{\alpha^m}\mathcal{D}^{(l)}( \varphi_{\alpha}).$$
With this notation we have obviously (recalling \eqref{def:fghk})
$$
f_{\alpha}=f_{\alpha}^{(0, 0,0)}, \; g_{\alpha}=f_{\alpha}^{(1, 0,0)}, \; h_{\alpha}=f_{\alpha}^{(0, 1,0)}, \; k_{\alpha}=f_{\alpha}^{(0, 0,1)}.
$$
Moreover an explicit computation gives
$$\partial^2_{\sigma \sigma} \ln (p_{1/n}(y, \theta)) = \biggl[ \frac{2}{\sigma^2} \frac{f_{\alpha}^{(2, 0,0)}}{f_{\alpha}} +
\frac{1}{\sigma^2} \frac{f_{\alpha}^{(1, 0,0)}}{f_{\alpha}} 
- \frac{1}{\sigma^2} \Bigl(\frac{f_{\alpha}^{(1, 0,0)}}{f_{\alpha}}\Bigr)^2 \biggr] \Bigl(\frac{\sqrt{n}}{\sigma}y, w_n(\theta)\Bigr).$$
Similarly, for $a,b \in \{\sigma, \delta, \alpha \}$ and $(a,b)\neq(\sigma, \sigma)$, $\partial^2_{a b} \ln (p_{1/n}(y, \theta))$ can be written as a sum and product of 
$$c_{k,l,m}(\theta) \ln(n)^p \,\frac{f_{\alpha}^{(k, l,m)}}{f_{\alpha}}\Bigl(\frac{\sqrt{n}}{\sigma}y, w_n(\theta)\Bigr),$$ for $k \leq 1$, $l+m>0$, $l,m \leq 2$ and some coefficients $c_{k,l,m}(\theta)$ that we do not need to explicit. This highlights the connection between $J_n$ and the auxiliary
 functions $f_{\alpha}^{(k,l,m)}$.  We study more deeply these functions  in Section \ref{S:suporder}.

We end this section with some additional notation related to the asymptotic behavior of the stable density $\varphi_{\alpha}$ and its derivative. From
Section 14 in Sato \cite{Sato} we have the asymptotic expansions
\begin{align*}\varphi_\alpha(z) &\underset{z\to\infty}{=} \frac{c_\alpha}{|z|^{\alpha+1}} + O\Bigl(\frac{1}{|z|^{2\alpha+1}}\Bigr),\\
        \mathcal{D}(\varphi_\alpha)(z) &\underset{z\to\infty}{=} -\frac{ \alpha c_\alpha}{|z|^{\alpha+1}}+ O\Bigl(\frac{1}{|z|^{2\alpha+1}}\Bigr), \\
        \partial_\alpha \varphi_\alpha(z) &\underset{z\to\infty}{=} \frac{\partial_\alpha c_\alpha - c_\alpha\ln|z|}{|z|^{\alpha+1}}+ O\Bigl(\frac{\ln|z|}{|z|^{2\alpha+1}}\Bigr),
    \end{align*}
and more generally for $p \in \N$
    $$\partial^p_{\alpha^p} \varphi_{\alpha}(z) \underset{z\to\infty}{\sim} \frac{C(\ln|z|)^p}{|z|^{\alpha+1}}.$$
This motivates the definition of the functions $\psi_{\alpha}^{(p)}$, for $p \in \N$
\begin{equation}\label{eq:fp}
    \psi_{\alpha}^{(p)}(y) = \left\{ \begin{tabular}{l l}
1&if $|y|\leq$ 1\\
$\frac{ (\ln|y|)^p}{|y|^{\alpha +1}}$&if $|y|\geq 3$,
\end{tabular} \right. 
\end{equation}
with $ \psi_{\alpha}^{(p)}$ $\mathcal{C}^2$ non-negative and even on $\R$.
We also define for $k,l \in \N$
\begin{equation}\label{eq:Ical}
    \mathcal{I}_\alpha^{ (k,l)}(y) = \int_{|z|>1} \mathcal{D}^{(k)}(\phi)(y-z) \psi_{\alpha}^{(l)}(z)dz,
\end{equation}
which is  an extension of the function $D^{\theta, \alpha, l}$ defined in (A.4) of  the Supplement to Aït-Sahalia and Jacod \cite{AitSahalia2012}.

\subsection{First order derivatives} \label{S:oneder}

The goal of this section is to give estimates for the score function. 
The first lemma gives some estimates and bounds for $\mathcal{I}_\alpha^{ (k,l)}$ defined in \eqref{eq:Ical}.

\begin{lem}\label{L:I} For $k,l \in \N$, we have the asymptotic expansions
    \begin{equation}\label{eq:Icaleq}
    \mathcal{I}_{\alpha}^{(k, l)}(y)  \underset{y\to\infty}{=} \psi_{\alpha}^{(l)}(y) \biggl(\int\mathcal{D}^{(k)}(\phi)(z)dz+O\Bigl(\frac{1}{\sqrt{|y|}}\Bigr) \biggr),
\end{equation}
and 
\begin{equation}\label{eq:D}
\exists C,\; \forall |y|>1, \quad 
    \bigl|\mathcal{I}_{\alpha}^{(0, l)}(y)  - \psi_{\alpha}^{(l)}(y)\bigr|\leq C \frac{\psi_{\alpha}^{(l)}(y)}{\sqrt{|y|}}.
\end{equation}
For $k,l \in \N$ and for any compact set $K \subset (0,2),$ $\exists C>0, \; \forall \alpha \in K, \; \forall y$
\begin{equation}\label{eq:majIcal}
    |\mathcal{I}_{\alpha}^{(k, l)}(y)|\leq C \psi_{\alpha}^{(l)}(y), \qquad \frac{1}{C} \psi_{\alpha}^{(0)}(y)\leq \mathcal{I}_{\alpha}^{(0, 0)}(y) \leq C \psi_{\alpha}^{(0)}(y).
\end{equation}
\end{lem}
The next lemma presents some first estimates for 
$f_{\alpha}, g_{\alpha}, h_{\alpha}$ and $k_{\alpha}$ defined in \eqref{def:fghk}. 
\begin{lem}\label{l:approxD} For any compact set $K \subset (0,2),$ $\exists C>0, \; \forall \alpha \in K, \; \forall w \in (0, 1/3], \, \forall y$
\begin{align}
    |f_{\alpha}(y, w) - & \phi(y) - c_\alpha w^\alpha \mathcal{I}_\alpha^{(0,0)}(y)|\leq C(w^\alpha \phi(\frac{y}{2}) + w^{2\alpha}\psi_{\alpha}^{(0)}(y)), \nonumber\\
    |g_{\alpha}(y, w) - & \mathcal{D}(\phi)(y) |\leq C w^\alpha \psi_{\alpha}^{(0)}(y), \nonumber \\
    |h_{\alpha}(y, w) + & \alpha c_\alpha w^\alpha \mathcal{I}_\alpha^{(0,0)}(y)|\leq C(w^\alpha \phi(\frac{y}{2}) + w^{2\alpha}\psi_{\alpha}^{(0)}(y)), \label{eq:funcapprox} \\
    |k_{\alpha}(y, w) + &c_\alpha w^\alpha \ln(1/w) \mathcal{I}_\alpha^{(0,0)}(y) + c_\alpha w^\alpha \mathcal{I}_\alpha^{(0,1)}(y) - \partial_\alpha c_\alpha w^\alpha \mathcal{I}_\alpha^{(0,0)}(y)| \nonumber \\
    &\qquad \qquad  \leq C w^\alpha \Big[ \ln(1/w)(\phi(\frac{y}{2}) + w^{\alpha}\psi_{\alpha}^{(0)}(y)) + w^\alpha \psi_{\alpha}^{(1)}(y)\Big]. \nonumber
\end{align}
In particular, for any compact set $K \subset (0,2),$ $\exists C>0, \; \forall \alpha \in K, \; \forall w \in (0, 1/3], \; \forall y$
\begin{align} 
    \frac{1}{C} (\phi(y) + &w^{\alpha} \psi_{\alpha}^{(0)}(y)) \leq f_{\alpha}(y, w) \leq C (\phi(y) + w^{\alpha} \psi_{\alpha}^{(0)}(y)), \nonumber\\
    &|g_{\alpha}(y, w)| \leq C \ln(1/w) (\phi(y) + w^{\alpha} \psi_{\alpha}^{(0)}(y)), \nonumber\\
    &| h_{\alpha}(y, w)| \leq C w^{\alpha} \psi_{\alpha}^{(0)}(y), \label{eq:majfghk} \\ 
    &|k_{\alpha}(y, w)| \leq C w^{\alpha} \ln(1/w) \psi_{\alpha}^{(1)}(y). \nonumber
\end{align}
\end{lem}

We give then a more precise approximation result for $f_{\alpha}, h_{\alpha}$ and $l_{\alpha}$, defined in \eqref{def:fghk} and \eqref{def:l},  by combining Lemma \ref{L:I} and Lemma \ref{l:approxD}.
\begin{lem}\label{L:fonctionsApprox}
$\exists C>0, \; \forall \Gamma \geq 1 , \; \forall  |y|>\Gamma $ and for any n large enough
\begin{align}
 \big|f_{\alpha}(y, w_n(\theta))-& \phi(y) - c_{\alpha} w_n(\theta)^{\alpha} \psi_{\alpha}^{(0)}(y) \big|
 \leq C f_{\alpha}(y, w_n(\theta)) \Bigl(\frac{1}{\sqrt{\Gamma}} + \frac{1}{n^{1-\alpha/2}} \Bigr), \label{eq:etape1Approxf} \\
 \big|h_{\alpha}(y, w_n(\theta)) + &\alpha c_{\alpha} w_n(\theta)^{\alpha} \psi_{\alpha}^{(0)}(y) \big| \leq C w_n(\theta)^{\alpha} \psi_{\alpha}^{(0)}(y) \Bigl(\frac{1}{\sqrt{\Gamma}} + \frac{1}{n^{1-\alpha/2}}\Bigr), \label{eq:etape1Approxh} \\
\Big|l_{\alpha}(y,  w_n(\theta)) +&  c_{\alpha} w_n(\theta)^{\alpha} \psi_{\alpha}^{(1)}(y) 
- \frac{c_{\alpha}}{2} \ln(\ln n) w_n(\theta)^{\alpha} \psi_{\alpha}^{(0)}(y)\nonumber \\
&-  w_n(\theta)^{\alpha} \bigl[c_{\alpha} \ln(\frac{\delta }{\sigma}) + \partial_\alpha c_{\alpha} \bigr]\psi_{\alpha}^{(0)}(y)) \Big|\label{eq:Lbarre}\\
    \leq& C w_n(\theta)^{\alpha} \big(\psi_{\alpha}^{(0)}(y)+\frac{\psi_{\alpha}^{(1)}(y)}{\ln(\ln n)}\big)\Bigl(\frac{\ln(\ln n)}{\sqrt{|\Gamma|}} + \ln(n) \phi(\frac{\Gamma}{3}) + \frac{\ln(n)}{n^{1- \alpha/2}}\Bigr). \nonumber
\end{align}
\end{lem}
Using the previous estimates, we finally obtain the next  proposition. With 
$ u_n(\theta)^T G_n(\theta)$ given by \eqref{E:scoren}, this proposition is the key
 ingredient to identify the information matrix (the limit of $\E^{\theta}( u_n(\theta)^T G_n(\theta)G_n(\theta)^T u_n(\theta))$). 
\begin{prop}\label{L:tout}We have
\begin{align*}
   &\bigg| \int \frac{(h_{\alpha})^2 }{f_{\alpha}} (y, w_n(\theta))dy 
   - \frac{\alpha^2 \kappa_0(\theta)}{n^{1 -\alpha/2 } \ln(n)^{\alpha/2}}  \bigg|\leq  C \frac{1}{n^{1 -\alpha/2 } \ln(n)^{\alpha/2}} \frac{\ln(\ln n)}{\ln(n)^{1/4\wedge\alpha}},\\
   &\bigg| \int \frac{h_{\alpha} l_{\alpha}}{f_{\alpha}} (y, w_n(\theta))dy 
   + \frac{\alpha \kappa_0(\theta)  \kappa_1(\theta)}{n^{1 -\alpha/2 } \ln(n)^{\alpha/2}}  \bigg| \leq  C \frac{1}{n^{1 -\alpha/2 } \ln(n)^{\alpha/2}} \frac{(\ln(\ln n))^2}{\ln(n)^{1/4\wedge\alpha}},\\
   &\bigg| \int \frac{(l_{\alpha})^2}{f_{\alpha}} (y, w_n(\theta))dy - \frac{\kappa_0(\theta)(\kappa_1(\theta)^2+1/\alpha^2)}{n^{1 -\alpha/2 } \ln(n)^{\alpha/2}}  \bigg|\leq  C \frac{1}{n^{1 -\alpha/2 } \ln(n)^{\alpha/2}} \frac{(\ln(\ln n))^3}{\ln(n)^{1/4\wedge\alpha}},
\end{align*}
with $\kappa_0(\theta)=\frac{2  c_{\alpha}}{\alpha(2-\alpha)^{\alpha/2}} (\frac{ \delta }{\sigma})^{\alpha}$ and $\kappa_1(\theta)=\ln(\frac{ \delta }{\sigma}) + \frac{ \partial_\alpha c_{\alpha}}{c_{\alpha}} - \frac{\ln(2-\alpha)}{2} - \frac{1}{\alpha}$.
\end{prop}

\subsection{Higher order derivatives} \label{S:suporder}

In this section, we give some bounds on the functions $f_{\alpha}^{(k, l, m)}$ defined in \eqref{def:fklm} that appear in the expression of the Hessian matrix $J_n$.

\begin{lem}\label{L:MajFunc}For any compact set $K \subset (0,2),$ and for $k,l,m \in \N$,
$\exists C>0, \; \forall \alpha \in K, \; \forall w \in (0, 1/3], \; \forall y$
\begin{align}
    |f_{\alpha}^{(k, 0, 0)}(y, w)| &\leq C\ln(1/w)^k \bigl(\phi(y)+w^\alpha \psi_\alpha^{(0)}(y) \bigr), \label{eq:majfk00}\\
    |f_{\alpha}^{(k, l, m)}(y, w)| &\leq C \ln(1/w)^{m} w^\alpha \psi_\alpha^{(m)}(y) \quad \text{if } l+m>0. \label{eq:majfklm}
\end{align}
\end{lem}

We are now able to state sharp upper bounds for the second and third order derivatives of $\ln p_{1/n}$.
\begin{lem}\label{L:der2der3} For $a,b \in \{\sigma, \delta, \alpha \}$, $\exists C>0, \;\exists p>0, \; \forall y$
\begin{align}
    |\partial^2_{\sigma \sigma} \ln p_{1/n}(y, \theta_0)| &\leq C \ln(n)^p, \label{eq:boundder2} \\
    |\partial^2_{a b} \ln p_{1/n}(y, \theta_0) |&\leq C \ln(n)^p \frac{w_n(\theta_0)^{\alpha_0}\psi_{\alpha_0}^{(2)}(\frac{\sqrt{n}}{\sigma_0} y)}{\phi(\frac{\sqrt{n}}{\sigma_0} y)+w_n(\theta_0)^{\alpha_0}\psi_{\alpha_0}^{(0)}(\frac{\sqrt{n}}{\sigma_0} y)} \quad \text{if } (a,b)\neq(\sigma, \sigma). \nonumber \end{align}
For any compact set $V \subset (0, \infty) \times (0, \infty) \times (0, 2)$ and for $a,b,c \in \{\sigma, \delta, \alpha \}$, 
$\exists C>0, \; \exists p>0, \;\forall \theta \in V, \;\forall y$
\begin{align}
    |\partial^3_{\sigma^3} \ln p_{1/n}(y, \theta)| &\leq C \ln(n)^p, \label{eq:boundder3}\\
    |\partial^3_{a b c} \ln p_{1/n}(y, \theta)| &\leq C \ln(n)^p \frac{w_n(\theta)^{\alpha}\psi_{\alpha}^{(3)} (\frac{\sqrt{n}}{\sigma} y)}{\phi(\frac{\sqrt{n}}{\sigma} y)+w_n(\theta)^{\alpha}\psi_{\alpha}^{(0)}(\frac{\sqrt{n}}{\sigma} y)} \quad \text{if } (a,b,c)\neq(\sigma, \sigma, \sigma).\nonumber \end{align}
\end{lem}
In particular, for $a,b \in \{\sigma, \delta, \alpha \}$ and $(a,b)\neq(\sigma, \sigma)$, we deduce from \eqref{eq:boundder2} that for some $p>0$, $\forall y$
$$|\partial^2_{a b} \ln p_{1/n}(y, \theta_0) |\leq C \ln(n)^p \bigl(1 + \ln(1+|y|)^2 \bigr).$$
 For $(Y_t)$ defined by \eqref{E:Levy} the previous bound 
 implies that
$ \E^{\theta_0}|\partial^2_{a b} \ln p_{1/n}(\Delta _i^n Y, \theta_0) | \leq C \ln(n)^p.$
However, we later require that 
$$ \E^{\theta_0}|\partial^2_{a b} \ln p_{1/n}(\Delta _i^n Y, \theta_0) | \to 0,$$
with some rate. This can be obtained from \eqref{eq:boundder2}. Indeed, from \eqref{eq:lienpf} and \eqref{eq:majfghk}, we have
\begin{align*}
    \E^{\theta_0}|\partial^2_{a b} \ln p_{1/n}(\Delta _i^n Y, \theta_0) | &= \int |\partial^2_{a b} \ln p_{1/n}(y, \theta_0) |  p_{1/n}(y, \theta_0) dy\\
    \leq C & \int |\partial^2_{a b} \ln p_{1/n}(y, \theta_0) | \frac{\sqrt{n}}{\sigma_0}\bigl(\phi(\frac{\sqrt{n}}{\sigma_0} y)+w_n(\theta_0)^{\alpha_0}\psi_{\alpha_0}^{(0)}(\frac{\sqrt{n}}{\sigma_0} y)\bigr)dy.
\end{align*}
After a change of variable $\frac{\sqrt{n}}{\sigma_0}y \to y$, the bound given in \eqref{eq:boundder2} allows for a rate $\frac{\ln(n)^p}{n^{1-\alpha_0/2}}$ as the denominator is compensated when taking the expectation. The same holds for the third-order derivatives.

\section{Proofs} \label{S:proof}

\subsection{LAN Property} \label{Ss:PLAN}

We recall that $p_{1/n}$ (given by \eqref{eq:lienpf} with $w_n(\theta)=  n^{1/2 - 1/\alpha} \delta/\sigma $) is the density of $Y_{1/n}$.

\subsubsection{Proof of Proposition \ref{Th:TCL}}
We set  $u_n(\theta_0)^T   G_n(\theta_0) =\sum_{i=0}^{n-1} \zeta_{i}^n$ with $\zeta_{i}^{n} = (\zeta_{i}^{k,n})_{1 \leq k \leq 3}$ defined by
\begin{align}
    \zeta_{i}^{1,n} &= \frac{1}{\sqrt{n}} \frac{\partial_\sigma p_{1/n}}{p_{1/n}}(\Delta _i^n Y, \theta_0),   \nonumber \\
    \zeta_{i}^{2,n} &= \frac{\ln(n)^{\alpha_0/4}}{n^{\alpha_0/4}} \frac{\partial_\delta p_{1/n}}{p_{1/n}}(\Delta _i^n Y, \theta_0),  \label{eq:zeta} \\
    \zeta_{i}^{3,n} &= \frac{\ln(n)^{\alpha_0/4}}{n^{\alpha_0/4}} \left(-\frac{ \delta_0}{2 \alpha_0}\bigl(\ln(n) - \ln(\ln n) \bigr) \frac{\partial_\delta p_{1/n}}{p_{1/n}} + \frac{\partial_\alpha p_{1/n}}{p_{1/n}} \right) (\Delta _i^n Y, \theta_0).  \nonumber
\end{align}
Since the variables  $(\zeta_{i}^n)_{0 \leq i \leq n-1}$ are i.i.d. the convergence in law of $\sum_{i=0}^{n-1} \zeta_{i}^n$ is a consequence of the Lindeberg-Feller Central Limit Theorem (we refer to Hall and Heyde \cite{HallHeyde}) and it is sufficient to show
\begin{align}
    & \E^{\theta_0} (\zeta_0^{ n} ) = 0,  \label{Eq:point1cvTCL} \\
    &  n \; \E^{\theta_0} (|\zeta_{0}^{k, n}|^{3}) \to 0, \quad \text{for} \quad k=1,2,3, \label{Eq:point2cvTCL} \\
    & n \; \E^{\theta_0} (\zeta_0^{k, n} \zeta_0^{l, n}) \to I_{kl}(\theta_0) 
     \quad \text{for} \quad k,l=1,2,3, \label{Eq:point3cvTCL}
    \end{align}
with $I(\theta_0)$ defined by \eqref{eq:I}.

Equation \eqref{Eq:point1cvTCL} is immediate since $p_{1/n}(y, \theta_0)$ is the density of $\Delta_0^n Y=Y_{1/n}$, hence 
\begin{align*}
    \E^{\theta_0} (\zeta_0^{1,n}) = \frac{1}{\sqrt{n}}\int \frac{\partial_\sigma p_{1/n}}{p_{1/n}} p_{1/n}(y, \theta_0) dy = \frac{1}{\sqrt{n}}\partial_\sigma \left(\int p_{1/n}(y,\theta_0) dy \right) = 0,
\end{align*}
and similar results hold for $\zeta_{0}^{2,n}$ and $\zeta_{0}^{3,n}$.

We now prove equation \eqref{Eq:point2cvTCL}. After the change of variable $\frac{\sqrt{n}}{\sigma_0} y \to y$, we have using the bounds \eqref{eq:majfghk}
\begin{align*}
    \E^{\theta_0} (|\zeta_0^{1,n}|^3) &= \frac{1}{n^{3/2}} \int \frac{|\partial_\sigma p_{1/n}|^3}{|p_{1/n}|^2}(y, \theta_0) dy
    =  \frac{1}{n^{3/2}} \frac{1}{\sigma_0^3} \int  \frac{|g_{\alpha_0}|^3}{|f_{\alpha_0}|^2}(y, w_n( \theta_0))dy\\
    &\leq C \frac{(\ln(1/w_n( \theta_0)))^3}{n^{3/2}} \int  |\phi(y) + w_n( \theta_0)^{\alpha} \psi_{\alpha}^{(0)}(y)|dy,
\end{align*}
hence 
$$ n \; \E^{\theta_0} (|\zeta_0^{1,n}|^3) \leq C \frac{(\ln(n))^3}{\sqrt{n}} \to 0.$$
We similarly obtain
\begin{align*}
    \E^{\theta_0} (|\zeta_0^{2,n}|^3) &= \frac{\ln(n)^{3\alpha_0/4}}{n^{3 \alpha_0/4}} \int \frac{|\partial_\delta p_{1/n}|^3}{|p_{1/n}|^2}(y, \theta_0) dy 
     =   \frac{\ln(n)^{3\alpha_0/4}}{n^{3 \alpha_0/4}} \frac{1}{\delta_0^3} \int  \frac{|h_{\alpha_0}|^3}{|f_{\alpha_0}|^2}(y, w_n(\theta_0))dy \\
    &\leq C \frac{\ln(n)^{3\alpha_0/4}}{n^{3 \alpha_0/4}} \int  \frac{|w_n( \theta_0)^{\alpha} \psi_{\alpha}^{(0)}(y)|^3}{|\phi(y) + w_n( \theta_0)^{\alpha} \psi_{\alpha}^{(0)}(y)|^2}dy \\
    &\leq C \frac{\ln(n)^{3\alpha_0/4}}{n^{3 \alpha_0/4}} \frac{1}{n^{1 - \alpha_0/2}} \int  \psi_{\alpha_0}^{(0)}(y)dy,
\end{align*}
hence
$$ n \; \E^{\theta_0} (|\zeta_0^{2,n}|^3) \leq C \frac{\ln(n)^{3\alpha_0/4}}{n^{\alpha_0/4}} \to 0,$$
and the proof is the same for $\zeta_0^{3,n}$.

We now consider the convergence \eqref{Eq:point3cvTCL}.
\paragraph*{Sigma-sigma.} Using again that $p_{1/n}(y, \theta_0)$ is the density of $\Delta_0^n Y$, we have after a change of variable
\begin{align*}
    \E^{\theta_0}((\zeta_0^{1,n})^2) & = \frac{1}{n} \int \frac{(\partial_\sigma p_{1/n})^2}{p_{1/n}}(y, \theta_0) dy 
=    \frac{1}{n}  \frac{1}{\sigma_0^2} \int \frac{g_{\alpha_0}^2}{f_{\alpha_0}}(y, w_n(\theta_0))dy.
\end{align*}
From \eqref{eq:funcapprox}, $g_{\alpha_0}$ can be approximated by $\mathcal{D}(\phi)$ and $f_{\alpha_0}$ can be approximated by $\phi$.
Therefore, we prove that
\begin{equation}\label{eq:intsigma}
    \bigg| \int \frac{g_{\alpha_0}^2}{f_{\alpha_0}}(y, w_n(\theta_0))dy - \int \frac{(\mathcal{D}(\phi)(y))^2}{\phi(y)}dy \bigg| \leq C \frac{1}{n^{(1-\alpha_0/2)/2}}.
\end{equation}
We decompose the integral to get 
\begin{align*}
    \bigg| \int \frac{g_{\alpha_0}^2}{f_{\alpha_0}}(y, w_n(\theta_0))dy - \int \frac{(\mathcal{D}(\phi)(y))^2}{\phi(y)}dy \bigg| \leq C \bigg( &\int \frac{| g_{\alpha_0}(y, w_n( \theta_0)) - \mathcal{D}(\phi)(y)|^2}{f_{\alpha_0}(y, w_n( \theta_0))}dy\\
    + \int \frac{| g_{\alpha_0}(y, w_n( \theta_0)) - \mathcal{D}(\phi)(y)| |\mathcal{D}(\phi)(y)|}{f_{\alpha_0}(y, w_n( \theta_0))}dy
    &+ \int \bigg|\frac{(\mathcal{D}(\phi)(y))^2}{\phi(y)} - \frac{(\mathcal{D}(\phi)(y))^2}{f_{\alpha_0}(y, w_n( \theta_0))} \bigg|dy \bigg).
\end{align*}
Using \eqref{eq:funcapprox} and \eqref{eq:majfghk}, we have as $\int \psi_{\alpha_0}^{(0)} < \infty$
\begin{align*}
    \int \frac{| g_{\alpha_0}(y, w_n( \theta_0)) - \mathcal{D}(\phi)(y)|^2}{f_{\alpha_0}(y, w_n( \theta_0))}dy \leq C \int \frac{w_n( \theta_0)^{2\alpha_0} (\psi_{\alpha_0}^{(0)}(y))^2}{\phi(y) + w_n( \theta_0)^{\alpha_0} \psi_{\alpha_0}^{(0)}(y)}dy \leq C \frac{1}{n^{1-\alpha_0/2}}.
\end{align*}
Recalling that $\mathcal{D}(\phi)(y) = (1-y^2)\phi(y)$, we have using Cauchy–Schwarz inequality that
\begin{align*}
    \int &\frac{| g_{\alpha_0}(y, w_n( \theta_0)) - \mathcal{D}(\phi)(y)| \; |\mathcal{D}(\phi)(y)|}{f_{\alpha_0}(y, w_n( \theta_0))}dy \\
    & \qquad \qquad \leq \left( \int \frac{| g_{\alpha_0}(y, w_n( \theta_0)) - \mathcal{D}(\phi)(y)|^2}{f_{\alpha_0}(y, w_n( \theta_0))}dy \right)^{1/2} \left(\int \frac{(\mathcal{D}(\phi)(y))^2}{f_{\alpha_0}(y, w_n( \theta_0))}dy\right)^{1/2}\\
    &\qquad \qquad \leq C \frac{1}{n^{(1-\alpha_0/2)/2}} \left(\int (1+y^4)\phi(y)dy\right)^{1/2}  \leq C \frac{1}{n^{(1-\alpha_0/2)/2}}.
\end{align*}
By decomposing the integral and using \eqref{eq:funcapprox} and \eqref{eq:majfghk}, we have
\begin{align*}
    \int \bigg|\frac{(\mathcal{D}(\phi)(y))^2}{\phi(y)} - & \frac{(\mathcal{D}(\phi)(y))^2}{f_{\alpha_0}(y, w_n( \theta_0))} \bigg|dy 
    = \int\frac{(\mathcal{D}(\phi)(y))^2}{\phi(y)} \frac{|f_{\alpha_0}(y, w_n( \theta_0)) - \phi(y)|}{f_{\alpha_0}(y, w_n( \theta_0))} dy \\
   &\leq C \int\frac{(\mathcal{D}(\phi)(y))^2}{\phi(y)} \frac{w_n( \theta_0)^{\alpha_0} \psi_{\alpha_0}^{(0)}(y)}{\phi(y) + w_n( \theta_0)^{\alpha_0} \psi_{\alpha_0}^{(0)}(y)} dy \\
    &\leq C \left(w_n( \theta_0)^{\alpha_0} \int_{|y|\leq\Gamma_n} \frac{(\mathcal{D}(\phi)(y))^2}{\phi(y)^2}\psi_{\alpha_0}^{(0)}(y)dy
    + \int_{|y|>\Gamma_n} \frac{(\mathcal{D}(\phi)(y))^2}{\phi(y)} \right)\\
    & \leq C \Bigl(\frac{\Gamma_n^4}{n^{1-\alpha_0/2}} + \phi(\frac{\Gamma_n}{2}) \Bigr),
\end{align*}
and we conclude \eqref{eq:intsigma} by combining these results and taking $\Gamma_n = n^{(1-\alpha_0/2)/8}$.

Finally, using that $\mathcal{D}(\phi)(y) = (1-y^2)\phi(y)$, we get 
$$\int \frac{(\mathcal{D}(\phi)(y))^2}{\phi(y)}dy = \int (1-2y^2+y^4)\phi(y)dy = 1-2+3=2,$$
and we conclude.

\paragraph*{Delta-delta.}
As previously, we have after a change of variable
\begin{align*}
    \E^{\theta_0}((\zeta_0^{2,n})^2) & = \frac{\ln(n)^{\alpha_0/2}}{n^{\alpha_0/2}} \int \frac{(\partial_\delta p_{1/n})^2}{p_{1/n}}(y, \theta_0) dy 
     =\frac{\ln(n)^{\alpha_0/2}}{n^{\alpha_0/2}}  \frac{1}{\delta_0^2} \int \frac{(h_{\alpha_0})^2}{f_{\alpha_0}}(y, w_n(\theta_0))dy.
\end{align*}
From Proposition \ref{L:tout}, we get
\begin{align*}
   \bigg| \int \frac{(h_{\alpha_0})^2 }{f_{\alpha_0}} (y, w_n(\theta_0))dy 
   - \frac{\alpha_0^2 \kappa_0(\theta_0) }{n^{1 -\alpha_0/2 } \ln(n)^{\alpha_0/2}} \bigg|\leq  C \frac{1}{n^{1 -\alpha_0/2 } \ln(n)^{\alpha_0/2}} \frac{\ln(\ln n)}{\ln(n)^{1/4\wedge\alpha_0}},
\end{align*}
with $\kappa_0(\theta_0)=\frac{2  c_{\alpha_0}}{\alpha_0(2-\alpha_0)^{\alpha_0/2}} (\frac{ \delta_0 }{\sigma_0})^{\alpha_0}$ and we conclude.

\paragraph*{Delta-alpha.}
Similarly, we have 
\begin{align*}
     \E^{\theta_0}(\zeta_0^{2,n} \zeta_0^{3,n}) &=\frac{\ln(n)^{\alpha_0/2}}{n^{\alpha_0/2}}  \int \frac{(-\frac{\delta_0}{2 \alpha_0}[\ln(n) - \ln(\ln n)]\partial_\delta p_{1/n} +\partial_\alpha p_{1/n})  \partial_\delta p_{1/n}}{p_{1/n}}(y, \theta_0) dy \\
    &= - \frac{\ln(n)^{\alpha_0/2}}{n^{\alpha_0/2}} \frac{1}{\delta_0}\int \frac{h_{\alpha_0} l_{\alpha_0}}{f_{\alpha_0}}(y, w_n(\theta_0))dy,
\end{align*}
with $l_{\alpha}$ defined in \eqref{def:l}. From Lemma \ref{L:tout} we have
\begin{align*}
   \bigg| \int \frac{h_{\alpha_0} l_{\alpha_0}}{f_{\alpha_0}} (y, w_n(\theta_0))dy 
   &+ \frac{\alpha_0 \kappa_0(\theta_0) \kappa_1(\theta_0)}{n^{1 -\alpha_0/2 } \ln(n)^{\alpha_0/2}}\bigg| \leq  C \frac{1}{n^{1 -\alpha_0/2 } \ln(n)^{\alpha_0/2}} \frac{(\ln(\ln n))^2}{\ln(n)^{1/4\wedge\alpha_0}},
\end{align*}
with $\kappa_1(\theta_0)=\ln(\frac{ \delta_0 }{\sigma_0}) + \frac{ \partial_\alpha c_{\alpha_0}}{c_{\alpha_0}} - \frac{\ln(2-\alpha_0)}{2} - \frac{1}{\alpha_0}$ and we conclude.

\paragraph*{Alpha-alpha.}
We now turn to
\begin{align*}
 \E^{\theta_0}((\zeta_0^{3,n})^2 ) &=\frac{\ln(n)^{\alpha_0/2}}{n^{\alpha_0/2}} \int \frac{(-\frac{\delta_0}{2 \alpha_0}[\ln(n) - \ln(\ln n)]\partial_\delta p_{1/n} +\partial_\alpha p_{1/n})^2 }{p_{1/n}}(y, \theta_0) dy \\
 &= \frac{\ln(n)^{\alpha_0/2}}{n^{\alpha_0/2}} \int \frac{(l_{\alpha_0})^2}{f_{\alpha_0}} (y, w_n(\theta_0))dy.
\end{align*}
From Proposition \ref{L:tout} we have
\begin{align*}
    \bigg| \int \frac{(l_{\alpha_0})^2}{f_{\alpha_0}} (y, w)dy - &\frac{\kappa_0(\theta_0)(\kappa_1(\theta_0)^2+1/\alpha_0^2)}{n^{1 -\alpha_0/2 } \ln(n)^{\alpha_0/2}}\bigg| \leq  C \frac{1}{n^{1 -\alpha_0/2 } \ln(n)^{\alpha_0/2}} \frac{(\ln(\ln n))^3}{\ln(n)^{1/4\wedge\alpha_0}},
\end{align*}
and we conclude.

\paragraph*{Cross-terms $\sigma \delta$ and $\sigma \alpha$.}
We finally consider the cross-terms $\zeta_0^{1,n} \zeta_0^{2,n}$ and $\zeta_0^{1,n} \zeta_0^{3,n}$. We write again that
\begin{align*}
    \E^{\theta_0}(\zeta_0^{1,n} \zeta_0^{2,n}) & = \frac{1}{\sqrt{n}} \frac{\ln(n)^{\alpha_0/4}}{n^{\alpha_0/4}} \int \frac{\partial_\sigma p_{1/n} \partial_\delta p_{1/n}}{p_{1/n}}(y, \theta_0) dy \\
    & = \frac{1}{\sqrt{n}} \frac{\ln(n)^{\alpha_0/4}}{n^{\alpha_0/4}} \frac{1}{\sigma_0 \delta_0} \int \frac{g_{\alpha_0} h_{\alpha_0}}{f_{\alpha_0}}(y, w_n(\theta_0))dy .
\end{align*}
Using the bounds given in \eqref{eq:majfghk}, we have
\begin{align*}
    \bigg|\int \frac{g_{\alpha_0} h_{\alpha_0}}{f_{\alpha_0}}(y, w_n(\theta_0))dy\bigg|
    &\leq C \int \frac{\ln(1/w_n(\theta_0))(\phi(y) +w_n(\theta_0)^{\alpha_0} \psi_{\alpha_0}^{(0)}(y))w_n(\theta_0)^{\alpha_0}\psi_{\alpha_0}^{(0)}(y)}{\phi(y)+w_n(\theta_0)^{\alpha_0}\psi_{\alpha_0}^{(0)}(y)}dy\\
    &\leq C \ln(1/w_n(\theta_0))w_n(\theta_0)^{\alpha_0} \int \psi_{\alpha_0}^{(0)}(y)dy\\
    &\leq C\frac{\ln(n)}{n^{1-\alpha_0/2}}.
\end{align*}
Hence,
\begin{align*}
    n \; |\E^{\theta_0}(\zeta_0^{1,n} \zeta_0^{2,n})| \leq C \frac{n}{\sqrt{n}} \frac{\ln(n)^{\alpha_0/4}}{n^{\alpha_0/4}} \frac{\ln(n)}{n^{1-\alpha_0/2}} \leq C \frac{\ln(n)^{1+\alpha_0/4}}{n^{(1-\alpha_0/2)/2}} \to 0.
\end{align*}
The term $\zeta_0^{1,n} \zeta_0^{3,n}$ is done similarly.

\subsubsection{Proof of Proposition \ref{Th:CvIn}}

\noindent
\underline{Proof of \eqref{E:CvIn}.}

We set $u_n(\theta_0)^T J_n(\theta_0) u_n(\theta_0) = \sum_{i=0}^{n-1} \xi_{i}^n$ where the symmetric matrix $\xi_{i}^n=(\xi_i^{k l,n})_{1\leq k,l \leq 3}$ is given by
\begin{align*}
    \xi_i^{11,n} =& \frac{1}{n} \partial_\sigma \bigl(\frac{\partial_\sigma p_{1/n}}{p_{1/n}} \bigr) (\Delta _i^n Y, \theta_0), \\
    \xi_i^{12,n} =& \frac{\ln(n)^{\alpha_0/4}}{\sqrt{n}n^{\alpha_0/4}} \partial_\sigma \bigl(\frac{\partial_\delta p_{1/n}}{p_{1/n}}\bigr) (\Delta _i^n Y, \theta_0),\\
    \xi_i^{13,n} =& \frac{\ln(n)^{\alpha_0/4}}{\sqrt{n}n^{\alpha_0/4}} \left(-\frac{\delta_0}{2 \alpha_0}\bigl[\ln(n) - \ln(\ln n)\bigr]\partial_\sigma \bigl(\frac{\partial_\delta p_{1/n}}{p_{1/n}}\bigr) + \partial_\sigma \bigl(\frac{\partial_\alpha p_{1/n}}{p_{1/n}}\bigr) \right) (\Delta _i^n Y, \theta_0),
    \end{align*}
  \begin{align*}
    \xi_i^{22,n} =& \frac{\ln(n)^{\alpha_0/2}}{n^{\alpha_0/2}} \partial_\delta \bigl(\frac{\partial_\delta p_{1/n}}{p_{1/n}}\bigr) (\Delta _i^n Y, \theta_0), \\
    \xi_i^{23,n} = & \frac{\ln(n)^{\alpha_0/2}}{n^{\alpha_0/2}} \left(-\frac{\delta_0}{2 \alpha_0}\bigl[\ln(n) - \ln(\ln n)\bigr]\partial_\delta \bigl(\frac{\partial_\delta p_{1/n}}{p_{1/n}}\bigr) + \partial_\delta \bigl(\frac{\partial_\alpha p_{1/n}}{p_{1/n}}\bigr) \right) (\Delta _i^n Y, \theta_0),\\
    \xi_i^{33,n} = &\frac{\ln(n)^{\alpha_0/2}}{n^{\alpha_0/2}} \bigg(\frac{\delta_0^2}{(2\alpha_0)^2}\bigl[\ln(n) - \ln(\ln n)\bigr]^2\partial_\delta \bigl(\frac{\partial_\delta p_{1/n}}{p_{1/n}}\bigr)\\
    \quad & - \frac{\delta_0}{\alpha_0}\bigl[\ln(n) - \ln(\ln n)\bigr]\partial_\alpha \bigl(\frac{\partial_\delta p_{1/n}}{p_{1/n}}\bigr) + \partial_\alpha \bigl(\frac{\partial_\alpha p_{1/n}}{p_{1/n}}\bigr) \bigg) (\Delta _i^n Y, \theta_0).
\end{align*}
The variables $(\xi_i^{n})_{0 \leq i \leq n-1}$ are i.i.d. and in order to prove that
$$u_n(\theta_0)^T J_n(\theta_0) u_n(\theta_0) \xrightarrow[]{\PP^{\theta_0}} - I(\theta_0),$$
we just check that for $1\leq k,l \leq 3$ 
\begin{align}
n &\E^{\theta_0}( \xi_0^{k l,n})\to - I_{kl}(\theta_0), \label{eq:Espi}\\
n &\E^{\theta_0}(|\xi_0^{k l,n}|^2) \to 0. \label{eq:Espi2} 
\end{align}

\paragraph*{Proof of \eqref{eq:Espi}.}
Using that $p_{1/n}(y, \theta_0)$ is the density of $\Delta _i^n Y$, we have
\begin{align*}
    \E^{\theta_0}\left(\partial_{a} \frac{\partial_{b} p_{1/n}}{p_{1/n}} (\Delta _i^n Y, \theta_0)\right)
    &= \int \frac{\partial^2_{a b} p_{1/n}  p_{1/n} -\partial_{a} p_{1/n} \partial_{b} p_{1/n} }{p_{1/n}} (y, \theta_0)dy\\
    &= \partial_{a} \Bigl(\int\partial_{b} p_{1/n}(y, \theta_0)dy\Bigr) - \int \frac{\partial_{a} p_{1/n} \partial_{b} p_{1/n} }{p_{1/n}} (y, \theta_0)dy\\
    &= -\int  \frac{\partial_{a} p_{1/n} \partial_{b} p_{1/n} }{p_{1/n}} (y, \theta_0)dy.
\end{align*}
Hence,
\begin{align*}
    \E^{\theta_0}(\xi_{0}^{11,n}) =&- \frac{1}{n} \int \frac{(\partial_\sigma p_{1/n})^2}{p_{1/n}} (y, \theta_0)dy,\\
    \E^{\theta_0}(\xi_{0}^{12,n}) =&- \frac{\ln(n)^{\alpha_0/4}}{\sqrt{n}n^{\alpha_0/4}} \int \frac{\partial_\sigma p_{1/n}\partial_\delta p_{1/n}}{p_{1/n}} (y, \theta_0)dy,\\
    \E^{\theta_0}(\xi_{0}^{13,n}) =&- \frac{\ln(n)^{\alpha_0/4}}{\sqrt{n}n^{\alpha_0/4}} \int \frac{\partial_\sigma p_{1/n}(-\frac{\delta_0}{2 \alpha_0}[\ln(n) - \ln(\ln n)]\partial_\delta p_{1/n} + \partial_\alpha p_{1/n} )}{p_{1/n}}  (y, \theta_0)dy,\\
    \E^{\theta_0}(\xi_{0}^{22,n}) =& -\frac{\ln(n)^{\alpha_0/2}}{n^{\alpha_0/2}} \int \frac{(\partial_\delta p_{1/n})^2}{p_{1/n}} (y, \theta_0)dy,\\
    \E^{\theta_0}(\xi_{0}^{23,n}) = & -\frac{\ln(n)^{\alpha_0/2}}{n^{\alpha_0/2}} \int \frac{\partial_\delta p_{1/n}(-\frac{\delta_0}{2 \alpha_0}[\ln(n) - \ln(\ln n)]\partial_\delta p_{1/n} + \partial_\alpha p_{1/n} )}{p_{1/n}} (y, \theta_0)dy,\\
    \E^{\theta_0}(\xi_{0}^{33,n}) = &-\frac{\ln(n)^{\alpha_0/2}}{n^{\alpha_0/2}} \int \frac{(-\frac{\delta_0}{2 \alpha_0}[\ln(n) - \ln(\ln n)]\partial_\delta p_{1/n} + \partial_\alpha p_{1/n} )^2}{p_{1/n}} (y, \theta_0)dy,
\end{align*}
and the proof of $n \E^{\theta_0}(\xi_0^{k l,n})\to - I_{kl}(\theta_0)$ for $1\leq k, l\leq 3$ has been done in equation \eqref{Eq:point3cvTCL}.

\paragraph*{Proof of \eqref{eq:Espi2}.}
Using Lemma \ref{L:der2der3}, $\exists p >0$
$$n\,\E^{\theta_0}\bigl(|\xi_{0}^{11,n}|^2\bigr)=\frac{1}{n} \E^{\theta_0}\bigl(|\partial^2_{\sigma \sigma} \ln p_{1/n}(\Delta _0^n Y, \theta_0)|^2\bigr) \leq \frac{C \ln(n)^p}{n} \to 0.$$
Similarly, using a change of variable $\frac{\sqrt{n}}{\sigma_0} y \to y$, \eqref{eq:majfghk} and Lemma \ref{L:der2der3}, we have for $a,b \in \{\sigma, \delta, \alpha \}$ with $(a,b)\neq(\sigma, \sigma)$, $\exists p >0$
\begin{align*}
    \E^{\theta_0}\bigl(|\partial^2_{a b} \ln p_{1/n}(\Delta _0^n Y, \theta_0)|^2\bigr) &= \int |\partial^2_{a b} \ln p_{1/n}(y, \theta_0)|^2 p_{1/n}(y, \theta_0) dy \\
    \leq C \ln(n)^p & \int \left(\frac{w_n(\theta_0)^{\alpha_0}\psi_{\alpha_0}^{(2)}(y)}{\phi(y)+w_n(\theta_0)^{\alpha_0}\psi_{\alpha_0}^{(0)}(y)}\right)^2 \bigl(\phi(y)+w_n(\theta_0)^{\alpha_0}\psi_{\alpha_0}^{(0)}(y) \bigr)dy\\
    \leq C \ln(n)^p & w_n(\theta_0)^{\alpha_0} \int \psi_{\alpha_0}^{(4)}(y)dy \leq \frac{C \ln(n)^p}{n^{1-\alpha_0/2}}.
\end{align*}
This allows us to conclude that for $(k,l)\neq(1,1)$, $n\,\E^{\theta_0}\bigl(|\xi_{0}^{kl,n}|^2\bigr) \to 0$.


\noindent
\underline{Proof of \eqref{E:CvUnif}.}

The matrix  $J_n(\theta)$ is given by
$$J_n(\theta)  = \sum_{i=0}^{n-1} {\small \begin{pmatrix} \partial^2_{\sigma \sigma} \ln p_{1/n} & \partial^2_{ \delta \sigma} \ln p_{1/n} & \partial^2_{\alpha\sigma } \ln p_{1/n} \\
    \partial^2_{\sigma\delta } \ln p_{1/n} & \partial^2_{\delta \delta} \ln p_{1/n}& \partial^2_{\alpha\delta } \ln p_{1/n}\\ 
    \partial^2_{\sigma\alpha } \ln p_{1/n} & \partial^2_{\delta\alpha } \ln p_{1/n} & \partial^2_{\alpha \alpha}\ln p_{1/n} \end{pmatrix}}(\Delta _i^n Y, \theta).$$
From Taylor's formula, we have for $a,b \in \{\sigma, \delta, \alpha \}$
\begin{align}
    \partial^2_{a b} \ln p_{1/n}(y, \theta) - \partial^2_{a b} \ln p_{1/n}(y, \theta_0)
&= (\sigma-\sigma_0) \int_0^1 \partial^3_{\sigma a b} \ln p_{1/n} (y, \theta_0 +  t(\theta-\theta_0)) dt \nonumber \\
+ (\delta-\delta_0) &\int_0^1 \partial^3_{\delta a b} \ln p_{1/n} (y, \theta_0 + t(\theta-\theta_0)) dt \nonumber \\
+ (\alpha-\alpha_0) &\int_0^1 \partial^3_{\alpha a b} \ln p_{1/n} (y, \theta_0 + t(\theta-\theta_0)) dt.\label{eq:Taylord2}
\end{align}
We have using the definition of $V_n^{(r)}$ in \eqref{def:V} that $\exists C>0, \; \exists p>0$, $\forall \theta \in V_n^{(r)},$
\begin{align}\label{eq:majVois}
    |\sigma - \sigma_0|\leq \frac{C}{\sqrt{n}}, \quad 
    |\delta - \delta_0|\leq \frac{C (\ln(n))^p}{n^{\alpha_0/4}}, \quad
    |\alpha - \alpha_0|\leq \frac{C (\ln(n))^p}{n^{\alpha_0/4}},
\end{align}
and we note that $1/\sqrt{n}$ is negligible compared to $(\ln(n))^p/n^{\alpha_0/4}$.

We first consider $\bigl(u_n(\theta_0)^T (J_n(\theta) - J_n(\theta_0)) u_n(\theta_0)\bigr)_{11}$ which has the strongest rate $1/n$. Combining Lemma \ref{L:der2der3}, \eqref{eq:Taylord2} and \eqref{eq:majVois}, we get that $\exists C>0, \; \exists p>0, \forall y$
$$\sup_{\theta \in V_n^{(r)} } |\partial^2_{\sigma \sigma} \ln p_{1/n}(y, \theta) - \partial^2_{\sigma \sigma} \ln p_{1/n}(y, \theta_0)| \leq C \ln(n)^{p} \Bigl(\frac{1}{\sqrt{n}} +\frac{1}{n^{\alpha_0/4}} \Bigr)\bigl(1 + |\ln|y||^3 {\bf 1}_{|y|>1}\bigr),$$
which allows us to conclude that $\exists C>0, \; \exists p>0$
\begin{align*}
    \E^{\theta_0} \Bigl|\sup_{\theta \in V_n^{(r)} }\bigl(u_n(\theta_0)^T (J_n(\theta) - J_n(\theta_0)) u_n(\theta_0)\bigr)_{11}\Bigr| \leq C\frac{n}{n} \frac{\ln(n)^{p}}{n^{\alpha_0/4}} \to 0.
\end{align*}

We now turn to $\bigl(u_n(\theta_0)^T (J_n(\theta) - J_n(\theta_0)) u_n(\theta_0)\bigr)_{kl}$ 
for $(k,l)\neq (1,1)$, which has the rate $\frac{1}{n^{ \alpha_0/2}}$ up to a factor $\ln(n)$. First, recalling that
$$w_n(\theta)^{\alpha} = \frac{\delta^\alpha}{\sigma^\alpha} \frac{1}{n^{1-\alpha/2}} = w_n(\theta_0)^{\alpha_0} \frac{\delta^\alpha \sigma_0^{\alpha_0}}{\sigma^\alpha \delta_0^{\alpha_0}} e^{(\alpha-\alpha_0)\ln(n)/2},$$
we get using \eqref{eq:majVois} that $\exists C>0$
\begin{equation} \label{E:Bw}
    \frac{1}{C}w_n(\theta_0)^{\alpha_0} \leq \inf_{\theta \in V_n^{(r)} } w_n(\theta)^{\alpha} \leq \sup_{\theta \in V_n^{(r)} } w_n(\theta)^{\alpha} \leq C w_n(\theta_0)^{\alpha_0}.
\end{equation}
Moreover, $\forall p>0, \; \forall \varepsilon, \eta>0, \; \exists C>0, \; \forall y$
\begin{align} \label{E:Bphi}
    \frac{1}{C} \psi_{\alpha_0+\varepsilon}^{(p)}(\frac{\sqrt{n}}{\sigma_0} y) \leq \inf_{\theta \in V_n^{(r)} }\psi_{\alpha}^{(p)} (\frac{\sqrt{n}}{\sigma} y) &\leq \sup_{\theta \in V_n^{(r)} }\psi_{\alpha}^{(p)} (\frac{\sqrt{n}}{\sigma} y)\leq C \psi_{\alpha_0-\varepsilon}^{(p)}(\frac{\sqrt{n}}{\sigma_0} y), \nonumber\\
\phi((1+\eta)\frac{\sqrt{n}}{\sigma_0} y) &\leq \inf_{\theta \in V_n^{(r)} }\phi(\frac{\sqrt{n}}{\sigma} y),
\end{align}
for $n$ large enough.
Combining these results with \eqref{eq:Taylord2}, \eqref{eq:majVois} and Lemma \ref{L:der2der3}, we have for $a,b \in \{\sigma, \delta, \alpha \}$ and $(a,b)\neq (\sigma, \sigma)$, $\forall \varepsilon, \eta>0, \; \exists C>0, \; \exists p>0, \forall y$ and for $n$ large enough
\begin{align*}
    \sup_{\theta \in V_n^{(r)} } |\partial^2_{a b} \ln p_{1/n}(y, \theta) &- \partial^2_{a b} \ln p_{1/n}(y, \theta_0)|\\
    & \leq C \ln(n)^p \Bigl(\frac{1}{\sqrt{n}} +\frac{1}{n^{\alpha_0/4}} \Bigr) \sup_{\theta \in V_n^{(r)} }  \frac{w_n(\theta)^{\alpha}\psi_{\alpha}^{(3)} (\frac{\sqrt{n}}{\sigma} y)}{\phi(\frac{\sqrt{n}}{\sigma} y)+w_n(\theta)^{\alpha}\psi_{\alpha}^{(0)}(\frac{\sqrt{n}}{\sigma} y)}\\
    &\leq C \frac{\ln(n)^{p}}{n^{\alpha_0/4}}  \frac{w_n(\theta_0)^{\alpha_0}\psi_{\alpha_0-\varepsilon}^{(3)} (\frac{\sqrt{n}}{\sigma_0} y)}{\phi((1+\eta)\frac{\sqrt{n}}{\sigma_0} y)+w_n(\theta_0)^{\alpha_0}\psi_{\alpha+\varepsilon}^{(0)} (\frac{\sqrt{n}}{\sigma_0} y)}.
\end{align*}
Using \eqref{eq:majfghk} and the change of variable $\frac{\sqrt{n}}{\sigma_0} y \to y$, we have
\begin{align*}
    \E^{\theta_0}\Bigl| \sup_{\theta \in V_n^{(r)} }&\bigl(u_n(\theta_0)^T (J_n(\theta) - J_n(\theta_0)) u_n(\theta_0)\bigr)_{kl} \Bigr| \\
    &\leq C n \ln(n)^{p} \frac{1}{n^{3 \alpha_0/4}} 
    w_n(\theta_0)^{\alpha_0} \int  \psi_{\alpha_0-\varepsilon}^{(3)}(y)
    \frac{\phi(y)+w_n(\theta_0)^{\alpha_0}\psi_{\alpha_0}^{(0)}(y)}{\phi((1+\eta)y)+w_n(\theta_0)^{\alpha_0}\psi_{\alpha_0+\varepsilon}^{(0)}(y)} dy\\
    &\leq C \frac{\ln(n)^{p}}{n^{\alpha_0/4}} \int  \psi_{\alpha_0-2\varepsilon}^{(3)}(y)
    \frac{\phi(y)+w_n(\theta_0)^{\alpha_0}\psi_{\alpha_0}^{(0)}(y)}{\phi((1+\eta)y)+w_n(\theta_0)^{\alpha_0}\psi_{\alpha_0}^{(0)}(y)} dy.
\end{align*}
We now bound this integral.
We define $y^*(w_n(\theta_0))$ as the solution of $\phi(y)=w_n(\theta_0)^{\alpha_0} \psi_{\alpha_0}^{(0)}(y)$, then $y^*(w_n(\theta_0)) \underset{n\to + \infty}{\sim}\sqrt{(2-\alpha_0) \ln(n)}:=M_n$.  We note that $y \to \frac{\phi(y)}{\psi_{\alpha_0}^{(0)}(y)}$ is decreasing on $(0, + \infty)$ after a certain rank, hence
\begin{align*}
    \forall |y|< y^*(w_n(\theta_0)), \quad   
    &\phi(y)+w_n(\theta_0)^{\alpha_0}\psi_{\alpha_0}^{(0)}(y) \leq C \phi(y),\\
    \forall |y|\geq y^*(w_n(\theta_0)), \quad 
    &\phi(y)+w_n(\theta_0)^{\alpha_0}\psi_{\alpha_0}^{(0)}(y) \leq C w_n(\theta_0)^{\alpha_0}\psi_{\alpha_0}^{(0)}(y).
\end{align*}
By decomposing the integral, we write
\begin{align*}
    \int  \psi_{\alpha_0-2\varepsilon}^{(3)}(y)&
    \frac{\phi(y)+w_n(\theta_0)^{\alpha_0}\psi_{\alpha_0}^{(0)}(y)}{\phi((1+\eta)y)+w_n(\theta_0)^{\alpha_0}\psi_{\alpha_0}^{(0)}(y)} dy \\
    &\qquad \leq C \Bigl( \int_{|y|< y^*(w_n(\theta_0))} \psi_{\alpha_0-2\varepsilon}^{(3)}(y) \frac{\phi(y)}{\phi((1+\eta)y)} dy 
    + \int_{|y|\geq y^*(w_n(\theta_0))} \psi_{\alpha_0-2\varepsilon}^{(3)}(y)dy\Bigr)\\
    &\qquad \leq C \Bigl(\int_{|y|< y^*(w_n(\theta_0))} e^{(2\eta + \eta^2) y^2/2}dy + \int_{|y|\geq y^*(w_n(\theta_0))} \psi_{\alpha_0-2\varepsilon}^{(3)}(y)dy\Bigr)\\
    &\qquad \leq C \Bigl( y^*(w_n(\theta_0))e^{(2\eta + \eta^2) y^*(w_n(\theta_0))^2/2} + \int \psi_{\alpha_0-2\varepsilon}^{(3)}(y)dy\Bigr).
\end{align*}
We choose $\varepsilon>0$ such that $2\varepsilon < \alpha_0$, and $\eta>0$ such that $\frac{(2\eta + \eta^2)(2-\alpha)}{2} < \frac{\alpha_0}{4} $. Then $\int \psi_{\alpha_0-2\varepsilon}^{(3)} <C$ and $\frac{\ln(n)^{p}}{n^{\alpha_0/4}} M_{n}e^{(2\eta + \eta^2) M_{n}^2/2} \to 0$, and we conclude that
$$\E^{\theta_0}\Bigl| \sup_{\theta \in V_n^{(r)} }\bigl(u_n(\theta_0)^T (J_n(\theta) - J_n(\theta_0)) u_n(\theta_0)\bigr)_{kl} \Bigr| \to 0.$$

\subsection{SDE estimation} \label{Ss:PSDE}
\subsubsection{Proof of Theorem \ref{Th:estEDS}}
In this proof the convergences in probability and the expectations are taken under $\PP^{\theta_0}$ but we write simply $\PP$ and $\E$ (or $\E_i$ for the conditional expectation).
We recall that  $u_n(\theta)$ is defined in \eqref{eq:un} and  $\tilde{G_n}(\theta)$ in \eqref{E:tildeG}. We set $\tilde{J}_n(\theta)=\nabla_{\theta} \tilde{G}_n(\theta)$. With this notation the result of Theorem \ref{Th:estEDS} is a consequence of the two following convergences (we refer to S{\o}rensen \cite{Sorensen} and to Jacod and S{\o}rensen \cite{JacodSorensenUniforme}).

\begin{enumerate}
\item $(u_n(\theta_0)^T \tilde{G}_n(\theta_0))_n$ stably converges in law with respect to the $\sigma$-field $\sigma(L^{\alpha_0}_s, s \leq 1)$ to $\overline{I}(\theta_0)^{1/2} \mathcal{N}$ where $\mathcal{N}$ is a standard Gaussian variable independent of $\overline{I}(\theta_0)$.
    \item 
    With $V_n^{(r)}$ defined by \eqref{def:V} for $r>0$,  we have the following convergence 
\[\sup_{\theta \in V_n^{(r)} } \left| \left|  u_n(\theta_0)^T \tilde{J}_n(\theta) u_n(\theta_0) + \overline{I}(\theta_0) \right|\right| \xrightarrow[]{\PP} 0.\]
\end{enumerate}

Before proceeding to the proof, we  recall  the following useful result to prove convergence in probability of triangular arrays (see \cite{JacodProtter}).
Let $(\zeta_i^n )$ be a triangular array such that $\zeta_i^n$ is $\mathcal{F}_{\frac{i+1}{n}}$-measurable then the two following conditions imply the convergence $\sum_{i=0}^{n-1} \zeta_i^n \xrightarrow[]{\PP} 0$ 
\begin{equation} \label{E:i}
\sum_{i=0}^{n-1} \vert \mathbb{E}_{ i} \zeta_i^n \vert \xrightarrow[]{\PP} 0, \quad  \quad 
\sum_{i=0}^{n-1} \mathbb{E}_{ i} \vert \zeta_i^n \vert^2  \xrightarrow[]{\PP} 0. 
\end{equation}
We also use  a standard localization procedure, setting $T_K= \inf \{ t \geq 0 ; |X_t| >K \text{ or } | \Delta L_t^{\alpha_0, \tau} | > K \}$ and $\Omega_K=\{ T_K>1\}$, which allows considering that the process is bounded and that the L\'evy process admits the representation (using also  the symmetry of the L\'evy measure)
$$
L_t^{\alpha_0}= \int_0^t \int_{ |z| \leq K} z \tilde{N}^{\alpha_0, \tau}(dz,ds) ,\quad t \in [0,1]
$$
where $\tilde{N}^{\alpha, \tau}$ is a compensated Poisson random measure with compensator $\overline{N}^{\alpha, \tau}(dz,ds)=F^{\alpha, \tau}(dz) ds$.
The localization  ensures  that the assumptions of Proposition \ref{P:euler} are satisfied.

\noindent
1. \underline{Stable convergence in law of  $(u_n(\theta_0)^T \tilde{G}_n(\theta_0))_n$.}

Using the notation of Section \ref{S:Section2} we have
\begin{eqnarray*}
u_n(\theta_0)^T  \tilde{G}_n(\theta_0)= \left(
\begin{array}{c}
-\frac{1}{\sqrt{n}} \sum_{i=0}^{n-1} \frac{\partial_{\sigma} a(X_{i/n}, \sigma_0)}{a(X_{i/n},\sigma_0)} \frac{g_{\alpha_0}}{f_{\alpha_0}}\Bigl(\frac{\sqrt{n}}{a(X_{i/n},\sigma_0)} 
\Delta _i^n X, w_n(X_{i/n},\theta_0)\Bigr) \\
-\frac{\ln(n)^{\alpha_0/4}}{n^{\alpha_0/4}} \sum_{i=0}^{n-1} \frac{1}{\delta_0} \frac{h_{\alpha_0}}{f_{\alpha_0}}\Bigl(\frac{\sqrt{n}}{a(X_{i/n},\sigma_0)} \Delta _i^n X, w_n(X_{i/n},\theta_0)\Bigr)\\
 \frac{\ln(n)^{\alpha_0/4}}{n^{\alpha_0/4}}\sum_{i=0}^{n-1} \frac{l_{\alpha_0}}{f_{\alpha_0}}\Bigl(\frac{\sqrt{n}}{a(X_{i/n},\sigma_0)} \Delta _i^n X, w_n(X_{i/n},\theta_0)\Bigr)
\end{array}
\right)
\end{eqnarray*}
with $f_{\alpha}, g_{\alpha}, h_{\alpha}$ defined by \eqref{def:fghk}, $l_{\alpha}$ defined by \eqref{def:l} and 
$$
w_n(x,\theta)=n^{1/2-1/ \alpha} \frac{\delta c(x)}{a(x, \sigma)}.
$$
The functions $f_{\alpha}, g_{\alpha}, h_{\alpha}, k_{\alpha}$ are even and $\mathcal{C}^2$, with 
first and second order derivatives   given by  (using the integration by part formula)
\begin{align*}
&f'_{\alpha}(y, w) = w^{-1}\int \phi(y- w z) \varphi'_{\alpha}(z) dz, &f^{\prime \prime}_{\alpha}(y, w) =w^{-2} \int \phi(y- w z) \varphi^{\prime \prime}_{\alpha}(z) dz,\\
&g'_{\alpha}(y, w) = w^{-1}\int \mathcal{D}(\phi)(y- w z) \varphi'_{\alpha}(z) dz, &g^{\prime \prime}_{\alpha}(y, w) =w^{-2} \int \mathcal{D}(\phi)(y- w z) \varphi^{\prime \prime}_{\alpha}(z) dz,\\
&h'_{\alpha}(y, w) = w^{-1}\int \phi(y- w z) \mathcal{D}(\varphi_{\alpha})'(z) dz, &h^{\prime \prime}_{\alpha}(y, w) =w^{-2} \int \phi(y- w z) \mathcal{D}(\varphi_{\alpha})^{\prime \prime}(z) dz,\\
&k'_{\alpha}(y, w) = w^{-1}\int \phi(y- w z) \partial_{\alpha}\varphi'_{\alpha}(z) dz, &k^{\prime \prime}_{\alpha}(y, w) =w^{-2} \int \phi(y- w z) \partial_{\alpha}\varphi^{\prime \prime}_{\alpha}(z) dz.
\end{align*}
With similar arguments as the one given in the proof of Lemma \ref{l:approxD}, we can check that
$\forall w \in (0, 1/3], \, \forall y \in \R$
\begin{align*}
&|f'_{\alpha}(y, w)| \leq Cw^{\alpha} (\phi(y)+\psi_{\alpha+1}^{(0)}(y)), &|f^{\prime \prime}_{\alpha}(y, w) | \leq  Cw^{\alpha} (\phi(y)+\psi_{\alpha+2}^{(0)}(y)),\\
&|g'_{\alpha}(y, w)| \leq Cw^{\alpha} (\mathcal{D}(\phi)(y)+\psi_{\alpha+1}^{(0)}(y)), &|g^{\prime \prime}_{\alpha}(y, w) | \leq  Cw^{\alpha} (\mathcal{D}(\phi)(y)+\psi_{\alpha+2}^{(0)}(y)),\\
&|h'_{\alpha}(y, w)| \leq Cw^{\alpha} (\phi(y)+\psi_{\alpha+1}^{(0)}(y)), &|h^{\prime \prime}_{\alpha}(y, w) | \leq  Cw^{\alpha} (\phi(y)+\psi_{\alpha+2}^{(0)}(y)),\\
&|k'_{\alpha}(y, w)| \leq Cw^{\alpha} \ln(1/w)(\phi(y)+\psi_{\alpha+1}^{(1)}(y)), &|k^{\prime \prime}_{\alpha}(y, w) | \leq  Cw^{\alpha} \ln(1/w)(\phi(y)+\psi_{\alpha+2}^{(1)}(y)).
\end{align*}
Combining the previous calculus with \eqref{eq:majfghk}, we deduce $\forall w \in (0, 1/3], \, \forall y \in \R$
\begin{eqnarray} \label{E:Bprime}
 |\frac{g_{\alpha}}{f_{\alpha}}(y,w)| + |(\frac{g_{\alpha}}{f_{\alpha}})'(y,w)| +  |(\frac{g_{\alpha}}{f_{\alpha}})^{\prime \prime}(y,w)| \leq C \ln(1/w),  \quad \quad  \quad \quad \nonumber\\
  |\frac{h_{\alpha}}{f_{\alpha}}(y,w)|+ |(\frac{h_{\alpha}}{f_{\alpha}})'(y,w)|+ |(\frac{h_{\alpha}}{f_{\alpha}})^{\prime \prime}(y,w)| \leq C , \quad \quad  \quad \quad\\
 |\frac{k_{\alpha}}{f_{\alpha}}(y,w)|\leq C \ln(1/w)(1+ \ln(1+|y|)), \quad |(\frac{k_{\alpha}}{f_{\alpha}})'(y,w)|+ |(\frac{k_{\alpha}}{f_{\alpha}})^{\prime \prime}(y,w)| \leq C \ln(1/w). \nonumber
\end{eqnarray}
So the functions $\frac{g_{\alpha_0}}{f_{\alpha_0}}( \cdot,w_n(x,\theta_0)) $, $\frac{h_{\alpha_0}}{f_{\alpha_0}}( \cdot,w_n(x,\theta_0))$ and  $\frac{l_{\alpha_0}}{f_{\alpha_0}}( \cdot,w_n(x,\theta_0))$ satisfy the assumption of Proposition \ref{P:euler} and checking the conditions \eqref{E:i} (we omit the details) we conclude that
$$
u_n(\theta_0)^T  \tilde{G}_n(\theta_0)-u_n(\theta_0)^T  \overline{G}_n(\theta_0) \xrightarrow[]{\PP} 0,
$$ 
where
\begin{eqnarray*}
u_n(\theta_0)^T  \overline{G}_n(\theta_0)= \left(
\begin{array}{c}
-\frac{1}{\sqrt{n}} \sum_{i=0}^{n-1} \frac{\partial_{\sigma} a(X_{i/n}, \sigma_0)}{a(X_{i/n},\sigma_0)} \frac{g_{\alpha_0}}{f_{\alpha_0}}\Bigl(\frac{\sqrt{n}}{a(X_{i/n},\sigma_0)} 
\Delta _i^n \overline{X}, w_n(X_{i/n},\theta_0)\Bigr) \\
-\frac{\ln(n)^{\alpha_0/4}}{n^{\alpha_0/4}} \sum_{i=0}^{n-1} \frac{1}{\delta_0} \frac{h_{\alpha_0}}{f_{\alpha_0}}\Bigl(\frac{\sqrt{n}}{a(X_{i/n},\sigma_0)} \Delta _i^n \overline{X}, w_n(X_{i/n},\theta_0)\Bigr)\\
 \frac{\ln(n)^{\alpha_0/4}}{n^{\alpha_0/4}}\sum_{i=0}^{n-1} \frac{l_{\alpha_0}}{f_{\alpha_0}}\Bigl(\frac{\sqrt{n}}{a(X_{i/n},\sigma_0)} \Delta _i^n \overline{X}, w_n(X_{i/n},\theta_0)\Bigr)
\end{array}
\right),
\end{eqnarray*}
with $\Delta _i^n \overline{X}$ given by \eqref{E:euler}. It remains to study the stable convergence in law of  $(u_n(\theta_0)^T  \overline{G}_n(\theta_0))_n$. 
We set $u_n(\theta_0)^T  \overline{G}_n(\theta_0)= \sum_{i=0}^{n-1} \overline{\zeta}_i^n$ with
$\overline{\zeta}_{i}^{n} = (\overline{\zeta}_{i}^{k,n})_{1 \leq k \leq 3}$ given by
\begin{align}
   \overline{ \zeta}_{i}^{1,n} &= -\frac{1}{\sqrt{n}}  \frac{\partial_{\sigma} a(X_{i/n}, \sigma_0)}{a(X_{i/n},\sigma_0)} \frac{g_{\alpha_0}}{f_{\alpha_0}}\Bigl(\frac{\sqrt{n}}{a(X_{i/n},\sigma_0)} 
\Delta _i^n \overline{X}, w_n(X_{i/n},\theta_0)\Bigr),   \nonumber \\
  \overline{  \zeta}_{i}^{2,n} &= -\frac{\ln(n)^{\alpha_0/4}}{n^{\alpha_0/4}} \frac{1}{\delta_0} \frac{h_{\alpha_0}}{f_{\alpha_0}}\Bigl(\frac{\sqrt{n}}{a(X_{i/n},\sigma_0)} \Delta _i^n \overline{X}, w_n(X_{i/n},\theta_0)\Bigr),  \label{eq:zetabar} \\
  \overline{  \zeta}_{i}^{3,n} &= \frac{\ln(n)^{\alpha_0/4}}{n^{\alpha_0/4}} \frac{l_{\alpha_0}}{f_{\alpha_0}}\Bigl(\frac{\sqrt{n}}{a(X_{i/n},\sigma_0)} \Delta _i^n \overline{X}, w_n(X_{i/n},\theta_0)\Bigr).  \nonumber
\end{align}
We show that $(u_n(\theta_0)^T \overline{G}_n(\theta_0))_n$ stably converges in law with respect to the $\sigma$-field $\sigma(L^{\alpha_0}_s, s \leq 1)$ to $\overline{I}(\theta_0)^{1/2} \mathcal{N}$ where $\mathcal{N}$ is a standard Gaussian variable independent of $\overline{I}(\theta_0)$ by checking the following convergences (see Theorem IX.7.26 in Jacod and Shiryaev \cite{JacodShiryaev})
\begin{align}
    & \sum_{i=0}^{n - 1}  \E_{i} ( \overline{ \zeta}_{i}^{ n} ) \xrightarrow[]{\PP} 0,  \label{E:TCL1} \\
    & \sum_{i=0}^{n - 1}   \E_{i} (|\overline{ \zeta}_{i}^{k, n}|^{3}) \xrightarrow[]{\PP} 0 ,  \quad k \in \{1,2,3\}, \label{E:TCL2} \\
    & \sum_{i=0}^{n - 1}   \E_{i} (\overline{ \zeta}_{i}^{k, n} \overline{ \zeta}_{i}^{l, n}) \xrightarrow[]{\PP} \overline{I}_{kl}(\theta_0), 
      \quad k,l \in \{1,2,3\}, \label{E:TCL3} \\
    & \sum_{i=0}^{n - 1}  \E_{i}(\overline{\zeta}_{i}^{k, n}  \Delta_i^n M) \xrightarrow[]{\PP} 0,   \quad k \in \{1,2,3\},  \label{E:TCL4}
\end{align}
where the last convergence holds for any square integrable martingale $M$. \\

\noindent
{\bf Proof of \eqref{E:TCL1}, \eqref{E:TCL2} and \eqref{E:TCL3}.}
The first step is to apply  Proposition \ref{P:TV} (ii), to replace the locally stable distribution by the stable distribution, and next we conclude with the same arguments as in the proof of Proposition \ref{Th:TCL}. We check easily from \eqref{E:Bprime} that  the assumptions of Proposition \ref{P:TV} (ii) are satisfied by  the functions  $|\frac{g_{\alpha_0}}{f_{\alpha_0}}( \cdot,w_n(x,\theta_0))|^p |\frac{h_{\alpha_0}}{f_{\alpha_0}}( \cdot,w_n(x,\theta_0))|^q$, 
$|\frac{g_{\alpha_0}}{f_{\alpha_0}}( \cdot,w_n(x,\theta_0))|^p |\frac{l_{\alpha_0}}{f_{\alpha_0}}( \cdot,w_n(x,\theta_0))|^q$
and $|\frac{h_{\alpha_0}}{f_{\alpha_0}}( \cdot,w_n(x,\theta_0))|^p |\frac{l_{\alpha_0}}{f_{\alpha_0}}( \cdot,w_n(x,\theta_0))|^q$, for $p,q \geq 0$, and $x \in \R$. We only give the details for the convergence of $ \sum_{i=0}^{n - 1}  \E_{i} ( \overline{ \zeta}_{i}^{ 2,n} )$  in  \eqref{E:TCL1} and  $\sum_{i=0}^{n - 1}  \E_{i} ( \overline{ \zeta}_{i}^{1, n} )^2$ in \eqref{E:TCL3} as we obtain the other items following exactly the same scheme. For $x \in \R$, we consider the variable $Y_{1/n}$ (defined on another probability space) given  by
\begin{equation} \label{E:Y}
Y_{1/n}= a(x, \sigma_0) n^{-1/2} B_1 + \delta_0 c(x) n^{-1/ \alpha_0} S_1^{\alpha_0},
\end{equation}
where $B_1$ and $S_1^{\alpha_0}$ are independent variables, with respectively the standard Gaussian distribution and the stable distribution \eqref{E:stable}. Next we set
 \begin{equation} \label{E:Zh}
\Xi^{h,n}( x)=\E  \frac{h_{\alpha_0}}{f_{\alpha_0}}\Bigl(B_1+\frac{\delta_0 c(x)}{a(x, \sigma_0)}n^{1/2-1/ \alpha_0} S^{\alpha_0}_{1}, w_n(x,\theta_0)\Bigr)= 
\E  \frac{h_{\alpha_0}}{f_{\alpha_0}}\Bigl(\frac{\sqrt{n}}{a(x, \sigma_0)} Y_{1/n}, w_n(x,\theta_0)\Bigr).
\end{equation}
Then we deduce from Proposition \ref{P:TV} (ii) that   for $\epsilon>0$ arbitrarily small
$$
\sup_{0 \leq i \leq n-1}| \E_i  \frac{h_{\alpha_0}}{f_{\alpha_0}}\Bigl(\frac{\sqrt{n}}{a(X_{i/n},\sigma_0)} \Delta _i^n \overline{X}, w_n(X_{i/n},\theta_0)\Bigr)-\Xi^{h,n}( X_{i/n})| \leq C  (1/n^{1- \epsilon} \lor 1/n^{1/ \alpha_0-\epsilon}),
$$
and   it yields (observing in the case $\alpha>1$  that $\alpha/4+ 1/ \alpha >1$)
$$
| \sum_{i=0}^{n-1} \E_{i} ( \overline{ \zeta}_{i}^{ 2,n} ) + \sum_{i=0}^{n-1}\frac{\ln(n)^{\alpha_0/4}}{n^{\alpha_0/4}}  \frac{1}{\delta_0} \Xi^{h,n}(X_{i/n})| \xrightarrow[]{\PP}  0.
$$
By construction the variable $Y_{1/n}$ admits the density $p_{1/n}( \cdot, \beta(x, \theta_0))$  (given by \eqref{def pi} and \eqref{E:tildep}) and from the definition of the functions $h_{\alpha}, f_{\alpha}$ (see Section \ref{S:Section2})  we have
$$
- \frac{1}{\delta_0}  \frac{h_{\alpha_0}}{f_{\alpha_0}}\Bigl(\frac{\sqrt{n}}{a(x, \sigma_0)} y, w_n(x,\theta_0)\Bigr)= \frac{\partial_{\delta}  p_{1/n}}{p_{1/n}}(y, \beta(x, \theta_0)),
$$
so we deduce from \eqref{E:Zh}  that $\forall x$, $\Xi^{h,n}(x)=0$ (as in \eqref{Eq:point1cvTCL}) and we conclude   $\sum_{i=0}^{n-1} \E_{i} ( \overline{ \zeta}_{i}^{ 2,n} ) \to 0$.
Turning to $ \sum_{i=0}^{n-1} \E_{i} ( \overline{ \zeta}_{i}^{ 1,n} )^2$, we follow the same way and set
\begin{equation} \label{E:Zg}
\Xi^{g,n}( x)
=\E  \frac{g_{\alpha_0}}{f_{\alpha_0}}\Bigl(\frac{\sqrt{n}}{a(x, \sigma_0)} Y_{1/n}, w_n(x,\theta_0)\Bigr)^2.
\end{equation}
From Proposition \ref{P:TV} (ii) we have  for some $\epsilon>0$ arbitrarily small
$$
\sup_{0 \leq i \leq n-1}|\E_i\frac{g_{\alpha_0}}{f_{\alpha_0}}\Bigl(\frac{\sqrt{n}}{a(X_{i/n},\sigma_0)} 
\Delta _i^n \overline{X}, w_n(X_{i/n},\theta_0)\Bigr)-\Xi^{g,n}( X_{i/n})| \leq C (1/n^{1- \epsilon} \lor 1/n^{1/ \alpha_0-\epsilon}),
$$
and   it yields 
$$
| \sum_{i=0}^{n-1} \E_{i} ( \overline{ \zeta}_{i}^{ 1,n} )^2 - \frac{1}{n} \sum_{i=0}^{n-1} \left(\frac{\partial_{\sigma} a(X_{i/n}, \sigma_0)}{a(X_{i/n},\sigma_0)}\right)^2  \Xi^{g,n}(X_{i/n})| \xrightarrow[]{\PP}  0.
$$
It remains to prove that
\begin{equation} \label{E:conv2}
\frac{1}{n} \sum_{i=0}^{n-1} \left(\frac{\partial_{\sigma} a(X_{i/n}, \sigma_0)}{a(X_{i/n},\sigma_0)}\right)^2  \Xi^{g,n}(X_{i/n}) \xrightarrow[]{\PP} 2 \int_0^1 \left(\frac{\partial_{\sigma} a(X_{s}, \sigma_0)}{a(X_{s},\sigma_0)}\right)^2 ds.
\end{equation}
From \eqref{E:Zg} we can show for $x \in \R$
$$
\Xi^{g,n}( x) = \int \frac{g_{\alpha_0}^2}{f_{\alpha_0}}(y, w_n(x,\theta_0))dy, \quad \text{with} \quad \sup_{ |x| \leq K} w_n(x, \theta_0) \leq C n^{1/2-1/ \alpha_0},
$$
and  using \eqref{eq:intsigma} in the proof of Proposition \ref{Th:TCL} combined with $\int \frac{(\mathcal{D}(\phi)(y))^2}{\phi(y)}dy=2$, it yields
$$
\sup_{0 \leq i \leq n-1} | \Xi^{g,n}( X_{i/n}) - 2 | \xrightarrow[]{\PP} 0.
$$
We finally obtain \eqref{E:conv2} with the convergence of the Riemann sum 
$$
\frac{1}{n} \sum_{i=0}^{n-1} \left(\frac{\partial_{\sigma} a(X_{i/n}, \sigma_0)}{a(X_{i/n},\sigma_0)}\right)^2 \xrightarrow[]{\PP} \int_0^1 \left(\frac{\partial_{\sigma} a(X_{s}, \sigma_0)}{a(X_{s},\sigma_0)}\right)^2 ds.
$$
As previously mentioned, all terms in \eqref{E:TCL1}, \eqref{E:TCL2} and \eqref{E:TCL3} can be treated analogously  and we omit the details. \\

\noindent
{\bf Proof of  \eqref{E:TCL4}.} We only consider  $k=1$ and $k=3$, as the case $k=2$ is obtained with similar arguments. According to Theorem 2.2.15 in \cite{JacodProtter}, we also recall that we only need   to consider that   $M$ is either the Brownian Motion $B$ or  $M$ is orthogonal to $B$. 

{\bf Case $k=1$.} We have to prove 
\begin{equation} \label{E:orth1}
\frac{1}{\sqrt{n} }\sum_{i=0}^{n-1} \frac{\partial_{\sigma} a(X_{i/n}, \sigma_0)}{a(X_{i/n},\sigma_0)} \E_i\left(\frac{g_{\alpha_0}}{f_{\alpha_0}}\Bigl(\frac{\sqrt{n}}{a(X_{i/n},\sigma_0)} 
\Delta _i^n \overline{X}, w_n(X_{i/n},\theta_0)\Bigr) \Delta_i^n M \right) \xrightarrow[]{\PP} 0.
\end{equation}
To shorten the notation, we set
\begin{equation} \label{E:Un}
\Upsilon_n(y)= \frac{g_{\alpha_0}}{f_{\alpha_0}}\Bigl(y, w_n(X_{i/n},\theta_0)\Bigr),
\end{equation}
and as observed previously, we have  $\sup_y( |\Upsilon_n(y)|+ |\Upsilon'_n(y)|) \leq C \ln(n)$. We also recall that
$$
\frac{\sqrt{n}}{a(X_{i/n},\sigma_0)} \Delta _i^n \overline{X}= \sqrt{n}  \Delta _i^n B + \sqrt{n} \frac{ \delta_0 c(X_{i/n})}{a(X_{i/n}, \sigma_0)}  \Delta _i^n L^{\alpha_0}.
$$
So it yields for $0<p<\alpha_0$, using  the bound for $\Upsilon_n$ and a first order Taylor expansion for the first inequality and Theorem 2 in \cite{LuschgyPages} for the second one
$$
\E_{i} | \Upsilon_n(\frac{\sqrt{n}}{a(X_{i/n},\sigma_0)} 
\Delta _i^n \overline{X})-  \Upsilon_n(\sqrt{n}  \Delta _i^n B)|^2 \leq C \ln(n)^2 n^{p/2} \E | \Delta _i^n L^{\alpha_0}|^p \leq C \ln(n)^2 n^{p/2-p/ \alpha_0}.
$$
Choosing $p$ close to $\alpha_0$ we deduce for $\epsilon >0$ (small enough to ensure $1- \alpha_0/2- \epsilon>0$)
$$
\E_{i} | \Upsilon_n(\frac{\sqrt{n}}{a(X_{i/n},\sigma_0)} 
\Delta _i^n \overline{X})-  \Upsilon_n(\sqrt{n}  \Delta _i^n B)|^2 \leq C /n^{1- \alpha_0/2- \epsilon}.
$$
Setting
$$
S_n^1=\frac{1}{ \sqrt{n}} \sum_{i=0}^{n - 1} \frac{\partial_{\sigma} a(X_{i/n}, \sigma_0)}{a(X_{i/n},\sigma_0)}  \E_{i} [( \Upsilon_n( \frac{\sqrt{n}}{a(X_{i/n},\sigma_0)} 
\Delta _i^n \overline{X})
- \Upsilon_n( \sqrt{n}  \Delta _i^n B) ) \Delta_{i}^{n}M ],
$$
we deduce  from Cauchy-Schwarz inequality and the previous bound that
\begin{align*}
\E|S_n^{1}| \leq \frac{C }{\sqrt{n}} \left(\sum_{i=0}^{n - 1} \frac{1}{n^{1- \alpha_0/2 - \epsilon}}
 \sum_{i=0}^{n - 1} \E (\Delta_{i}^{n}M)^2 \right)^{1/ 2}
  \to 0,
\end{align*}
observing that the second sum is bounded since $M$ is a square-integrable martingale.  
Thus the proof of \eqref{E:orth1} reduces to show $S_n^2 \to 0$ where
$$
S_n^2=\frac{1}{ \sqrt{n}} \sum_{i=0}^{n - 1}\frac{\partial_{\sigma} a(X_{i/n}, \sigma_0)}{a(X_{i/n},\sigma_0)}  \E_{i} [ \Upsilon_n( \sqrt{n}  \Delta _i^n B)  
 \Delta_{i}^{n}M ].
$$
If $M=B$ and since  $\Upsilon_n$ is even, we check that $\E_{i}  [ \Upsilon_n( \sqrt{n}  \Delta _i^n B)  
 \Delta_{i}^{n}M ]=0$. If $M$ is a square integrable martingale orthogonal to $B$,  we have from the representation Theorem (Theorem III 4.34 in \cite{JacodShiryaev}) 
 $$
 M_t =\int_0^t \int G_s \tilde{N}^{\alpha_0, \tau}(ds,dz) \; \text{with} \; \E \int_0^1 \int G_s^2 \overline{N}^{\alpha_0, \tau}(dz,ds) < \infty,
$$
where $\tilde{N}^{\alpha_0, \tau}$ and $\overline{N}^{\alpha_0, \tau}$ are respectively the compensated Poisson measure and the compensator associated to $(L_t^{\alpha_0})_{t \geq 0}$ and 
 $(G_t)$  is a predictable process.
 We also have the representation with $(h_t^n)$ predictable
 $$
\Upsilon_n( \sqrt{n}  \Delta _i^n B)=\E_i \Upsilon_n( \sqrt{n}  \Delta _i^n B)+\int_{i /n}^{(i+1)/n} h_s^n dB_s.
 $$
We conclude with It\^{o}'s formula 
that $\E_{i} [ \Upsilon_n( \sqrt{n}  \Delta _i^n B) 
 \Delta_{i}^{n}M ]=0$ and the proof of \eqref{E:orth1} is finished.

{\bf Case $k=3$.} We now show 
\begin{equation} \label{E:orth3}
\frac{\ln(n)^{\alpha_0/4}}{n^{\alpha_0/4} }\sum_{i=0}^{n-1} \E_i\left(\frac{l_{\alpha_0}}{f_{\alpha_0}}\Bigl(\frac{\sqrt{n}}{a(X_{i/n},\sigma_0)} 
\Delta _i^n \overline{X}, w_n(X_{i/n},\theta_0)\Bigr) \Delta_i^n M \right) \xrightarrow[]{\PP} 0.
\end{equation}
 We first assume that $M=B$. From H\"{o}lder's inequality with $q>1$, we have
 \begin{align*}
 \E_{i} &  \left| \frac{l_{\alpha_0}}{f_{\alpha_0}}  (\frac{\sqrt{n}}{a(X_{i/n},\sigma_0)} 
\Delta _i^n \overline{X},w_n(X_{i/n},\theta_0) ) 
 \Delta_{i}^{n}B \right|  \leq  \\
&  C\frac{1}{\sqrt{n}}\left(\E_{i}\left| \frac{l_{\alpha_0}}{f_{\alpha_0}}(\frac{\sqrt{n}}{a(X_{i/n},\sigma_0)}
\Delta _i^n \overline{X},w_n(X_{i/n},\theta_0))\right|^q\right)^{1/q}.
  \end{align*}
 Proceeding as in the proof of \eqref{E:TCL1}, we replace $\Delta _i^n \overline{X}$ by the variable $Y_{1/n}$ given by \eqref{E:Y} and we obtain from Proposition \ref{P:TV} (for $\epsilon>0$ close to $0$)
 $$
 \sup_{0 \leq i \leq n-1}
\left|\E_{i}|\frac{l_{\alpha_0}}{f_{\alpha_0}}(\frac{\sqrt{n}}{a(X_{i/n},\sigma_0)} 
\Delta _i^n \overline{X}, w_n(X_{i/n},\theta_0))|^q - \Xi^{l,n}(X_{i/n})\right|^{1/q} \leq C(1/n^{1- \epsilon} \lor 1/n^{1/ \alpha_0-\epsilon})^{1/q},
 $$
 where
 \begin{equation} \label{E:Zl}
 \Xi^{l,n}(x)=\E \left|\frac{l_{\alpha_0}}{f_{\alpha_0}}\Bigl(\frac{\sqrt{n}}{a(x,\sigma_0)} Y_{1/n}, w_n(x,\theta_0)\Bigr)\right|^q.
 \end{equation}
 For $1<q<2$, we see that $(1/n^{1- \epsilon} \lor 1/n^{1/ \alpha_0-\epsilon})^{1/q}\sqrt{n} \ln(n)^{\alpha_0/4}/n^{\alpha_0/4} \to 0$ and the proof of \eqref{E:orth3} reduces to show
 \begin{equation*} 
\sqrt{n}\frac{\ln(n)^{\alpha_0/4}}{n^{\alpha_0/4} } \sup_i |\Xi^{l,n}(X_{i/n})|^{1/q}\xrightarrow[]{\PP} 0.
\end{equation*}
 But since $Y_{1/n}$ admits the density $p_{1/n}( \cdot, \beta(x, \theta_0))$  (given by \eqref{def pi} and \eqref{E:tildep}) and from the definition of the functions $l_{\alpha}, f_{\alpha}$ (see Section \ref{S:Section2})  we have using the bound \eqref{eq:majfghk} in Lemma \ref{l:approxD}
$$
\Xi^{l,n}(x)=\int  \frac{|l_{\alpha_0}|^q}{|f_{\alpha_0}|^{q-1}}(y, w_n(x,\theta_0)dy \leq C \ln(n)^p \int (|k_{\alpha_0}(y)| + |h_{\alpha_0}(y)|)dy \leq C \ln(n)^p /n^{1- \alpha_0/2}.
$$
With $1<q<2$, we have $\alpha_0/4+1/q-\alpha_0/(2q)-1/2>0$, this finishes the proof of \eqref{E:orth3} if $M=B$.
 
 We now assume that $M$ is orthogonal to $B$. As previously,  we apply Theorem III 4.34 in \cite{JacodShiryaev} to write for some predictable process $(G_t)$
 $$
 M_t =\int_0^t \int G_s \tilde{N}^{\alpha_0, \tau}(dz,ds) \; \text{with} \; \E \int_0^1 \int G_s^2 \overline{N}^{\alpha_0, \tau}(dz,ds) < \infty,
 $$
 and also for $(h_t^n)$ and $(H_t^n)$ predictable 
 \begin{align*}
\frac{l_{\alpha_0}}{f_{\alpha_0}}\Bigl(\frac{\sqrt{n}}{a(X_{i/n},\sigma_0)} 
\Delta _i^n \overline{X}, w_n(X_{i/n},\theta_0)\Bigr)= &
\E_i \frac{l_{\alpha_0}}{f_{\alpha_0}}\Bigl(\frac{\sqrt{n}}{a(X_{i/n},\sigma_0)} 
\Delta _i^n \overline{X}, w_n(X_{i/n},\theta_0)\Bigr) \\
& +\int_{i /n}^{(i+1) /n} h_s^n dB_s + \int_{i /n}^{(i+1) /n} \int H_s^n \tilde{N}^{\alpha_0, \tau}(dz,ds).
 \end{align*}
 We then deduce from It\^{o}'s Formula for the product of stochastic integrals
 $$
  \E_i\left(\frac{l_{\alpha_0}}{f_{\alpha_0}}\Bigl(\frac{\sqrt{n}}{a(X_{i/n},\sigma_0)} 
\Delta _i^n \overline{X}, w_n(X_{i/n},\theta_0)\Bigr) \Delta_i^n M \right)= \E_{i}  \int_{i /n}^{(i+1) /n}  \int H_s^n G_s  \overline{N}^{\alpha_0, \tau}(dz,ds).
 $$
 With the decomposition $G_s =G_s{\bf 1}_{| G_s| > \epsilon_n}  +G_s {\bf 1}_{| G_s| \leq \epsilon_n} $, the proof reduces to study  the convergence of
 \begin{align*}
 T^1_n=& \frac{\ln(n)^{\alpha_0/4}}{n^{\alpha_0/4} }\ \sum_{i=0}^{n - 1} \E_{i}  \int_{i /n}^{(i+1) /n}  \int H_s^n G_s {\bf 1}_{| G_s| > \epsilon_n}  \overline{N}^{\alpha_0, \tau}(dz,ds), \\
  T^2_n=& \frac{\ln(n)^{\alpha_0/4}}{n^{\alpha_0/4} }\sum_{i=0}^{n - 1} \E_{i}  \int_{i /n}^{(i+1) /n}  \int H_s^n G_s {\bf 1}_{| G_s| \leq \epsilon_n}  \overline{N}^{\alpha_0, \tau}(dz,ds).
 \end{align*}
From the  bound for $(l_{\alpha_0}/f_{\alpha_0})^{\prime}$ and using that the jumps of $(L_t^{\alpha_0})$ are bounded, we have $\forall t \in [0,1]$, $|H_t^n| \leq C \ln(n)^p$ for some $p>0$. It yields
 \begin{align} \label{eq:Tn1}
\E  |T^1_n| &\leq   C \ln(n)^p \frac{\ln(n)^{\alpha_0/4}}{n^{\alpha_0/4} }\sum_{i=0}^{n - 1} \E  \int_{i /n}^{(i+1) /n}  \int \frac{G^2_s}{\epsilon_n}  \overline{N}^{\alpha_0, \tau}(dz,ds) \nonumber \\
& \leq C \frac{\ln(n)^{\alpha_0/4+p}}{\epsilon_n n^{\alpha_0/4} } \E \int_0^1 \int G_s^2  \overline{N}(dz,ds) \leq  C \frac{\ln(n)^{\alpha_0/4+p}}{\epsilon_n n^{\alpha_0/4} }.
\end{align}
With $\epsilon_n= n^{-\alpha_0/8}$, we  deduce $\E |T_n^1 |\to 0$.
 
Turning to  $T_n^2$,  we first observe that
$$
\E_{i}  \int_{i /n}^{(i+1) /n}  \int (H^n_s)^2  \overline{N}^{\alpha_0, \tau}(dz,ds) \leq      \E_i \frac{l_{\alpha_0}^2}{f_{\alpha_0}^2}\Bigl(\frac{\sqrt{n}}{a(X_{i/n},\sigma_0)} 
\Delta _i^n \overline{X}, w_n(X_{i/n},\theta_0) \Bigr).
$$
As previously we replace $\Delta _i^n \overline{X}$ by $Y_{1/n}$ introducing $\Xi^{l,n}$ defined by \eqref{E:Zl} with $q=2$, and we deduce  from the preceding inequality using successively Proposition \ref{P:TV} and Proposition \ref{L:tout}
$$
\sup_i \E_{i}  \int_{i /n}^{(i+1) /n}  \int (H^n_s)^2  \overline{N}^{\alpha_0, \tau}(dz,ds) \leq 
C [(1/n^{1- \epsilon} \lor 1/n^{1/ \alpha_0-\epsilon})+ 1/(\ln(n)^{\alpha_0/2}n^{1- \alpha_0/2})].
$$
 From Cauchy-Schwarz inequality it yields
 \begin{align*}
|\E_{i} & \int_{i /n}^{(i+1) /n}  \int H_s^n G_s {\bf 1}_{| G_s| \leq \epsilon_n}  \overline{N}^{\alpha_0, \tau}(dz,ds)| \leq 
 \sqrt{C_n}
\left( \E_{i}  \int_{i /n}^{(i+1) /n}  \int G^2_s {\bf 1}_{| G_s| \leq \epsilon_n}  \overline{N}^{\alpha_0, \tau}(dz,ds)\right)^{1/2}, 
\end{align*}
 with $C_n=C [(1/n^{1- \epsilon} \lor 1/n^{1/ \alpha_0-\epsilon})+ 1/(\ln(n)^{\alpha_0/2}n^{1- \alpha_0/2})]$. Applying once again Cauchy-Schwarz inequality
\begin{align*}
\E |T_n^2|  \leq  
  \sqrt{nC_n} \frac{\ln(n)^{\alpha_0/4}}{n^{\alpha_0/4} }
\left( \E \sum_{i=0}^{n - 1} \int_{i /n}^{(i+1)/n}  \int G^2_s {\bf 1}_{| G_s| \leq \epsilon_n}  \overline{N}^{\alpha_0, \tau}(dz,ds) \right)^{1/ 2}\\
 \leq C\left( \E \int_0^1  \int G^2_s {\bf 1}_{| G_s| \leq \epsilon_n}  \overline{N}^{\alpha_0, \tau}(dz,ds)\right)^{1/2}.
\end{align*}
As $\epsilon_n$ goes to zero, we conclude $\E |T_n^2|  \to 0$ by dominated convergence and the proof of \eqref{E:orth3} is finished.

\noindent
2. \underline{Uniform convergence of $u_n^T(\theta_0) \tilde{J}_n(\theta) u_n(\theta_0)$.} 

We recall that $ \tilde{J}_n(\theta)= \nabla_{\theta}  \tilde{G}_n(\theta)$ where
$$
 \tilde{G}_n(\theta)=\sum_{i=0}^{n-1} \nabla_{\theta} \ln p_{1/n} (\Delta_i^n X, \beta(X_{i/n}, \theta)).
$$
Furthermore since $\beta(x, \theta)= (a(x, \sigma), \delta c(x), \alpha)$, we have  for $a,b \in \{\sigma, \delta, \alpha\}$ 
\begin{align*}
&\partial^2_{\sigma \sigma} [ \ln p_{1/n} (y, \beta(x, \theta))]= [\partial^2_{\sigma \sigma } \ln p_{1/n}] (y, \beta(x, \theta)) (\partial_{\sigma} \beta(x, \theta))^2  + [\partial_{ \sigma} \ln p_{1/n}] (y, \beta(x, \theta))  \partial^2_{\sigma \sigma} \beta(x, \theta),  \\
& \partial^2_{a b} [ \ln p_{1/n} (y, \beta(x, \theta))]=[\partial^2_{ab } \ln p_{1/n}] (y, \beta(x, \theta)) \partial_{a} \beta(x, \theta) \partial_{b} \beta(x, \theta) \quad \text{if} \quad (a,b) \neq (\sigma, \sigma).
\end{align*}
So  we can write $u_n(\theta_0)^T \tilde{J_n}(\theta) u_n(\theta_0) = \sum_{i=0}^{n-1} \tilde{\xi}_{i}^n(\theta)$
with
\begin{align*}
    \tilde{\xi}_i^{11,n}(\theta) =& \frac{1}{n} \left( \partial_{\sigma} a(X_{i/n}, \sigma)^2 \partial^2_{\sigma \sigma} \ln p_{1/n} + \partial^2_{\sigma \sigma} a(X_{i/n}, \sigma) \partial_{ \sigma} \ln p_{1/n} \right) (\Delta _i^n X, \beta(X_{i/n},\theta)), \\
    \tilde{\xi}_i^{12,n} (\theta)=& \frac{\ln(n)^{\alpha_0/4}}{\sqrt{n}n^{\alpha_0/4}}\partial_{\sigma} a(X_{i/n}, \sigma) c(X_{i/n}) \partial^2_{\sigma \delta}\ln p_{1/n}(\Delta _i^n X, \beta(X_{i/n},\theta)),\\
    \tilde{\xi}_i^{13,n}(\theta) =& \frac{\ln(n)^{\alpha_0/4}}{\sqrt{n}n^{\alpha_0/4}} \partial_{\sigma} a(X_{i/n}, \sigma)\left(-\frac{\delta_0 c(X_{i/n})}{2 \alpha_0}\bigl[\ln(n) 
   - \ln(\ln n)\bigr]\partial^2_{\sigma \delta} \ln p_{1/n}  \right. \\
 \quad & \quad  \left.  + \partial^2_{\sigma \alpha}
    \ln  p_{1/n} \right)(\Delta _i^n X, \beta(X_{i/n},\theta)), \\
    \tilde{\xi}_i^{22,n} (\theta)=& \frac{\ln(n)^{\alpha_0/2}}{n^{\alpha_0/2}} c(X_{i/n})^2\partial^2_{\delta \delta} \ln p_{1/n} (\Delta _i^n X, \beta(X_{i/n},\theta)), \\
    \tilde{\xi}_i^{23,n}(\theta) = & \frac{\ln(n)^{\alpha_0/2}}{n^{\alpha_0/2}} c(X_{i/n})\left(-\frac{\delta_0 c(X_{i/n})}{2 \alpha_0}\bigl[\ln(n) - \ln(\ln n)\bigr]\partial^2_{\delta \delta} \ln p_{1/n} \right. \\
    \quad & \quad  \left. + 
    \partial^2_{\delta \alpha}  \ln p_{1/n} \right)(\Delta _i^n X, \beta(X_{i/n},\theta)),\\
  \tilde{\xi}_i^{33,n} (\theta)= & \frac{\ln(n)^{\alpha_0/2}}{n^{\alpha_0/2}} \left( \frac{\delta_0^2 c(X_{i/n})^2}{(2\alpha_0)^2}\bigl[\ln(n) - \ln(\ln n)\bigr]^2 \partial^2_{\delta \delta} \ln p_{1/n} \right. \\
    \quad &\left. - \frac{\delta_0 c(X_{i/n})}{\alpha_0}\bigl[\ln(n) - \ln(\ln n)\bigr]\partial^2_{\alpha \delta} \ln  p_{1/n} + \partial^2_{\alpha \alpha} \ln p_{1/n}  \right)
    (\Delta _i^n X, \beta(X_{i/n},\theta)).
    \end{align*}
To prove the result, it is sufficient to check
    \begin{equation} \label{E:lln}
    u_n(\theta_0)^T \tilde{J}_n(\theta_0) u_n(\theta_0) \xrightarrow[]{\PP} -\overline{I}(\theta_0),
\end{equation}
and $\forall r>0$
    \begin{equation} \label{E:ulln} 
    \sup_{\theta \in V_n^{(r)} } ||u_n(\theta_0)^T (\tilde{J}_n(\theta) - J_n(\theta_0)) u_n(\theta_0)|| \xrightarrow[]{\PP} 0. 
\end{equation}

\noindent
{\bf Proof of \eqref{E:lln}.}
We just give the sketch of the proof. As previously, the first step is to replace the increments $\Delta_i^n X$ by the Euler approximation without drift $\Delta_i^n \overline{X}$.
We set $ \overline{J}_n(\theta)= \nabla_{\theta}  \overline{G}_n(\theta)$ with
$$
 \overline{G}_n(\theta)=\sum_{i=0}^{n-1} \nabla_{\theta} \ln p_{1/n} (\Delta_i^n \overline{X}, \beta(X_{i/n}, \theta)).
$$
So we have $u_n(\theta_0)^T \overline{J}_n(\theta_0) u_n(\theta_0) = \sum_{i=0}^{n-1} \overline{\xi}_{i}^n(\theta_0)$, where the variables $(\overline{\xi}_{i}^n(\theta))$ are similar to the variables $(\tilde{\xi}_i^n(\theta))$ replacing 
$\Delta_i^n X$ by $\Delta_i^n \overline{X}$. As explained in Section \ref{S:Section2}, 
 $\forall a,b \in \{\sigma, \delta, \alpha\}$, we can write   $\partial^2_{ab} \ln p_{1/n}( y,\beta(x,\theta)) $  as a finite sum of terms of the form
 $$
 c_{k,l,m}(x, \theta) \ln(n)^p \frac{f_{\alpha}^{(k,l,m)}}{f_{\alpha}}(\frac{\sqrt{n}}{a(x, \sigma)} y, w_n(x, \theta)), \quad k,l,m,p \geq 0,
 $$
 where the functions $f^{(k,l,m)}$ are defined by \eqref{def:fklm} and the coefficients $c_{k,l,m}(x, \theta)$ are bounded for $|x| \leq K$ and $\theta$ in a compact subset of $\R\times (0, + \infty) \times (0,2)$. Some bounds for $f^{(k,l,m)}$  are given in Lemma \ref{L:MajFunc}. With similar arguments that we not detail here,  we can extend these results to $ (f^{(k,l,m)})^{\prime}$ and $ (f^{(k,l,m)})^{\prime \prime}$ to see that
   the assumptions of Proposition \ref{P:euler} are satisfied by the functions $y \mapsto \partial^2_{ab} \ln p_{1/n}( y,\beta(x,\theta)) $. Then from Proposition \ref{P:euler}, it is easy to verify the conditions
 \eqref{E:i} to obtain
$$
u_n(\theta_0)^T  \tilde{J}_n(\theta_0)u_n( \theta_0)-u_n(\theta_0)^T  \overline{J}_n(\theta_0) u_n(\theta_0) \xrightarrow[]{\PP} 0.
$$ 
Thus the proof will be finished if we show 
 \begin{equation} \label{E:llneuler}
    u_n(\theta_0)^T \overline{J}_n(\theta_0) u_n(\theta_0) = \sum_{i=0}^{n-1} \overline{\xi}_i^n(\theta_0) \xrightarrow[]{\PP} -\overline{I}(\theta_0).
\end{equation}
This is done by checking once again the two conditions  \eqref{E:i}. From Lemma \ref{L:der2der3}, we have for $a,b \in \{\sigma, \delta, \alpha\}$, 
$$
\forall y, \; \forall |x| \leq K, \; | \partial^2_{ab} \ln p_{1/n}( y,\beta(x,\theta_0))| + | \partial^2_{ab} \ln p_{1/n}( y,\beta(x,\theta_0))|^2 \leq C \ln(n)^p (1 + \ln(1 + |y|)^q), 
$$
for some $p,q>0$. Then it yields from Proposition \ref{P:TV} (with $Y_{1/n}$ defined by \eqref{E:Y}) for $|x| \leq K$
\begin{align}
& |\E_i \partial^2_{ab} \ln p_{1/n}( \Delta_i^n\overline{X},\beta(x,\theta_0)) -\E \partial^2_{ab} \ln p_{1/n}(Y_{1/n},\beta(x,\theta_0))| \leq C (1/n^{1- \epsilon} \lor 1/n^{1/ \alpha_0-\epsilon}), \label{E:Z1}\\
& |\E_i |\partial^2_{ab} \ln p_{1/n}( \Delta_i^n\overline{X},\beta(x,\theta_0))|^2 -\E |\partial^2_{ab} \ln p_{1/n}(Y_{1/n},\beta(x,\theta_0))|^2| \leq C (1/n^{1- \epsilon} \lor 1/n^{1/ \alpha_0-\epsilon}).  \label{E:Z2}
\end{align}
Setting $\Xi^{a,b,n}(x)= \E \partial^2_{ab} \ln p_{1/n}(Y_{1/n},\beta(x,\theta_0))$, we deduce from \eqref{E:Z1} that we can replace, in the expression of $\E_i \overline{\xi}_i^n(\theta_0)$, the conditional expectation 
$\E_i \partial^2_{ab} \ln p_{1/n}(\Delta_i^n \overline{X},\beta(X_{i/n},\theta_0))$ by $\Xi^{a,b,n}(X_{i/n})$. More precisely we have  the convergence (that we only write for $\overline{\xi}_i^{22,n}(\theta_0)$ as similar results hold for the other terms) 
$$
|\sum_{i=0}^{n-1} \E_i \overline{\xi}_i^{22,n}(\theta_0) -\frac{\ln(n)^{\alpha_0/2}}{n^{\alpha_0/2}} \sum_{i=0}^{n-1}c(X_{i/n})^2 \Xi^{\delta, \delta,n}(X_{i/n})| \xrightarrow[]{\PP} 0.
$$ 
Proceeding as in the proof of Proposition \ref{Th:CvIn} \eqref{E:CvIn} (proof of \eqref{eq:Espi})
 it yields
$$
\frac{\ln(n)^{\alpha_0/2}}{n^{\alpha_0/2}} \sum_{i=0}^{n-1}c(X_{i/n})^2 \Xi^{\delta, \delta,n}(X_{i/n}) \xrightarrow[]{\PP} -\overline{I}_{22}(\theta_0).
$$
In a similar way, we obtain $\sum_{i=0}^{n-1} \E_i \overline{\xi}_i^n(\theta_0)^2 \rightarrow 0$  using \eqref{E:Z2} to replace $\Delta_i^n \overline{X}$ by $Y_{1/n}$ in the expression of 
$ \E_i \overline{\xi}_i^n(\theta_0)^2$ and next concluding as in the proof of Proposition \ref{Th:CvIn} (proof of \eqref{eq:Espi2}). This achieves the proof of  \eqref{E:llneuler}.

\noindent
{\bf Proof of \eqref{E:ulln}.} 
With
 $V_n^{(r)}$defined by \eqref{def:V} we recall that we have $\forall \theta \in V_n^{(r)},$
\begin{align*}
    |\sigma - \sigma_0|\leq \frac{C}{\sqrt{n}}, \quad 
    |\delta - \delta_0|\leq \frac{C (\ln(n))^p}{n^{\alpha_0/4}}, \quad
    |\alpha - \alpha_0|\leq \frac{C (\ln(n))^p}{n^{\alpha_0/4}}.
\end{align*}
Considering the variables $(\tilde{\xi}_i^n(\theta))$ introduced previously, we have to prove for $1 \leq k,l \leq 3$
\begin{equation} \label{E:Ukl}
  \sup_{\theta \in V_n^{(r)} } | \sum_{i=0}^{n-1}( \tilde{\xi}_i^{kl,n}(\theta)-  \tilde{\xi}_i^{kl,n}(\theta_0)) | \xrightarrow[]{\PP} 0.  
\end{equation}
As in the proof of Proposition \ref{Th:CvIn} \eqref{E:CvUnif},  we bound $| \tilde{\xi}_i^{k,l,n}(\theta)-  \tilde{\xi}_i^{kl,n}(\theta_0)|$ using a first order expansion and Lemma \ref{L:der2der3} \eqref{eq:boundder3}. It yields for some $p>0$ (using that $1/\sqrt{n} \leq 1/n^{\alpha_0/4}$)
$$
 \sup_{\theta \in V_n^{(r)} } | \tilde{\xi}_i^{11,n}(\theta)-  \tilde{\xi}_i^{11,n}(\theta_0)| \leq \frac{C}{n}  \frac{\ln(n)^p}{n^{\alpha_0/4} }(1+ \ln (1+ |\Delta_i^n X|)^p),
$$
and taking the expectation we obtain \eqref{E:Ukl} for $(k,l)=(1,1)$. For $(k,l) \neq (1,1)$, we still have the bound \eqref{E:Bw}  and \eqref{E:Bphi} (replacing $\sigma$ by $a(x, \sigma)$ and $\delta$ by $\delta c(x)$) and we obtain $\forall \varepsilon, \eta>0$, $\exists C>0$, $\exists p>0$ such that for $n$ large enough
\begin{equation} \label{E:Bgn}
 \sup_{\theta \in V_n^{(r)} } | \tilde{\xi}_i^{kl,n}(\theta)-  \tilde{\xi}_i^{kl,n}(\theta_0)|  \leq  
  \frac{C}{n^{\alpha_0/2}}\frac{\ln(n)^p}{n^{\alpha_0/4}} g_n(\frac{\sqrt{n}}{a(X_{i/n}, \sigma_0)}\Delta _i^n X, X_{i/n})
\end{equation}
with
$$  
g_n(y,x)=  \frac{w_n(x,\theta_0)^{\alpha_0}\psi_{\alpha_0-\varepsilon}^{(3)} (y)}{\phi((1+\eta) y)+w_n(x,\theta_0)^{\alpha_0}\psi_{\alpha_0+\varepsilon}^{(0)} ( y)},
$$
where the functions $\psi_{\alpha}^{(p)}$ are defined in \eqref{eq:fp}. We check that for all $|x| \leq K$, $y \mapsto g_n(y,x)$ is $\mathcal{C}^2$, even, and satisfies for some $p>0$, $\forall y, \; \forall |x| \leq K, \;$
$$
 0 \leq g_n(y,x) \leq C(1+ \ln(1+ |y|)^3)(1+ |y|^{2 \varepsilon}), \quad |g^{\prime}_n(y,x)| + |g^{\prime \prime}_n(y,x)| \leq C \ln (n)^p.
$$
The bound for $g_n$ is immediate. The bounds for the first and second order derivatives require a little more work. We just explain the main arguments to bound $g^{\prime}_n$. An explicit calculus gives
$
g^{\prime}_n(y,x)= D_n^1(y,x) - D_n^2(y,x)
$
with
$$
 D_n^1(y,x) =\frac{w_n(x,\theta_0)^{\alpha_0}\psi_{\alpha_0-\varepsilon}^{(3) \prime} (y)}{\phi((1+\eta) y)+w_n(x,\theta_0)^{\alpha_0}\psi_{\alpha_0+\varepsilon}^{(0)} ( y)} ,
 $$
 $$
D_n^2(y,x)  =g_n(y,x) \frac{[(1+ \eta)\phi^{\prime}((1+\eta) y)+w_n(x,\theta_0)^{\alpha_0}\psi_{\alpha_0+\varepsilon}^{(0) \prime} ( y) ] }
 {\phi((1+\eta) y)+w_n(x,\theta_0)^{\alpha_0}\psi_{\alpha_0+\varepsilon}^{(0) } ( y)}.
$$
We check easily that $\sup_{|x| \leq K} |D_n^1(y,x) | \leq C$. For the second term, using that
for $y>0$ large enough $y \mapsto y\phi((1+\eta) y)/ \psi_{\alpha_0+1+\varepsilon}^{(0) } ( y) ) $ is decreasing, we introduce  
 $y^*_n(x)$ the unique solution of $y\phi((1+\eta) y)=w_n(x,\theta_0)^{\alpha_0}\psi_{\alpha_0+1+\varepsilon}^{(0) } ( y)$ (that exists since $w_n(x,\theta_0)^{\alpha_0} \rightarrow 0$). We can check that $y^*_n(x)\underset{n\to + \infty}{\sim} C \sqrt{\ln n}$ (where $C$ does not depend on $x$). By definition of $y^*_n(x)$, we have
 $$
 |y|\phi((1+\eta) y) \leq w_n(x,\theta_0)^{\alpha_0}\psi_{\alpha_0+1+\varepsilon}^{(0) } ( y) \quad \text{if} \quad |y| \geq y^*_n(x).
 $$
Consequently we obtain 
$$
|D_n^2(y,x) | \leq \left\{
 \begin{array}{l} 
C g_n(y,x) (|y| +1)  \quad \text{if} \quad |y| \leq  y^*_n(x), \\
C g_n(y,x) / |y| \quad \text{if} \quad |y| \geq y^*_n(x),
\end{array} \right.
$$
and  finally, using the bound for $g_n$, it yields $\forall y$, $\sup_{|x| \leq K} |D_n^2(y,x) | \leq C \ln(n)^p$ for $p>0$.

Coming back to the proof of \eqref{E:Ukl}, we have from \eqref{E:Bgn}
$$
\E \sup_{\theta \in V_n^{(r)} } | \sum_{i=0}^{n-1}( \tilde{\xi}_i^{kl,n}(\theta)-  \tilde{\xi}_i^{kl,n}(\theta_0)) | \leq  \frac{C}{n^{\alpha_0/2}}\frac{\ln(n)^p}{n^{\alpha_0/4}} \E \sum_{i=0}^{n-1} \E_i g_n(\frac{\sqrt{n}}{a(X_{i/n}, \sigma_0)}\Delta _i^n X, X_{i/n}).
 $$
From Proposition \ref{P:euler},  we deduce that
$$
\frac{1}{n^{\alpha_0/2}}\frac{\ln(n)^p}{n^{\alpha_0/4}} \E \sum_{i=0}^{n-1}| \E_i g_n(\frac{\sqrt{n}}{a(X_{i/n}, \sigma_0)}\Delta _i^n X, X_{i/n})-
\E_i g_n(\frac{\sqrt{n}}{a(X_{i/n}, \sigma_0)}\Delta _i^n \overline{X}, X_{i/n})| \xrightarrow[]{} 0,
$$
and from Proposition \ref{P:TV},
$$
\frac{1}{n^{\alpha_0/2}}\frac{\ln(n)^p}{n^{\alpha_0/4}} \E \sum_{i=0}^{n-1}| \E_i g_n(\frac{\sqrt{n}}{a(X_{i/n}, \sigma_0)}\Delta _i^n \overline{X}, X_{i/n})
-\Xi^n(X_{i/n})| \xrightarrow[]{} 0,
$$
where $\Xi^n(x)= \E g_n(\frac{\sqrt{n}}{a(x, \sigma_0)}Y_{1/n}, x)$. But as done at the end of the proof of \eqref{E:CvUnif}, we can show with $\varepsilon$ and $\eta$ small enough
$$
n \frac{1}{n^{\alpha_0/2}}\frac{\ln(n)^p}{n^{\alpha_0/4}} \sup_{|x| \leq K} \Xi^n(x) \xrightarrow[]{} 0,
$$
and the proof of \eqref{E:Ukl} is finished.
\subsubsection{Proof of Proposition \ref{P:TV}}

(i) We first observe that $n^{1/\alpha} L_{1/n}^{\alpha}$ has the distribution of the variable $L^{n, \alpha}_1$ with characteristic function
$$
e^{iu L^{n, \alpha}_1}=e^{ \int(e^{iuz}-1-z1_{\{ |z| \leq 1\} }) \frac{c_{\alpha} \tau(z/n^{1/\alpha})}{|z|^{\alpha+1}} dz}.
$$
We can decompose $L_1^{n, \alpha}$ and $S_1^{\alpha}$ as a sum of three independent variables
\begin{equation} \label{E:Ldec}
L^{n, \alpha}_1=L^{n, \alpha, s}_1+L^{n, \alpha, l}_1+L^{n, \alpha, xl}_1, \quad S^{n, \alpha}_1=S^{ \alpha, s}_1+S^{ \alpha, l}_1+S^{ \alpha, xl}_1
\end{equation}
with
$$
e^{iu L^{n, \alpha,s}_1}=e^{ \int_{\{ |z| \leq 1\}}(e^{iuz}-1-z1_{\{ |z| \leq 1\} }) \frac{c_{\alpha} \tau(z/n^{1/\alpha})}{|z|^{\alpha+1}} dz}, \quad 
e^{iu L^{n, \alpha,l}_1}=e^{ \int_{\{1< |z| \leq \eta n^{1/ \alpha}\}}(e^{iuz}-1) \frac{c_{\alpha} \tau(z/\eta n^{1/\alpha})}{|z|^{\alpha+1}} dz}, 
$$
$$
e^{iu L^{n, \alpha,xl}_1}=e^{ \int_{\{|z|>\eta n^{1/ \alpha} \}}(e^{iuz}-1) \frac{c_{\alpha} \tau(z/n^{1/\alpha})}{|z|^{\alpha+1}} dz},
$$
and where $S^{ \alpha, s}_1$, $S^{ \alpha, l}_1$, $S^{ \alpha, xl}_1$ are defined similarly with $\tau=1$. From the sub-additivity property of the total variation distance we deduce that
$$
d_{TV}( n^{1/\alpha} L_{1/n}^{\alpha}, S_1^{\alpha}) \leq d_{TV}(  L^{n,\alpha,s}_1, S_1^{\alpha, s}) +d_{TV}(  L^{n,\alpha,l}_1, S_1^{\alpha, l})+d_{TV}(  L^{n,\alpha,xl}_1, S_1^{\alpha, xl}).
$$
\noindent
\underline{Small jumps.}
We first bound $ d_{TV}(  L^{n,\alpha,s}_1, S_1^{\alpha, s}) $. Using Proposition 2.8 in Liese \cite{Liese}, we have
$$
d_{TV}(  L^{n,\alpha,s}_1, S_1^{\alpha, s}) \leq 2 \left(1-e^{-H^2} \right)^{1/2},
$$
where $H$ is the Hellinger distance between the L\'evy measures of $ L^{n,\alpha,s}_1$ and $S_1^{\alpha, s}$ defined by
$$
H^2= \int_{\{ |z| \leq 1\}} \frac{c_{\alpha}}{|z|^{\alpha+1}} \left( \sqrt{ \tau(z/n^{1/ \alpha}})-1\right)^2 dz.
$$
Using {\bf H2} we have
$$
H^2 \leq C  \int_{\{ |z| \leq 1\}} \frac{z^2}{n^{2/ \alpha} |z|^{\alpha+1}}dz  \leq C/n^{2/ \alpha},
$$
and we deduce
\begin{equation} \label{E:TV1}
d_{TV}(  L^{n,\alpha,s}_1, S_1^{\alpha, s}) \leq C/n^{1/ \alpha}.
\end{equation}
\noindent
\underline{Large bounded jumps.}
We next turn to $d_{TV}(  L^{n,\alpha,l}_1, S_1^{\alpha, l})$. By construction $ L^{n,\alpha,l}_1$ (respectivey  $S^{\alpha,l}_1$) has a compound Poisson distribution
with intensity $\lambda_{n,L}$ (respectively $\lambda_{n,S}$) and jump size distribution $\mu_{n,L}$ (respectively $\mu_{n,S}$) with
$$
\lambda_{n,L}= 2 c_{\alpha} \int_1^{\eta n^{1/ \alpha}} \frac{ \tau(z/n^{1/ \alpha})}{z^{\alpha+1} }dz, \quad \lambda_{n,S}= 2 c_{\alpha} \int_1^{\eta n^{1/ \alpha}} \frac{ 1}{z^{\alpha+1} }dz,
$$
$$
\mu_{n,L}(dz)= \frac{c_{\alpha} }{\lambda_{n,L}}\frac{ \tau(z/n^{1/ \alpha})}{|z|^{\alpha+1} }1_{\{ 1 \leq |z| \leq \eta n^{1/ \alpha}\} }dz, \quad \mu_{n,S}(dz)=  \frac{c_{\alpha} }{\lambda_{n,S}}\frac{ 1}{|z|^{\alpha+1} }1_{\{ 1 \leq |z| \leq \eta n^{1/ \alpha}\} }dz.
$$
From Theorem 13 in Mariucci and Rei{\ss}  \cite{ReissM}, we have
\begin{equation} \label{E:MR}
d_{TV}(  L^{n,\alpha,l}_1, S_1^{\alpha, l}) \leq (\lambda_{n,L} \wedge \lambda_{n,S}) d_{TV} (\mu_{n,L}, \mu_{n,S}) +1-e^{ |\lambda_{n,L}-\lambda_{n,S}|}.
\end{equation}
Using {\bf H2}, we can write
\begin{equation*}
\lambda_{n,L}=\lambda_{n,S} + R_n, \; \text{with} \;  |R_n| \leq \frac{C}{n^{1/ \alpha}} \int_1^{\eta n^{1/ \alpha}} \frac{ z}{z^{\alpha+1} }dz,
\end{equation*}
and we check easily 
$$
 |\lambda_{n,L}-\lambda_{n,S}| \leq \left\{
 \begin{array}{l}
 C/n \; \text{if} \; \alpha \neq 1, \\
 C \ln(n)/ n \; \text{if} \; \alpha =1.
 \end{array}
 \right.
$$
Moreover
\begin{align*}
d_{TV} (\mu_{n,L}, \mu_{n,S}) & = \frac{ C}{  \lambda_{n,S}\lambda_{n,L}}  \int_1^{\eta n^{1/ \alpha}} \frac{ 1}{z^{\alpha+1}}|\lambda_{n,L} - \lambda_{n,S} \tau(z/ n^{1/ \alpha})  |dz, \\
\end{align*}
and we deduce
$$
d_{TV} (\mu_{n,L}, \mu_{n,S})\leq \left\{
 \begin{array}{l}
 C/n \; \text{if} \; \alpha \neq 1, \\
 C \ln(n)/ n \; \text{if} \; \alpha =1.
 \end{array}
 \right.
$$
Plugging these results in \eqref{E:MR}, it yields
\begin{equation} \label{E:TV2}
d_{TV} (L^{n,\alpha,l}_1, S_1^{\alpha, l})\leq \left\{
 \begin{array}{l}
 C/n \; \text{if} \; \alpha \neq 1, \\
 C \ln(n)/ n \; \text{if} \; \alpha =1.
 \end{array}
 \right.
\end{equation}
\noindent
\underline{Large unbounded jumps.}
We finally bound $d_{TV}(  L^{n,\alpha,xl}_1, S_1^{\alpha, xl})$.  As previously, we observe that  $ L^{n,\alpha,xl}_1$ has a compound Poisson distribution. We denote by $N^{n,L}$ the number of jumps of $ L^{n,\alpha,xl}_1$ and by $N^{n,S}$ the number of jumps of $ S^{\alpha,xl}_1$.  We check easily that
$$
\PP (N^{n,L} \geq 1) \leq C/n, \quad \PP (N^{n,S} \geq 1) \leq C/n.
$$
Now for any bounded function $g$ we have
$$
\E g( L^{n,\alpha,xl}_1)=g(0)+ \E g(L^{n,\alpha,xl}_1)1_{ \{ N^{n,L} \geq1\}},
$$ 
and similarly
$$
\E g( S^{\alpha,xl}_1)=g(0)+ \E g(S^{\alpha,xl}_1)1_{ \{ N^{n,S} \geq1\}}.
$$
Consequently
$$
| \E g( L^{n,\alpha,xl}_1)-\E g( S^{\alpha,xl}_1)| \leq ||g||_{\infty} (\PP (N^{n,L} \geq 1)+ \PP (N^{n,S} \geq 1)) \leq C/n,
$$
and we deduce  
\begin{equation} \label{E:TV3}
d_{TV}(  L^{n,\alpha,xl}_1, S_1^{\alpha, xl})\leq C/n.
\end{equation}
Collecting \eqref{E:TV1} , \eqref{E:TV2} and \eqref{E:TV3} we obtain (i).

(ii) We recall that $\phi$ is the density of the standard Gaussian variable, $\varphi_{\alpha}$ is the density of $S_1^{\alpha}$, and  we denote by $\varphi_{n, \alpha}$ the density of $n^{1/\alpha} L_{1/n}^{\alpha}$ (we refer to Picard \cite{Picard} for the existence). We have
$$
\E g_n(  B_1+ c_n n^{1/\alpha} L_{1/n}^{\alpha})-\E g_n(  B_1+ c_n S_1^{\alpha})= \int \int   g_n(x+c_ny)\phi(x) (\varphi_{n, \alpha}(y)-\varphi_{ \alpha}(y) )dxdy. 
$$
By assumption 
$$
|g_n(x+c_ny)| \leq C  \ln(n)^p(1+ |x|^{\varepsilon \alpha} + | y|^{\varepsilon \alpha} ),
$$
and it yields
\begin{align} \label{E:TVg}
& |\E g_n(  B_1+ c_n n^{1/\alpha} L_{1/n}^{\alpha})-  \E g_n(  B_1+ c_n S_1^{\alpha})|  \leq  C \ln(n)^p  [ d_{TV}( n^{1/\alpha} L_{1/n}^{\alpha}, S_1^{\alpha})  \nonumber\\
& \quad \quad  \quad \quad + \int  | y|^{\varepsilon \alpha} |\varphi_{n, \alpha}(y)-\varphi_{ \alpha}(y) | dy ],
\end{align}
where we have used that $ d_{TV}( n^{1/\alpha} L_{1/n}^{\alpha}, S_1^{\alpha})=\frac{1}{2} \int | \varphi_{n, \alpha}(y)-\varphi_{ \alpha}(y) | dy$.

It remains to bound the second term in the right-hand side of \eqref{E:TVg}. We split the integral in two parts.
We have immediately
\begin{equation} \label{E:TVg1}
\int_{|y| \leq K_n} |y|^{\varepsilon \alpha} | \varphi_{n, \alpha}(y)-\varphi_{ \alpha}(y) | dy \leq C  K_n^{\varepsilon \alpha} d_{TV}( n^{1/\alpha} L_{1/n}^{\alpha}, S_1^{\alpha}).
\end{equation}
and  for $0<q'<\alpha(1-\varepsilon)$ (using Theorem 2 in \cite{LuschgyPages} for the last inequality)
\begin{eqnarray} \label{E:TVg2}
\int_{|y| > K_n} | y|^{\varepsilon \alpha} | \varphi_{n, \alpha}(y)-\varphi_{ \alpha}(y) | dy  \leq C(\E | n^{1/\alpha} L_{1/n}^{\alpha}|^{q'+\varepsilon \alpha}+ \E |S_1^{\alpha} |^{q'+\varepsilon \alpha} )/K_n^{q'} 
  \leq C/K_n^{q'}. 
\end{eqnarray}
We finish the proof of (ii) from  (i) and \eqref{E:TVg}, by choosing $K_n=n^{1/q'}$ with $\alpha/2<q'< \alpha(1-\varepsilon)$ (that gives $\varepsilon< \varepsilon \alpha/q'< 2 \varepsilon$) in \eqref{E:TVg1} and \eqref{E:TVg2}.

\subsubsection{Proof of Proposition \ref{P:euler}}

To simplify the notation,  we omit the parameter $\sigma_0$ in the expression of the diffusion coefficient $a$ and  write $\E_i$ instead of $\E_i^{\theta_0}$.
From the expressions of $\Delta^{n}_{i} X$ and $\Delta^{n}_{i} \overline{X}$ (see \eqref{E:euler}), we deduce the decomposition $\sqrt{n}\Delta^{n}_{i} X   =  \sqrt{n} \Delta^{n}_{i} \overline{X}+ \theta(1)_i^n + \theta(2)_i^n + \theta(3)_i^n$ with
\begin{align*}
    \theta(1)_i^n & = \sqrt{n} \int_{i/n}^{(i+1)/n}   (a(X_t)-a(X_{i/n}))dB_t+ \frac{1}{\sqrt{n}} b(X_{i/n}),\\
    \theta(2)_i^n & = \delta_0 \sqrt{n}   \int_{i/n}^{(i+1)/n}  (c(X_{t-})-c(X_{i/n}))dL^{\alpha_0}_t,\\
    \theta(3)_i^n & =\sqrt{n}\int_{i/n}^{(i+1)/n}  (b(X_t )- b(X_{i/n}))dt.
\end{align*}
We check easily from Burkholder's inequality and the assumptions on the coefficients $a,b,c$ of the stochastic equation and on the support of the L\'evy measure that 
\begin{equation*}
 \E_i (\sup_{i/n \leq t \leq (i+1)/n}|X_t-X_{i/n}|^2) \leq C/n.
\end{equation*}
Then using the isometry property for the stochastic integrals, we have
\begin{equation} \label{E:intsquare}
\E_i(  |\theta(1)_i^n|^2+ |\theta(2)_i^n|^2+ |\theta(3)_i^n|^2) \leq C/n,
\end{equation}
and we deduce immediately \eqref{E:square}.

We now prove \eqref{E:weak}. With the previous decomposition, we can write
\begin{align*}
    f_n( \sqrt{n}\Delta^{n}_{i} X) - f_n(\sqrt{n} \Delta^{n}_{i} \overline{X} )  = &
     f_n(\sqrt{n}\Delta^{n}_{i} X ) - f_n(\sqrt{n} \Delta^{n}_{i} \overline{X} + \theta(1)_i^n) \\
  &  +  f_n( \sqrt{n} \Delta^{n}_{i} \overline{X} + \theta(1)_i^n) - f_n( \sqrt{n} \Delta^{n}_{i} \overline{X} ),
\end{align*}
and it yields (using the bound for $f'_n$ for the first term in the right-hand side of the previous equality)
\begin{eqnarray} \label{E:weak1}
| \E_i  (f_n( \sqrt{n}\Delta^{n}_{i} X) - f_n(\sqrt{n} \Delta^{n}_{i} \overline{X} )) | \leq  C \ln(n)^p  \E_i  (|\theta(2)_i^n| + |\theta(3)_i^n|) \nonumber \\
+ |\E_i  (f_n( \sqrt{n} \Delta^{n}_{i} \overline{X} + \theta(1)_i^n) - f_n( \sqrt{n} \Delta^{n}_{i} \overline{X} )) |.
\end{eqnarray}
We see easily that
\begin{equation} \label{E:L1-3}
\E_i  |\theta(3)_i^n| \leq C/n,
\end{equation}
and  for $\epsilon >0$
\begin{eqnarray} \label{E:L1-2}
\E_i  |\theta(2)_i^n| \leq \left\{
\begin{array}{l}
C_{\epsilon}/n^{1/ \alpha_0 - \epsilon}, \; \text{if} \; \alpha_0 >1,  \\
C_{\epsilon}/n^{1- \epsilon}, \; \text{if} \; \alpha_0 \leq 1.
\end{array} \right.
\end{eqnarray}
Let us give some details for \eqref{E:L1-2}. We can represent $L_t^{\alpha_0}$ as
$$
L_t^{\alpha_0}= \int_0^t \int_{0< |z| \leq K} z \tilde{N}^{\alpha_0, \tau}(dz,ds),
$$
where $\tilde{N}^{\alpha_0, \tau}$ is a compensated Poisson random measure with compensator $\overline{N}^{\alpha_0, \tau}(dz,ds)= F^{\alpha_0, \tau}(dz)  ds$. Then using successively H\"older's inequality with $q \in (1\lor \alpha_0,2]$ and Burkholder's inequality (Lemma 2.1.5 in \cite{JacodProtter}) we obtain
\begin{eqnarray*}
\E_i  |\int_{i/n}^{(i+1)/n}  (c(X_{t-})-c(X_{i/n}))dL^{\alpha_0}_t |   \leq  \left(\E_i |\int_{i/n}^{(i+1)/n}  (c(X_{t-})-c(X_{i/n}))dL^{\alpha_0}_t |^q\right)^{1/q} \\
 \leq C \left( \E_i \int_{i/n}^{(i+1)/n} \int_{0< |z| \leq K} | c(X_{t-})-c(X_{i/n})|^q |z|^q F^{\alpha_0, \tau}(dz) dt\right)^{1/q}  \\
 \leq C   (\E_i \sup_{i/n \leq t \leq (i+1)/n}|X_t-X_{i/n}|^2)^{1/2} /n^{1/q}.
\end{eqnarray*}
We deduce  \eqref{E:L1-2} by taking  $q $ arbitrarily close to $\alpha_0$, if $\alpha_0>1$, or arbitrarily close to $1$ if $\alpha_0 \leq 1$. 

To finish the proof,  it remains to bound the second term in the right-hand side of  \eqref{E:weak1}.  A second order Taylor expansion gives
\begin{eqnarray}
 |\E_i  (f_n( \sqrt{n} \Delta^{n}_{i} \overline{X} + \theta(1)_i^n) - f_n( \sqrt{n} \Delta^{n}_{i} \overline{X} )) | \leq |\E_i  f'_n( \sqrt{n} \Delta^{n}_{i} \overline{X} )\theta(1)_i^n |
 + C \ln(n)^p \E_i |\theta(1)_i^n|^2,
\end{eqnarray}
where the last term is bounded by \eqref{E:intsquare}. We end  the proof of \eqref{E:weak} by showing  for $\epsilon >0$
\begin{eqnarray} \label{E:L1-1}
 |\E_i  f'_n( \sqrt{n} \Delta^{n}_{i} \overline{X} )\theta(1)_i^n |\leq \left\{
\begin{array}{l}
C_{\epsilon}/n^{1/ \alpha_0 - \epsilon}, \; \text{if} \; \alpha_0 >1,  \\
C_{\epsilon}/n^{1- \epsilon}, \; \text{if} \; \alpha_0 \leq 1.
\end{array} \right.
\end{eqnarray}
To prove \eqref{E:L1-1}, we write with It\^{o}'s formula 
\[ a(X_t) -a(X_s) = \int_s^t \overline{b}_v dv + \int_s^t (a'a)(X_v)dB_v
    + \int_s^t \int \overline{c}_{v-} \tilde{N}^{\alpha_0, \tau}(dz,dv),\]
where
\begin{align*}
    \overline{b}_v & = \frac{1}{2} (a ^{\prime \prime} a)(X_v) + (a^{\prime} b)(X_v) \\
    &  + \int_{ |z| \leq K} (a(X_{v-}+ \delta_0 c(X_{v-}) z)- a(X_{v-})-\delta a'(X_{v-}) c(X_{v-}) z) F^{\alpha_0, \tau}(dz), \\
    \overline{c}_{v-}& = a(X_{v-}+ \delta_0 c(X_{v-}) z)- a(X_{v-}).
\end{align*}
By assumption we have  $| \overline{b}_v | \leq C$ and $|  \overline{c}_{v-}| \leq C|z|  1_{ 0< |z| \leq K}$.
Then we obtain  the decomposition
\begin{equation} \label{eq:Dtheta1}
\theta(1)_i^n=\mu(1)_i^n + \mu(2)_i^n + \mu(3)_i^n +\mu(4)_i^n,
\end{equation}
with
\begin{align*}
 \mu(1)_i^n   &  =\sqrt{n}  \int_{i/n}^{(i+1)/n} \left(\int_{i/n}^t \overline{b}_v dv \right)dB_t ,\\
\mu(2)_i^n  &   = \sqrt{n}  (a' a)(X_{i/n} )\int_{i/n}^{(i+1)/n}(B_t - B_{i/n})dB_t +  \frac{1}{\sqrt{n}} b(X_{i/n}), \\
\mu(3)_i^n  & = \sqrt{n}  \int_{i/n}^{(i+1)/n}  \left(\int_{i/n}^t ((a'a)(X_v)-(a' a)(X_{i/n}))d B_v \right) dB_t, \\
\mu(4)_i^n   &  = \sqrt{n}  \int_{i/n}^{(i+1)/n} \left(\int_{i/n}^t \int\overline{c}_{v-} \tilde{N}^{\alpha_0, \tau}(dz,dv)\right) dB_t.
\end{align*}
Consequently, it yields
\begin{align*}
 |\E_i  f'_n( \sqrt{n} \Delta^{n}_{i} \overline{X} )\theta(1)_i^n | & \leq C \ln(n)^p \E_i ( | \mu(1)_i^n| + | \mu(3)_i^n | + | \mu(4)_i^n |) 
 +   |\E_i  f'_n( \sqrt{n} \Delta^{n}_{i} \overline{X} )\mu(2)_i^n | .
\end{align*}
We check easily from Burkholder's inequality
$$
\E_i ( | \mu(1)_i^n| + | \mu(3)_i^n |) \leq C/n,
$$
and with similar arguments as in the proof of \eqref{E:L1-2}
$$
 \E_i | \mu(4)_i^n | \leq \left\{
 \begin{array}{l}
C_{\epsilon}/n^{1/ \alpha_0 - \epsilon}, \; \text{if} \; \alpha_0 >1,  \\
C_{\epsilon}/n^{1- \epsilon}, \; \text{if} \; \alpha_0 \leq 1.
 \end{array}
 \right.
$$
For the last term, we separate the Brownian contribution from the jumps by setting $\Delta^{n}_{i} \overline{X} = \Delta^{n}_{i} \overline{X}^{1}+ \Delta^{n}_{i} \overline{X}^{2}$ with
$$
 \Delta^{n}_{i} \overline{X}^{1}=a(X_{i/n})\Delta_i^nB \quad \text{and} \quad  \Delta^{n}_{i} \overline{X}^{2}=\delta_0  c(X_{i/n}) \Delta_i^n L^{\alpha_0}.
$$
It yields
$$
\E_{i}  f'_n( \sqrt{n} \Delta^{n}_{i} \overline{X} )\mu(2)_i^n =\E_i  f'_n( \sqrt{n} \Delta^{n}_{i} \overline{X}^{1} )\mu(2)_i^n+ \E_i (f'_n( \sqrt{n} \Delta^{n}_{i} \overline{X} )-f'_n( \sqrt{n} \Delta^{n}_{i} \overline{X}^{1} ) ) \mu(2)_i^n .
$$
Using that $f'_n$ is odd, an explicit calculus gives $\E_i  f'_n( \sqrt{n} \Delta^{n}_{i} \overline{X}^{1} )\mu(2)_i^n=0$. For the second term, using  the bound for $f_n^{\prime \prime}$ and the independence between $\mu(2)_i^n$ and $\Delta^{n}_{i} \overline{X}^{2}$ we have
\begin{align*}
\E_i| (f'_n( \sqrt{n} \Delta^{n}_{i} \overline{X} )-f'_n( \sqrt{n} \Delta^{n}_{i} \overline{X}^{1} ) ) \mu(2)_i^n| & \leq \ln(n)^p \sqrt{n} \E_i |\mu(2)_i^n | \E_i|\Delta^{n}_{i} \overline{X}^{2}|,
\end{align*}
where obviously $\E_i |\mu(2)_i^n |  \leq C/ \sqrt{n}$. Moreover with $q=\alpha_0+ \epsilon$ if  $\alpha_0>1$ or $q=1+ \epsilon$ otherwise and using once again Burkholder's inequality for a stochastic integral with respect to a compensated Poisson measure (as in the proof of \eqref{E:L1-2}), we can show
$$
\E_i|\Delta^{n}_{i} \overline{X}^{2}| \leq  \E_i^{1/q}|\Delta^{n}_{i} \overline{X}^{2}|^q
\leq 
 \left\{
 \begin{array}{l}
C_{\epsilon}/n^{1/ \alpha_0 - \epsilon}, \; \text{if} \; \alpha_0 >1,  \\
C_{\epsilon}/n^{1-\epsilon}, \; \text{if} \; \alpha_0 \leq 1,
 \end{array}
 \right.
$$
and we finally deduce
$$
\E_i| (f'_n( \sqrt{n} \Delta^{n}_{i} \overline{X} )-f'_n( \sqrt{n} \Delta^{n}_{i} \overline{X}^{1} ) ) \mu(2)_i^n| \leq 
 \left\{
 \begin{array}{l}
C_{\epsilon}/n^{1/ \alpha_0 - \epsilon}, \; \text{if} \; \alpha_0 >1,  \\
C_{\epsilon}/n^{1-\epsilon}, \; \text{if} \; \alpha_0 \leq 1.
 \end{array}
 \right.
$$
This achieves the proof of \eqref{E:L1-1}. 


\subsection{Convolution} \label{Ss:Paprox}
\subsubsection{Proof of Lemma \ref{L:I}}
Let $K\subset (0,2)$ be a compact set. We prove that $\exists C>0, \; \forall \alpha \in K,$ for $y$ large
\begin{equation}\label{eq:LemmeIpreuve}
    \Bigl|\mathcal{I}_{\alpha}^{(k, l)}(y) - \psi_{\alpha}^{(l)}(y) \int\mathcal{D}^{(k)}(\phi)(z)dz \Bigr| \leq C \frac{\psi_{\alpha}^{(l)}(y)}{\sqrt{|y|}},
\end{equation}
which implies \eqref{eq:Icaleq}. \eqref{eq:D} is a direct consequence as $\int \phi(z)dz=1$. 
We conclude \eqref{eq:majIcal} using that $y \to \mathcal{I}_{\alpha}^{(0, 0)}(y)$ is continuous, and that $\forall y, \; 0 < \mathcal{I}_{\alpha}^{(0, 0)}(y)$, so that $\mathcal{I}_{\alpha}^{(0, 0)}$ is lower bounded by a positive constant on any compact.

We prove \eqref{eq:LemmeIpreuve} by following the proof of (A.5) in the Supplement to Aït-Sahalia and Jacod \cite{AitSahalia2012}. With the notations $A_y=\{z: |z|>1, \; |y-z|\geq \sqrt{|y|} \}$ and $\overline{A}_y=\{z: |z|>1, \; |y-z|< \sqrt{|y|} \}$ we have
$$\mathcal{I}_{\alpha}^{(k, l)}(y) = \int_{A_y} \mathcal{D}^{(k)}(\phi)(y-z) \psi_{\alpha}^{(l)}(z)dz  + \int_{\overline{A}_y} \mathcal{D}^{(k)}(\phi)(y-z) \psi_{\alpha}^{(l)}(z)dz  := I_{\alpha}(y) + \overline{I}_{\alpha}(y).$$
We note that $\mathcal{D}^{(k)}(\phi)(y)=P_{2k}(x)\phi(y)$ for $P_{2k}$ a polynomial of order $2k$.
On $A_y$ we have $|\mathcal{D}^{(k)}(\phi)(y-z)| \leq C \phi(\frac{y-z}{2}) \leq C e^{-|y|/4}$. As $\sup_{\alpha \in K} \int \psi_{\alpha}^{(l)}< \infty$, 
$$\exists C>0, \; \forall \alpha \in K, \; \forall y, \quad |I_{\alpha}(y)| \leq C e^{-|y|/4},$$ and we conclude that $I_{\alpha}$ is negligible. We now turn to $\overline{I}_{\alpha}$. We set
$$\eta_{\alpha, y}(z) = \frac{\psi_{\alpha}^{(l)}(z)}{\psi_{\alpha}^{(l)}(y)} \mathbb{1}_{\overline{A}_y}(z) - 1,$$
with this notation we have
$$\frac{\overline{I}_{\alpha}(y)}{\psi_{\alpha}^{(l)}(y)}-\int\mathcal{D}^{(k)}(\phi)(z)dz=\int \mathcal{D}^{(k)}(\phi)(y-z)\eta_{\alpha, y}(z)dz.$$   
Using Taylor's formula, we can write
$$\eta_{\alpha, y}(z) = -\mathbb{1}_{A_y}(z) + (z-y) \frac{(\psi_{\alpha}^{(l)})'(c)}{\psi_{\alpha}^{(l)}(y)}\mathbb{1}_{\overline{A}_y}(z),$$    
for some $c\in [y,z] \subset [y-\sqrt{|y|} ,y+\sqrt{|y|} ]\subset[y/2, 2y]$ from the definition of $\overline{A}_y$ and for $y$ large.
Then, using again the definition of $\overline{A}_y$, we have $\exists C>0, \; \forall \alpha \in K,$ such that for $y$ large
$$|\eta_{\alpha, y}(z)| \leq \mathbb{1}_{A_y}(z) + C\frac{1}{\sqrt{|y|}}. $$
Hence $\exists C>0, \; \forall \alpha \in K,$ such that for $y$ large
\begin{align*}
    \Bigl|\frac{\overline{I}_{\alpha}(y)}{\psi_{\alpha}^{(l)}(y)}-\int\mathcal{D}^{(k)}(\phi)(z)dz \Bigr|&\leq \int_{A_y} |\mathcal{D}^{(k)}(\phi)(y-z)|dz + C\frac{1}{\sqrt{|y|}} \int |\mathcal{D}^{(k)}(\phi)(y-z)| dz\\
    \leq C \Bigl(\int_{A_y} & e^{-|y-z|^2/8}e^{-|y-z|^2/8}dz + \frac{1}{\sqrt{|y|}}\Bigr) \leq C\Bigl(e^{-|y|/8}+\frac{1}{\sqrt{|y|}}\Bigr),
\end{align*}
and we conclude \eqref{eq:LemmeIpreuve}.

\subsubsection{Proof of Lemma \ref{l:approxD}}
We first prove \eqref{eq:funcapprox}. These results are obtained from estimate (A.10) of the Supplement to Aït-Sahalia and Jacod \cite{AitSahalia2012}. We detail the proof for $g_{\alpha}$. From Section 14 in Sato \cite{Sato}
$$\varphi_\alpha(z) \underset{z\to\infty}{=} \frac{c_\alpha}{|z|^{\alpha+1}} + O\Bigl(\frac{1}{|z|^{2\alpha+1}}\Bigr),$$
with $c_\alpha$ defined in \eqref{eq:calpha}.
More precisely, let $K\subset (0,2)$ be a compact set. We obtain from Section 14 in Sato \cite{Sato} that
\begin{align}
    \exists C>0, \; \forall \alpha \in K, \; \forall z>3, \quad &|\varphi_\alpha(z) - c_\alpha \psi_{\alpha}^{(0)}(z)| \leq C \psi_{2\alpha}^{(0)}(z),\label{eq:eqdensite}\\
    \exists C>0, \; \forall \alpha \in K, \; \forall z \neq 0, \quad  &\varphi_\alpha(z) \leq \frac{C}{|z|^{\alpha+1}},\label{eq:majdensite}
\end{align}
where $\psi_{\alpha}^{(p)}$ is defined in \eqref{eq:fp}. We write the decomposition
    \begin{align}
       g_{\alpha}(y, w) &= \int_{|z|\leq1/w} \bigl(\mathcal{D}(\phi)(y-zw)-\mathcal{D}(\phi)(y)\bigr)\varphi_\alpha(z) dz \label{eq:gdecomp} \\
       &+ \int_{|z|>1/w} \bigl(\mathcal{D}(\phi)(y-zw)-\mathcal{D}(\phi)(y)\bigr)\varphi_\alpha(z) dz 
       + \mathcal{D}(\phi)(y) \int \varphi_\alpha(z) dz,\nonumber
    \end{align}
    with $\int \varphi_\alpha= 1$. We first note using Taylor's formula that $\forall y,\; \forall |x|\leq 1,\;  \exists c \in [y-x, y]$
    \begin{align*}
    \left|\mathcal{D}(\phi)(y-x)-\mathcal{D}(\phi)(y)+\mathcal{D}(\phi)'(y)x\right|&\leq x^2 |\mathcal{D}(\phi)^{\prime \prime}(c)| \leq C x^2 \phi(\frac{y}{2}),
    \end{align*}
    where we used that $|\mathcal{D}(\phi)^{\prime \prime}(c)| \leq C \phi(\frac{y}{2})$, as $\mathcal{D}(\phi)^{\prime \prime}(c)=P(c) \phi(c)$ with P a polynomial.
    Hence, using that $z \to z \varphi_\alpha(z)$ is an odd function and \eqref{eq:majdensite}, $\exists C>0, \; \forall \alpha \in K, \; \forall y$
    \begin{align}
        \Bigl|\int_{|z|\leq1/w} \bigl(\mathcal{D}(\phi)(y-zw) & -\mathcal{D}(\phi)(y) +\mathcal{D}(\phi)'(y)zw \bigr)\varphi_\alpha(z) dz\Bigr| \nonumber \\
        &= \Bigl|\int_{|z|\leq1/w} \bigl(\mathcal{D}(\phi)(y-zw) -\mathcal{D}(\phi)(y)\bigr)\varphi_\alpha(z) dz\Bigr| \nonumber \\
        &\leq  C \phi(\frac{y}{2}) \int_{|z|\leq1/w}(zw)^2  \varphi_\alpha(z) dz \nonumber \\  
        &\leq  C \phi(\frac{y}{2}) \int_{|z|\leq1/w}(zw)^2  \frac{dz}{|z|^{\alpha+1}}  \leq C w^\alpha \phi(\frac{y}{2}). \label{eq:g1}
    \end{align}
    We have from \eqref{eq:majdensite} that $\exists C>0, \; \forall \alpha \in K$
    \begin{equation}\label{eq:g2}
        \int_{|z|>1/w} \varphi_\alpha(z) dz \leq C\int_{|z|>1/w} \frac{dz}{|z|^{\alpha+1}} \leq C w^\alpha.
    \end{equation}
    We get from \eqref{eq:eqdensite} that $ \exists C>0, \; \forall \alpha \in K, \; \forall w \in (0, 1/3],\; \forall y$    
    \begin{align*}
        \Bigl|\int_{|z|>1/w} \mathcal{D}(\phi)(y-zw)\varphi_\alpha(z) dz - c_\alpha \int_{|z|>1/w} & \mathcal{D}(\phi)(y-zw)\psi_{\alpha}^{(0)}(z) dz\Bigr|\\
    &\leq C \Bigl|\int_{|z|>1/w} \mathcal{D}(\phi)(y-zw)\psi_{2\alpha}^{(0)}(z) dz\Bigr|.
    \end{align*}
    After the change of variable $wz \to z$, we have 
    \begin{align*}
        \Bigl|\int_{|z|>1/w} \mathcal{D}(\phi)(y-zw)\varphi_\alpha(z) dz - c_\alpha w^\alpha \int_{|z|>1} & \mathcal{D}(\phi)(y-z)\psi_{\alpha}^{(0)}(z) dz\Bigr|\\& \leq C w^{2\alpha} \Bigl|\int_{|z|>1} \mathcal{D}(\phi)(y-z)\psi_{2\alpha}^{(0)}(z) dz\Bigr|,
    \end{align*}
    i.e. recalling the definition of $\mathcal{I}^{(k,l)}_\alpha$ given in \eqref{eq:Ical}, $ \exists C>0, \; \forall \alpha \in K, \; \forall w \in (0, 1/3],\; \forall y$    
    \begin{align}\label{eq:g3}
        \Bigl|\int_{|z|>1/w} \mathcal{D}(\phi)(y-zw)\varphi_\alpha(z) dz - c_\alpha w^\alpha \mathcal{I}_\alpha^{(1, 0)}(y)\Bigr| \leq C w^{2\alpha} \bigl|\mathcal{I}_{2\alpha}^{(1, 0)}(y)\bigr|.
    \end{align}
    Combining the decomposition \eqref{eq:gdecomp} with \eqref{eq:g1}, \eqref{eq:g2} and \eqref{eq:g3}, we have proved that $\exists C>0, \; \forall \alpha \in K, \; \forall w \in (0, 1/3], \; \forall y$
    \begin{equation}
        \bigl|g_{\alpha}(y, w) - \mathcal{D}(\phi)(y)- c_\alpha w^\alpha \mathcal{I}_\alpha^{(1, 0)}(y)\bigr|\leq C \left[ w^\alpha \left(\phi(\frac{y}{2})+|\mathcal{D}(\phi)(y)|\right) +w^{2\alpha} \bigl|\mathcal{I}_{2\alpha}^{(1, 0)}(y)\bigr| \right].
    \end{equation}
    Using that from \eqref{eq:majIcal} $|\mathcal{I}_{\alpha}^{(1, 0)}(y)| \leq C \psi_{\alpha}^{(0)}(y)$ and $|\mathcal{I}_{2\alpha}^{(1, 0)}(y)| \leq C \psi_{\alpha}^{(0)}(y)$, and that $|\mathcal{D}(\phi)(y)| \leq C \phi(\frac{y}{2}) \leq C \psi_{\alpha}^{(0)}(y)$, we conclude the estimate of $g_\alpha$ in \eqref{eq:funcapprox}.

    The estimates for $f_\alpha$, $h_\alpha$ and $k_\alpha$ in \eqref{eq:funcapprox} are proved similarly, where we have from Section 14 in Sato \cite{Sato}
    \begin{align*}
        \mathcal{D}(\varphi_\alpha)(z) &\underset{z\to\infty}{=} -\frac{ \alpha c_\alpha}{|z|^{\alpha+1}}+ O\Bigl(\frac{1}{|z|^{2\alpha+1}}\Bigr), \\
        \partial_\alpha \varphi_\alpha(z) &\underset{z\to\infty}{=} \frac{\partial_\alpha c_\alpha - c_\alpha\ln|z|}{|z|^{\alpha+1}}+ O\Bigl(\frac{\ln|z|}{|z|^{2\alpha+1}}\Bigr).
    \end{align*}
    More precisely, we obtain from Section 14 in Sato \cite{Sato} that
    \begin{align*}
        \exists C>0, \; \forall \alpha \in K, \; \forall z>3, \quad &|\mathcal{D}(\varphi_\alpha)(z) + \alpha c_\alpha \psi_{\alpha}^{(0)}(z)| \leq  C \psi_{2\alpha}^{(0)}(z),\\
        &|\partial_\alpha \varphi_\alpha(z) - (\partial_\alpha c_\alpha \psi_{\alpha}^{(0)}(z) - c_\alpha \psi_{\alpha}^{(1)}(z))| \leq C \psi_{2\alpha}^{(1)}(z),\\
        \exists C>0, \; \forall \alpha \in K, \; \forall z \neq 0, \quad &|\mathcal{D}(\varphi_\alpha)(z)| \leq \frac{C}{|z|^{\alpha+1}}, \quad |\partial_\alpha \varphi_\alpha(z)| \leq \frac{C\ln|z|}{|z|^{\alpha+1}}.
    \end{align*}

We now prove \eqref{eq:majfghk}. The bounds for $f_\alpha$, $h_\alpha$ and $k_\alpha$ are immediate consequences of \eqref{eq:majIcal} and \eqref{eq:funcapprox}.
We detail the bound for $g_\alpha$. As $\mathcal{D}(\phi)(y) = (1-y^2)\phi(y)$, we have from \eqref{eq:majIcal} and \eqref{eq:funcapprox} that $\exists C>0, \; \forall \alpha \in K, \; \forall w \in (0, 1/3], \; \forall y$
$$    |g_\alpha(y, w)| \leq C \big((1+|y|^{2})\phi(y) + w^\alpha \psi_{\alpha}^{(0)}(y)\big).$$
We refine this bound. We define $y^*(w)$ as the positive solution of $(1+|y|^{2})\phi(y)= w^{\alpha} \psi_{\alpha}^{(0)}(y)$, then $y^*(w) \underset{w\to0}{\sim}\sqrt{2\alpha \ln(1/w)}$.  As $y \to \frac{(1+|y|^{2})\phi(y)}{\psi_{\alpha}^{(0)}(y)}$ is decreasing after a certain rank, we have
\begin{align}
        \forall |y|< y^*(w), \;   
    |g_\alpha(y, w)| &\leq C \bigl((1+|y|^{2})\phi(y) + w^\alpha \psi_{\alpha}^{(0)}(y)\bigr) \leq C(1+y^{2}) \phi(y) \leq C \ln(1/w) \phi(y), \nonumber \\
    \forall |y|\geq y^*(w), \; 
    |g_\alpha(y, w)| &\leq C \bigl((1+|y|^{2})\phi(y) + w^\alpha \psi_{\alpha}^{(0)}(y)\bigr) \leq C w^{\alpha} \; \psi_{\alpha}^{(0)}(y),\label{eq:refineboundg}
\end{align}
and we conclude.

\subsubsection{Proof of Lemma \ref{L:fonctionsApprox}}
We detail the proof of \eqref{eq:etape1Approxf}. From Lemma \ref{l:approxD}, $\forall w \in (0, 1/3], \, \forall y $
$$\bigl|f_{\alpha}(y, w) -  \phi(y) - c_{\alpha} w^{\alpha} \mathcal{I}_{\alpha}^{(0, 0)}(y)\bigr|\leq C\bigl(w^{\alpha} \phi(\frac{y}{2}) + w^{2\alpha}\psi_{\alpha}^{(0)}(y)\bigr).$$
Hence, using \eqref{eq:D}, we have for $n$ large $\forall |y|>1 $
\begin{align*}
    \bigl|f_{\alpha}(y, w_n(\theta)) &-  \phi(y) - c_{\alpha} w_n(\theta)^{\alpha} \psi_{\alpha}^{(0)}(y)\bigr|\\
    &\leq C\Bigl(w_n(\theta)^{\alpha} \phi(\frac{y}{2}) + w_n(\theta)^{\alpha} \psi_{\alpha}^{(0)}(y) \frac{1}{\sqrt{|y|}}+ w_n(\theta)^{2\alpha}\psi_{\alpha}^{(0)}(y)\Bigr).
\end{align*}
Noting that 
$$\forall |y|>1, \quad \phi(\frac{y}{2}) \leq C \psi_{\alpha}^{(0)}(y) \frac{1}{\sqrt{|y|}},$$
we obtain using \eqref{eq:majfghk} that $\forall \Gamma \geq 1 , \forall |y|>1 $
\begin{align*}
    \bigl|f_{\alpha}(y, w_n(\theta)) -  \phi(y) - c_{\alpha} w_n(\theta)^{\alpha} \psi_{\alpha}^{(0)}(y)\bigr|&\leq C\Bigl(\frac{1}{\sqrt{\Gamma}}+ w_n(\theta)^{\alpha}\Bigr) w_n(\theta)^{\alpha} \psi_{\alpha}^{(0)}(y)\\
    &\leq C\Bigl(\frac{1}{\sqrt{\Gamma}}+ \frac{1}{n^{1-\alpha/2}}\Bigr)f_{\alpha}(y, w_n(\theta)).
\end{align*}
\eqref{eq:etape1Approxh} is obtained similarly.

We turn to \eqref{eq:Lbarre}. Noting that 
$$\ln(1/w_n(\theta))= \frac{2-\alpha}{2 \alpha}\ln(n) - \ln(\frac{\delta }{\sigma}),$$
we have from Lemma \ref{l:approxD} that for $n$ large, $\forall y$
\begin{align*}
    \bigl|h_{\alpha}(y, w_n(\theta)) &+ \alpha c_{\alpha}  w_n(\theta)^{\alpha} \mathcal{I}_{\alpha}^{(0, 0)}(y)\bigr| \leq C (w_n(\theta)^{\alpha} \phi(\frac{y}{2}) + w_n(\theta)^{2\alpha} \psi_{\alpha}^{(0)}(y)),\\
    \Bigl|k_{\alpha}(y, w_n(\theta)) 
    &+ c_{\alpha} w_n(\theta)^{\alpha} \mathcal{I}_{\alpha}^{(0, 1)}(y) 
     + w_n(\theta)^{\alpha} \bigl[c_{\alpha}\frac{2-\alpha}{2 \alpha} \ln(n) - c_{\alpha} \ln(\frac{\delta }{\sigma}) -\partial_\alpha c_{\alpha} \bigr]\mathcal{I}_{\alpha}^{(0, 0)}(y) 
    \Bigr|\\
    &\leq C  \Bigl(\ln(n)w_n(\theta)^{\alpha} \phi(\frac{y}{2})
    + \ln(n) w_n(\theta)^{2\alpha} \psi_{\alpha}^{(0)}(y) + w_n(\theta)^{2\alpha} \psi_{\alpha}^{(1)}(y) \Bigr).
\end{align*}
Recalling the definition of $l_{\alpha}$ given in \eqref{def:l}, the terms in $\ln(n) w_n(\theta)^{\alpha} \mathcal{I}_{\alpha}^{(0, 0)}$ cancel out, hence
\begin{align*}
    \Bigl|l_{\alpha}(y, w_n(\theta)) &+  c_{\alpha} w_n(\theta)^{\alpha} \mathcal{I}_{\alpha}^{(0, 1)}(y) - \frac{c_{\alpha}}{2} \ln(\ln n) w_n(\theta)^{\alpha} \mathcal{I}_{\alpha}^{(0, 0)}(y) 
-  w_n(\theta)^{\alpha} \bigl[c_{\alpha}\ln(\frac{\delta }{\sigma})+\partial_\alpha c_{\alpha}\bigr]\mathcal{I}_{\alpha}^{(0, 0)}(y)\Bigr|\\
&\leq C \Bigl(\ln(n)w_n(\theta)^{\alpha} \phi(\frac{y}{2}) + \ln(n)w_n(\theta)^{2\alpha} \psi_{\alpha}^{(0)}(y)+w_n(\theta)^{2\alpha} \psi_{\alpha}^{(1)}(y) \Bigr).
\end{align*}
We conclude using \eqref{eq:D} and that $\forall y, \; \phi(\frac{y}{2}) \leq C \psi_{\alpha}^{(0)}(y)\;\phi(\frac{y}{3})$.

\subsubsection{Proof of Proposition \ref{L:tout}}
With  $M_n=\sqrt{(2- \alpha)\ln(n)}$,  $\Psi^{(p)}(z)=\int_{|y|>z} \psi_{\alpha}^{(p)}(y) dy$ (for $p \in \{0,1\}$) and $K(\theta)= \partial_\alpha c_{\alpha} + c_{\alpha} \ln(\frac{\delta }{\sigma})$,
we obtain by combining Lemma \ref{L:IntGamman} and Lemma \ref{l:ramenerInt} below
$$\bigg| \int \frac{h_{\alpha}^2 }{f_{\alpha}} (y, w_n(\theta))dy 
   - \alpha^2 c_{\alpha} w_n(\theta)^{\alpha} \Psi^{(0)} (M_n)\bigg|\leq C \frac{1}{n^{1 -\alpha/2 } \ln(n)^{\alpha/2}} \frac{\ln(\ln n)}{\ln(n)^{1/4\wedge\alpha}},$$
\begin{align*}
    \bigg| \int \frac{h_{\alpha} l_{\alpha}}{f_{\alpha}} (y, w_n(\theta))dy 
   - \alpha w_n(\theta)^{\alpha}\Bigl( c_{\alpha} \Psi^{(1)} (M_n) -\left[\frac{c_{\alpha}}{2} \ln(\ln n) +  K (\theta) \right] \Psi^{(0)}(M_n) \Bigr)\bigg| \\
\leq  C \frac{1}{n^{1 -\alpha/2 } \ln(n)^{\alpha/2}} \frac{(\ln(\ln n))^2}{\ln(n)^{1/4\wedge\alpha}},
\end{align*}
\begin{align*}
    \bigg| \int \frac{(l_{\alpha})^2}{f_{\alpha}} (y, w_n(\theta))dy - \frac{w_n(\theta)^{\alpha}}{c_{\alpha}}
    \bigg(c_{\alpha}^2 \Psi^{(2)} (M_n) - 2  c_{\alpha}  \left[\frac{c_{\alpha}}{2} \ln(\ln n) + K(\theta)  \right] \Psi^{(1)} (M_n) \\
    +\bigg[\frac{c_{\alpha}^2}{4} (\ln(\ln n))^2  
    +  c_{\alpha}\ln(\ln n) K(\theta) +  K( \theta)^2 \bigg] \Psi^{(0)} (M_n) \bigg)\bigg| \\
   \leq  C \frac{1}{n^{1 -\alpha/2 } \ln(n)^{\alpha/2}} \frac{(\ln(\ln n))^3}{\ln(n)^{1/4\wedge\alpha}}.
\end{align*}
After integrating by parts when necessary, we have
\begin{align*}
\Psi^{(0)}(M_n)&=\int_{|y|>M_n} \psi_{\alpha}^{(0)}(y)dy = \frac{2}{\alpha M_n^{\alpha}} = \frac{2}{\alpha (2-\alpha)^{\alpha/2} \ln(n)^{\alpha/2}},\\
\Psi^{(1)}(M_n)&=\int_{|y|>M_n} \psi_{\alpha}^{(1)}(y)dy = \frac{2\ln(M_n)+2/\alpha}{\alpha M_n^{\alpha}} = \frac{\ln(\ln n) + \ln(2-\alpha)+2/\alpha}{\alpha (2-\alpha)^{\alpha/2} \ln(n)^{\alpha/2}},\\
\Psi^{(2)}(M_n)&=\int_{|y|>M_n} \psi_{\alpha}^{(2)}(y)dy = \frac{2\ln(M_n)^2 + 4 \ln(M_n)/\alpha+4/\alpha^2}{\alpha M_n^{\alpha}} \\
& = \frac{(\ln(\ln n))^2/2 + \ln(\ln n)(\ln(2-\alpha)+2/\alpha)
+\ln(2-\alpha)^2/2 + 2\ln(2-\alpha)/ \alpha+ 4/\alpha^2}{\alpha (2-\alpha)^{\alpha/2} \ln(n)^{\alpha/2}}.
\end{align*}
We observe that both the terms in $(\ln(\ln n))^2$ and $\ln(\ln n)$ cancel out, and we conclude.

\begin{lem}\label{L:IntGamman}
With $\Gamma_n=\sqrt{(1-\alpha/2) \ln(n)}$, we have
\begin{align}
    \bigg| \int \frac{h_{\alpha}^2}{f_{\alpha}} (y, w_n(\theta))dy 
- w_n(\theta)^{2 \alpha} \alpha^2 c_{\alpha}^2 \int_{|y|>\Gamma_n} \frac{(\psi_{\alpha}^{(0)}(y))^2}{\phi(y) + c_{\alpha} w_n(\theta)^{\alpha} \psi_{\alpha}^{(0)}(y)}dy \bigg| \nonumber \\
\leq C \frac{1}{\ln(n)^{1/4}} \frac{1}{n^{1-\alpha/2} \ln(n)^{\alpha/2}}, \label{eq:IntGammaH2} 
\end{align}
\begin{align}
    \bigg| \int \frac{h_{\alpha} l_{\alpha}}{f_{\alpha}} (y, w_n(\theta))dy 
- w_n(\theta)^{2 \alpha}  \bigg[ \alpha c_{\alpha}^2 \int_{|y|>\Gamma_n} \frac{\psi_{\alpha}^{(1)}(y) \psi_{\alpha}^{(0)}(y)}{\phi(y) + c_{\alpha} w_n(\theta)^{\alpha} \psi_{\alpha}^{(0)}(y)}dy \nonumber \\
-  \alpha c_{\alpha} \bigl(\frac{c_{\alpha}}{2} \ln(\ln n) + K( \theta)\bigr)  \int_{|y|>\Gamma_n} \frac{(\psi_{\alpha}^{(0)}(y))^2}{\phi(y) + c_{\alpha} w_n(\theta)^{\alpha} \psi_{\alpha}^{(0)}(y)}dy \bigg] \bigg|  \nonumber \\
\leq C \frac{\ln(\ln n)}{\ln(n)^{1/4}} \frac{1}{n^{1-\alpha/2} \ln(n)^{\alpha/2}}, \label{eq:IntGammaHL}
\end{align}
\begin{align}
\biggl| \int \frac{l_{\alpha}^2}{f_{\alpha}} (y, w_n(\theta))dy 
- w_n(\theta)^{2 \alpha}  \biggl[ c_{\alpha}^2 & \int_{|y|>\Gamma_n} \frac{(\psi_{\alpha}^{(1)}(y))^2}{\phi(y) + c_{\alpha} w_n(\theta)^{\alpha} \psi_{\alpha}^{(0)}(y)}dy \nonumber \\
-  2 c_{\alpha} \bigl(\frac{c_{\alpha}}{2} \ln(\ln n) + K(\theta)\bigr) & \int_{|y|>\Gamma_n} \frac{\psi_{\alpha}^{(1)}(y) \psi_{\alpha}^{(0)}(y)}{\phi(y) + c_{\alpha} w_n(\theta)^{\alpha} \psi_{\alpha}^{(0)}(y)}dy  \nonumber \\
+ \bigl(\frac{c_{\alpha}}{2} \ln(\ln n) + K(\theta)\bigr)^2 & \int_{|y|>\Gamma_n} \frac{(\psi_{\alpha}^{(0)}(y))^2}{\phi(y) + c_{\alpha} w_n(\theta)^{\alpha} \psi_{\alpha}^{(0)}(y)}dy\biggr] \biggr|  \nonumber \\
& \leq C \frac{(\ln(\ln n))^2}{\ln(n)^{1/4}} \frac{1}{n^{1-\alpha/2} \ln(n)^{\alpha/2}}, \label{eq:IntGammaL2}
\end{align}
where $K(\theta)= \partial_\alpha c_{\alpha} + c_{\alpha} \ln(\frac{\delta }{\sigma})$.
\end{lem}

\begin{proof}We introduce the approximating functions
\begin{align*}    
\overline{f}_{\alpha}(y, w_n(\theta)) =&\; \phi(y) + c_{\alpha} w_n(\theta)^{\alpha} \psi_{\alpha}^{(0)}(y),\\
\overline{h}_{\alpha}(y, w_n(\theta))=& - \alpha c_{\alpha} w_n(\theta)^{\alpha} \psi_{\alpha}^{(0)}(y),\\
\overline{l}_{\alpha}(y, w_n(\theta)) =& -  c_{\alpha} w_n(\theta)^{\alpha} \psi_{\alpha}^{(1)}(y) 
+ \frac{c_{\alpha}}{2} \ln(\ln n) w_n(\theta)^{\alpha} \psi_{\alpha}^{(0)}(y) 
 +  w_n(\theta)^{\alpha} K( \theta) \psi_{\alpha}^{(0)}(y).    
\end{align*}
We detail the proof of \eqref{eq:IntGammaH2}, and the other two unfold similarly. We decompose the integral to get
\begin{align*}
    \bigg| \int \frac{h_{\alpha}^2}{f_{\alpha}}(y, &w_n(\theta))dy 
    - \int_{|y|>\Gamma_n} \frac{\overline{h}_{\alpha}^2}{\overline{f}_{\alpha}}(y, w_n(\theta))dy \bigg| \\
    &\leq C \bigg(\int_{|y|<\Gamma_n}\frac{h_{\alpha}^2}{f_{\alpha}}(y, w_n(\theta))dy 
    + \int_{|y|>\Gamma_n} \frac{( h_{\alpha} - \overline{h}_{\alpha})^2}{f_{\alpha}}(y, w_n(\theta))dy\\
    &\quad + \int_{|y|>\Gamma_n} \frac{| h_{\alpha} - \overline{h}_{\alpha}| \; |\overline{h}_{\alpha}|}{f_{\alpha}}(y, w_n(\theta))dy
    + \int_{|y|>\Gamma_n} \bigg|\frac{\overline{h}_{\alpha}^2}{\overline{f}_{\alpha}} - \frac{\overline{h}_{\alpha}^2}{f_{\alpha}} \bigg|(y, w_n(\theta))dy \bigg).
\end{align*}
First, we have using \eqref{eq:majfghk} that
\begin{align*}
\int_{|y|<\Gamma_n}\frac{h_{\alpha}^2}{f_{\alpha}}(y, w_n(\theta))dy  
& \leq C w_n(\theta)^{2\alpha} \int_{|y|<\Gamma_n} \frac{(\psi_{\alpha}^{(0)}(y))^2}{\phi(y)} dy  \leq C \frac{\Gamma_n}{n^{2-\alpha} \phi(\Gamma_n)} \leq C \frac{\sqrt{\ln(n)}}{n^{3(1-\alpha/2)/2}}.
\end{align*}
Then, using \eqref{eq:majfghk}, Lemma \ref{L:fonctionsApprox} and recalling that $\int_{|y|>z} \psi_{\alpha}^{(0)}(y)dy = \frac{2}{\alpha z^{\alpha}}$
\begin{align*}
\int_{|y|>\Gamma_n} \frac{( h_{\alpha} - \overline{h}_{\alpha})^2}{f_{\alpha}}&(y, w_n(\theta))dy\\
\leq C &\Bigl(\frac{1}{\sqrt{\Gamma_n}} + \frac{1}{n^{1-\alpha/2}}\Bigr)^2 w_n(\theta)^{2\alpha} \int_{|y|>\Gamma_n}  \frac{(\psi_{\alpha}^{(0)}(y))^2}{\phi(y) + w_n(\theta)^{\alpha} \psi_{\alpha}^{(0)}(y)}dy\\
\leq C &\Bigl(\frac{1}{\sqrt{\Gamma_n}} + \frac{1}{n^{1-\alpha/2}}\Bigr)^2 w_n(\theta)^{\alpha} \int_{|y|>\Gamma_n} \psi_{\alpha}^{(0)}(y)dy\\
\leq C &\frac{1}{\ln(n)^{1/2}} \frac{1}{n^{1-\alpha/2} \ln(n)^{\alpha/2}}.
\end{align*}
Similarly, we use Cauchy–Schwarz inequality and \eqref{eq:majfghk} to obtain
\begin{align*}
\int_{|y|>\Gamma_n} &\frac{| h_{\alpha} - \overline{h}_{\alpha}| \; \overline{h}_{\alpha}}{f_{\alpha}}(y, w_n(\theta))dy \\
&\leq C \left(\frac{1}{\ln(n)^{1/2}}  w_n(\theta)^{\alpha} \int_{|y|>\Gamma_n} \psi_{\alpha}^{(0)}(y)dy \right)^{1/2} \left(w_n(\theta)^{\alpha} \int_{|y|>\Gamma_n} \psi_{\alpha}^{(0)}(y)dy\right)^{1/2}\\
&\leq C \frac{1}{\ln(n)^{1/4}} \frac{1}{n^{1-\alpha/2} \ln(n)^{\alpha/2}}.
\end{align*}
Finally, we have using Lemma \ref{L:fonctionsApprox}
\begin{align*}
\int_{|y|>\Gamma_n} \Bigl|\frac{\overline{h}_{\alpha}^2}{\overline{f}_{\alpha}} - \frac{\overline{h}_{\alpha}^2}{f_{\alpha}}\Bigr|&(y, w_n(\theta))dy = \int_{|y|>\Gamma_n} \Bigl|\frac{\overline{h}_{\alpha}^2}{\overline{f}_{\alpha} f_{\alpha}} (f_{\alpha} - \overline{f}_{\alpha})\Bigr|(y, w_n(\theta))dy\\
\leq C &\Bigl(\frac{1}{\sqrt{\Gamma_n}} + \frac{1}{n^{1-\alpha/2}}\Bigr) \int_{|y|>\Gamma_n} \frac{\overline{h}_{\alpha}^2}{\overline{f}_{\alpha}}(y, w_n(\theta))dy\\
\leq C &\frac{1}{\ln(n)^{1/4}} w_n(\theta)^{\alpha} \int_{|y|>\Gamma_n} \psi_{\alpha}^{(0)}(y)dy\leq C \frac{1}{\ln(n)^{1/4}} \frac{1}{n^{1-\alpha/2} \ln(n)^{\alpha/2}}.
\end{align*}
Combining these results allows us to conclude.
\end{proof}

\begin{lem}\label{l:ramenerInt}
With $\Gamma_n=\sqrt{(1-\alpha/2) \ln(n)}$ and $M_n=\sqrt{(2- \alpha)\ln(n)}$, we have
\begin{align}
      \bigg|\int_{|y|>\Gamma_n} \frac{\psi_{\alpha}^{(p)} (y) \; \psi_{\alpha}^{(0)} (y)}{\phi(y) + c_{\alpha} w_n(\theta)^{\alpha} \psi_{\alpha}^{(0)}(y)}dy - & \frac{\Psi^{(p)}(M_n)}{ c_{\alpha} w_n(\theta)^{\alpha}}   \bigg| 
      \leq C \frac{n^{1-\alpha/2} }{\ln(n)^{\alpha/2}}  \frac{(\ln(\ln n))^{p+1}}{\ln(n)^{1/2\wedge\alpha}},
\end{align}
with $\Psi^{(p)}(z)=\int_{|y|>z} \psi_{\alpha}^{(p)}(y) dy$.
\end{lem}
\begin{proof}
The proof is similar to section A.6 of the Supplement to Aït-Sahalia and Jacod \cite{AitSahalia2012}.
We start by proving that
\begin{align}
      \bigg|\int_{|y|>\Gamma_n} \frac{\psi_{\alpha}^{(p)} (y) \; \psi_{\alpha}^{(0)} (y)}{\phi(y) + c_{\alpha} w_n(\theta)^{\alpha} \psi_{\alpha}^{(0)}(y)}dy   & - \frac{\Psi^{(p)} (y^*(w_n(\theta)))}{ c_{\alpha} w_n(\theta)^{\alpha}}   \bigg| 
      \leq C \frac{\Psi^{(p)}  (y^*(w_n(\theta)))}{ w_n(\theta)^{\alpha}}  \frac{\ln(\ln n)}{\ln(n)^{1/2\wedge\alpha}},\label{eq:etape1}
\end{align}
where $y^*(w)$ is defined by $\phi(y^*(w))= w^{\alpha} \psi_{\alpha}^{(0)}(y^*(w))$. We note that $y^*(w) \underset{w\to0}{\sim} \sqrt{2 \alpha \ln(1/w) }$.

We first lower bound the difference. We set $\mu_n = 1/\ln(n) $.
We check that for $n$ large, $\Gamma_n < y^*(\mu_n w_n(\theta))$ as $y^*(\mu_n w_n(\theta))\underset{n\to\infty}{\sim} \sqrt{(2-\alpha)\ln(n)}=M_n= \sqrt{2} \Gamma_n$. As $\phi/\psi_{\alpha}^{(0)}$ is decreasing after a certain rank,
$$\forall |y|>y^*(\mu_n w_n(\theta)), \quad \phi(y)< (\mu_n w_n(\theta))^{\alpha} \psi_{\alpha}^{(0)}(y),$$
so we have for $n$ large
\begin{align*}
    \int_{|y|>\Gamma_n} \frac{\psi_{\alpha}^{(p)} (y) \; \psi_{\alpha}^{(0)} (y)}{\phi(y) + c_{\alpha} w_n(\theta)^{\alpha} \psi_{\alpha}^{(0)}(y)}dy &\geq  
    \frac{1}{w_n(\theta)^{\alpha}}\int_{|y|>y^*(\mu_n w_n(\theta))} \frac{\psi_{\alpha}^{(p)} (y) \; \psi_{\alpha}^{(0)} (y)}{\psi_{\alpha}^{(0)}(y)}dy \; \frac{1}{c_{\alpha} + \mu_n^{\alpha}} \\
    &\geq \frac{1}{c_{\alpha}w_n(\theta)^{\alpha}} \Psi^{(p)}(y^*(\mu_n w_n(\theta))) \frac{1}{1 + \mu_n^{\alpha}/c_{\alpha}}.
\end{align*}
Hence,
\begin{align*}
\int_{|y|>\Gamma_n} &\frac{\psi_{\alpha}^{(p)} (y) \; \psi_{\alpha}^{(0)} (y)}{\phi(y) + c_{\alpha} w_n(\theta)^{\alpha} \psi_{\alpha}^{(0)}(y)}dy  - \frac{\Psi^{(p)}(y^*( w_n(\theta)))}{ c_{\alpha} w_n(\theta)^{\alpha}} \\
    &\quad \geq C \frac{\Psi^{(p)}(y^*( w_n(\theta)))}{ w_n(\theta)^{\alpha}}
    \biggl(\frac{1}{1 + \mu_n^{\alpha}/c_{\alpha}} \Bigl(\frac{\Psi^{(p)}(y^*( \mu_nw_n(\theta)))}{\Psi^{(p)}(y^*( w_n(\theta)))} - 1\Bigr)+ \Bigl(\frac{1}{1 + \mu_n^{\alpha}/c_{\alpha}} - 1\Bigr) \biggr).
\end{align*}
Using that $\Psi^{(p)}(y) \underset{y\to\infty}{\sim} C\frac{(\ln y)^p}{y^{\alpha}}$ and $\frac{y^*( w_n(\theta))}{y^*( \mu_n w_n(\theta))} \sim \frac{1}{\sqrt{1+|\frac{\ln(\mu_n)}{\ln(w_n(\theta))}|}}$, we conclude that 
\begin{align}
\int_{|y|>\Gamma_n} \frac{\psi_{\alpha}^{(p)} (y) \; \psi_{\alpha}^{(0)} (y)}{\phi(y) + c_{\alpha} w_n(\theta)^{\alpha} \psi_{\alpha}^{(0)}(y)}dy  & - \frac{\Psi^{(p)}(y^*( w_n(\theta)))}{ c_{\alpha} w_n(\theta)^{\alpha}} \nonumber \\
& \geq -C \frac{\Psi^{(p)}(y^*( w_n(\theta)))}{ w_n(\theta)^{\alpha}}  \left(\bigl|\frac{\ln(\mu_n)}{\ln(w_n(\theta))}\bigr| + \mu_n^{\alpha} \right) \nonumber \\
& \geq -C  \frac{\Psi^{(p)}(y^*( w_n(\theta)))}{ w_n(\theta)^{\alpha}}  \frac{\ln(\ln n)}{\ln(n)^{\alpha \wedge 1}}. \label{eq:lower}
\end{align}

We now prove an upper bound. We introduce the sequence $(\gamma_n)$ such that $\gamma_n \to 1$ with a rate which we later specify. We use that 
$$ \forall  |y| <\gamma_n y^*( w_n(\theta)), \quad \phi(y)>\phi(\gamma_n y^*( w_n(\theta)))=\phi(y^*( w_n(\theta)))^{\gamma_n^2} > C\Bigl(\frac{w_n(\theta)^{\alpha}}{y^*( w_n(\theta))^{\alpha+1}}\Bigr)^{\gamma_n^2},$$
and that $\int \psi_{\alpha}^{(p)} \; \psi_{\alpha}^{(0)}< \infty$ to obtain
\begin{align*}
    \int_{|y|>\Gamma_n} & \frac{\psi_{\alpha}^{(p)} (y) \; \psi_{\alpha}^{(0)} (y)}{\phi(y) + c_{\alpha} w_n(\theta)^{\alpha} \psi_{\alpha}^{(0)}(y)} dy\\
    & \leq \int_{|y|< \gamma_n y^*( w_n(\theta))} \frac{\psi_{\alpha}^{(p)} (y) \; \psi_{\alpha}^{(0)} (y)}{\phi(y)}dy 
    + \frac{1}{c_{\alpha} w_n(\theta)^{\alpha}} \int_{|y|> \gamma_n y^*( w_n(\theta))} \frac{\psi_{\alpha}^{(p)} (y) \; \psi_{\alpha}^{(0)} (y)}{\psi_{\alpha}^{(0)}(y)}dy \\
    & \leq C \frac{y^*( w_n(\theta))^{(\alpha+1)\gamma_n^2} }{w_n(\theta)^{\alpha \gamma_n^2}}
    +  \frac{1}{ c_{\alpha} w_n(\theta)^{\alpha}} \Psi^{(p)}(\gamma_n y^*( w_n(\theta))) .
\end{align*}
Recalling that $y^*( w_n(\theta)) \underset{n\to\infty}{\sim} \sqrt{(2- \alpha)\ln(n)}$, and that $\Psi^{(p)}(y) \underset{y\to\infty}{\sim} C\frac{(\ln y)^p}{y^{\alpha}}$, we have 
$$\frac{\Psi^{(p)}(\gamma_n y^*( w_n(\theta)))}{\Psi^{(p)}( y^*( w_n(\theta)))} \underset{n\to\infty}{\sim} 1/\gamma_n^{\alpha}.$$
We obtain
\begin{align*}
    \int_{|y|>\Gamma_n} \frac{\psi_{\alpha}^{(p)} (y) \; \psi_{\alpha}^{(0)} (y)}{\phi(y) + c_{\alpha} w_n(\theta)^{\alpha} \psi_{\alpha}^{(0)}(y)}dy  &- \frac{1}{ c_{\alpha} w_n(\theta)^{\alpha}} \Psi^{(p)}(y^*( w_n(\theta))) \\
    \leq C \frac{ \Psi^{(p)}(y^*( w_n(\theta)))}{ w_n(\theta)^{\alpha}} &\left( \frac{\ln(n)^{(\alpha +1)\gamma_n^2/2 + \alpha/2}}{n^{(1 - \alpha/2)(1-\gamma_n^2)} (\ln(\ln n))^p} + \Bigl( \frac{1}{\gamma_n^{\alpha}} - 1\Bigr) \right).
\end{align*}
We choose $\gamma_n=1- \ln(n)^{-1/2}$, then
$$\frac{1}{n^{(1 - \alpha/2)(1-\gamma_n^2)}} = \frac{1}{n^{(1 - \alpha/2)(\frac{2}{\sqrt{\ln(n)}}-\frac{1}{\ln(n)})}} =\frac{e^{1 - \alpha/2}}{e^{(2 - \alpha)\ln(n)^{1/2}}},$$
so $\frac{\ln(n)^{(\alpha +1)\gamma_n^2/2 + \alpha/2}}{n^{(1 - \alpha/2)(1-\gamma_n^2)} (\ln(\ln n))^p}$ is negligible compared to $\frac{1}{\gamma_n^{\alpha}} - 1$. We conclude that
\begin{align}
    \int_{|y|>\Gamma_n} &\frac{\psi_{\alpha}^{(p)} (y) \; \psi_{\alpha}^{(0)} (y)}{\phi(y) + c_{\alpha} w_n(\theta)^{\alpha} \psi_{\alpha}^{(0)}(y)}dy  - \frac{\Psi^{(p)}(y^*( w_n(\theta)))}{ c_{\alpha} w_n(\theta)^{\alpha}} \leq C \frac{\Psi^{(p)}(y^*( w_n(\theta)))}{ w_n(\theta)^{\alpha}}  \frac{1}{\ln(n)^{1/2}}. \label{eq:upper}
\end{align}
Combining  \eqref{eq:lower} and \eqref{eq:upper} we conclude \eqref{eq:etape1}.

Finally, we bring $y^*( w_n(\theta))$ to $M_n=\sqrt{(2- \alpha)\ln(n)}$. Recalling that $y^*( w_n(\theta))$ is defined by $\phi(y^*( w_n(\theta)))= w_n(\theta)^{\alpha} \psi_{\alpha}^{(0)}(y^*( w_n(\theta)))$, it yields $y^*( w_n(\theta))= M_n + O(\ln(\ln n))$. This allows us to write
\begin{equation}\label{eq:MaM}
|\Psi^{(p)}(y^*( w_n(\theta))) - \Psi^{(p)}( M_n)|\leq C \Psi^{(p)}( M_n) \frac{\ln(\ln n)}{\sqrt{\ln(n)}}.    
\end{equation}
We conclude combining \eqref{eq:etape1}, \eqref{eq:MaM} and that $\Psi^{(p)}( M_n) \underset{n\to\infty}{\sim} C\frac{(\ln(\ln n))^p}{\ln(n)^{\theta/2}}$.
\end{proof}

\subsubsection{Proof of Lemma \ref{L:MajFunc}}
Let $K\subset (0,2)$ be a compact set. From Section 14 in Sato \cite{Sato} we have
$$\mathcal{D}^{(l)}(\partial^m_{\alpha^m} \varphi_{\alpha})(z) \underset{z\to\infty}{\sim} \frac{C(\ln|z|)^m}{|z|^{\alpha+1}},$$
and we also can check  the uniform bound
$$\exists C>0, \; \forall \alpha \in K, \; \forall z\neq0, \quad |\mathcal{D}^{(l)}(\partial^m_{\alpha^m} \varphi_{\alpha})(z)|\leq C \psi_{\alpha}^{(m)}(z).$$
We follow the proof of Lemma \ref{l:approxD} for $f_{\alpha}^{(k, l, m)}$ and write the decomposition
    \begin{align*}
       f_{\alpha}^{(k, l, m)}(y, w) &=
       \mathcal{D}^{(k)}(\phi)(y) \int \mathcal{D}^{(l)}(\partial^m_{\alpha^m} \varphi_{\alpha})(z) dz\\
       &+\int_{|z|\leq1/w} \bigl(\mathcal{D}^{(k)}(\phi)(y-zw)-\mathcal{D}^{(k)}(\phi)(y)\bigr) \mathcal{D}^{(l)}(\partial^m_{\alpha^m} \varphi_{\alpha})(z) dz \\
       &+ \int_{|z|>1/w} \bigl(\mathcal{D}^{(k)}(\phi)(\phi)(y-zw)-\mathcal{D}^{(k)}(\phi)(y)\bigr) \mathcal{D}^{(l)}(\partial^m_{\alpha^m} \varphi_{\alpha})(z) dz .
    \end{align*}
As both $\mathcal{D}$ and $\partial_{\alpha}$ conserve the parity $\forall l, m \in \N$, $z\to \mathcal{D}^{(l)}(\partial^m_{\alpha^m} \varphi_{\alpha})(z)$ is an even function, and we obtain similarly to \eqref{eq:g1} that
$$\Bigl|\int_{|z|\leq1/w} \bigl(\mathcal{D}^{(k)}(\phi)(y-zw)-\mathcal{D}^{(k)}(\phi)(y) \bigr) \mathcal{D}^{(l)}(\partial^m_{\alpha^m} \varphi_{\alpha})(z) dz \Bigr|\leq C w^\alpha \ln(1/w)^m \phi(\frac{y}{2}),$$
and similarly to \eqref{eq:g2} and \eqref{eq:g3} we obtain
\begin{align*}
    \Bigl|\int_{|z|>1/w} \mathcal{D}^{(l)}(\partial^m_{\alpha^m} \varphi_{\alpha})(z) dz\Bigr| & \leq C w^\alpha \ln(1/w)^m,\\
    \Bigl|\int_{|z|>1/w} \mathcal{D}^{(k)}(\phi)(\phi)(y-zw) \mathcal{D}^{(l)}(\partial^m_{\alpha^m} \varphi_{\alpha})(z) dz\Bigr| &\leq  C w^\alpha \ln(1/w)^m \psi_{\alpha}^{(m)}(y).\\
\end{align*}
Hence, using that $\forall y , \;|\mathcal{D}^{(k)}(\phi)(y)| \leq C \phi(\frac{y}{2}) \leq C  \psi_{\alpha}^{(m)}(y)$, we get that $\exists C>0, \; \forall \alpha \in K, \; \forall w \in (0, 1/3], \; \forall y$
\begin{equation}\label{eq:estimfklm}
    \Bigl|f_{\alpha}^{(k, l, m)}(y, w) - \mathcal{D}^{(k)}(\phi)(y) \int \mathcal{D}^{(l)}(\partial^m_{\alpha^m} \varphi_{\alpha})(z) dz \Bigr| \leq C  \ln(1/w)^m w^\alpha \psi_{\alpha}^{(m)}(y).
\end{equation}
We conclude \eqref{eq:majfklm} by noting that
\begin{align*}
    &\text{if } m>0, \quad \int \mathcal{D}^{(l)}(\partial^m_{\alpha^m} \varphi_{\alpha})(z) dz = \partial^m_{\alpha^m}\int \mathcal{D}^{(l)}( \varphi_{\alpha})(z) dz =0,\\
    &\text{if } l>0, \quad \int \mathcal{D}^{(l)}(\partial^m_{\alpha^m} \varphi_{\alpha})(z) dz = \bigl[x\mathcal{D}^{(l-1)}(\partial^m_{\alpha^m} \varphi_{\alpha})(x)\bigr]_{-\infty}^{+\infty}=0,
\end{align*}
where we used that
$$\mathcal{D}^{(l-1)}(\partial^m_{\alpha^m} \varphi_{\alpha})(z) \underset{z\to\infty}{\sim} \frac{C(\ln|z|)^m}{|z|^{\alpha+1}}.$$
Turning to $f_{\alpha}^{(k, 0, 0)}$, we note that $\mathcal{D}^{(k)}(\phi)(y)=P_{2k}(y)\phi(y)$ with $P_{2k}$ a polynomial of order $2k$. 
From \eqref{eq:estimfklm}, we hence have $\exists C>0, \; \forall \alpha \in K, \; \forall w \in (0, 1/3], \; \forall y$
$$    |f_{\alpha}^{(k, 0, 0)}(y, w)| \leq C \big((1+|y|^{2k})\phi(y) + \ln(1/w)^m w^\alpha \psi_{\alpha}^{(m)}(y)\big).$$
We refine this bound exactly as done for the bound of $g_\alpha$ in \eqref{eq:majfghk}, detailed in \eqref{eq:refineboundg}, to obtain \eqref{eq:majfk00}.


\subsubsection{Proof of Lemma \ref{L:der2der3}}
We start by noting that $\forall k,l,m \in \N$
\begin{align*}
    \partial_\sigma \Bigl(\frac{\sqrt{n}}{\sigma} f_{\alpha}^{(k, l,m)}(y, w(\theta)) \Bigr) &=- \frac{\sqrt{n}}{\sigma^2}f_{\alpha}^{(k+1, l,m)}(y, w(\theta)), \\
    \partial_\delta \Bigl(\frac{\sqrt{n}}{\sigma} f_{\alpha}^{(k, l,m)}(y, w(\theta))\Bigr)&=- \frac{\sqrt{n}}{\sigma \delta}  f_{\alpha}^{(k, l+1,m)}(y, w(\theta)),\\
    \partial_\alpha \Bigl(\frac{\sqrt{n}}{\sigma} f_{\alpha}^{(k, l,m)}(y, w(\theta)) \Bigr)&= \frac{\sqrt{n}}{\sigma}\left[f_{\alpha}^{(k, l,m+1)}(y, w(\theta)) - \frac{\ln(n)}{\alpha^2} f_{\alpha}^{(k, l+1,m)}(y, w(\theta)) \right].
\end{align*}
Moreover, for $a,b,c \in \{\sigma, \delta, \alpha \}$, we have
\begin{align*}
\partial^2_{ a b} \ln p_{1/n}=&\frac{\partial^2_{ a b} p_{1/n} }{p_{1/n}} - \frac{\partial_{ a} p_{1/n} \partial_{ b} p_{1/n} }{(p_{1/n})^2}, \\
\partial^3_{ a b c} \ln p_{1/n} =& 
\frac{\partial^3_{ a b c}p_{ 1/n}}{p_{ 1/n}} - \frac{\partial^2_{ a b}p_{ 1/n} \partial_{ c}p_{1/n}
+ \partial^2_{ a c}p_{1/n} \partial_{ b}p_{1/n}
+ \partial^2_{b c}p_{1/n} \partial_{a}p_{1/n}}{(p_{1/n})^2}
+2\frac{\partial_{ a} p_{1/n} \partial_{ b} p_{1/n} \partial_{c}p_{1/n}}{(p_{1/n})^3}.
\end{align*}  
Recalling that
$$p_{1/n}(y, \theta) = \frac{\sqrt{n}}{\sigma}f_{\alpha}\Bigl(\frac{\sqrt{n}}{\sigma}y, w_n(\theta)\Bigr)= \frac{\sqrt{n}}{\sigma}f_{\alpha}^{(0, 0,0)}\Bigl(\frac{\sqrt{n}}{\sigma}y, w_n(\theta)\Bigr),$$
we have
$$\partial^2_{\sigma \sigma} \ln (p_{1/n}(y, \theta)) = \biggl[ \frac{2}{\sigma^2} \frac{f_{\alpha}^{(2, 0,0)}}{f_{\alpha}} +
\frac{1}{\sigma^2} \frac{f_{\alpha}^{(1, 0,0)}}{f_{\alpha}} 
- \frac{1}{\sigma^2} \Bigl(\frac{f_{\alpha}^{(1, 0,0)}}{f_{\alpha}}\Bigr)^2 \biggr] \Bigl(\frac{\sqrt{n}}{\sigma}y, w_n(\theta)\Bigr),$$
and we conclude using \eqref{eq:majfghk}, Lemma \ref{L:MajFunc} and that $\ln(1/w_n(\theta)) \leq C \ln(n)$ that
$$\exists C>0, \;\exists p>0, \; \forall y, \qquad
    |\partial^2_{\sigma \sigma} \ln p_{1/n}(y, \theta_0)| \leq C \ln(n)^p.$$

Similarly, for $a,b \in \{\sigma, \delta, \alpha \}$ and $(a,b)\neq(\sigma, \sigma)$, $\partial^2_{a b} \ln (p_{1/n}(y, \theta))$ can be written as a sum and product of 
$$c_{k,l,m}(\theta) \ln(n)^p \,\frac{f_{\alpha}^{(k, l,m)}}{f_{\alpha}}\Bigl(\frac{\sqrt{n}}{\sigma}y, w_n(\theta)\Bigr),$$ for $l+m>0$, $l,m \leq 2$ and $c_{k,l,m}(\theta)$ coefficients which are uniformly bounded on $V$. Hence, we have from \eqref{eq:majfghk} and Lemma \ref{L:MajFunc} that $\exists C >0, \; \exists p,q >0, \; \forall y$
$$|\partial^2_{a b} \ln (p_{1/n}(y, \theta)) |\leq C \ln(n)^p \ln(1/w_n(\theta))^q \frac{w_n(\theta)^{\alpha}\psi_{\alpha}^{(2)}(\frac{\sqrt{n}}{\sigma} y)}{\phi(\frac{\sqrt{n}}{\sigma} y)+w_n(\theta)^{\alpha}\psi_{\alpha}^{(0)}(\frac{\sqrt{n}}{\sigma} y)},$$
and we conclude \eqref{eq:boundder2}. The bounds in \eqref{eq:boundder3} are obtained similarly.

\medskip

\noindent
{\bf Funding.}
This work was supported by french ANR reference ANR-21-CE40-0021.


\end{document}